\documentclass[english,11pt]{smfartVScomptage}

\usepackage{smfenum,smfthm, amsmath,amssymb,amscd,mathrsfs,euscript,color}
\usepackage{stmaryrd,mathabx,upgreek} 
\usepackage[latin1]{inputenc}
\usepackage[T1]{fontenc}
\usepackage{xcolor}
\input{xypic}

\numberwithin{equation}{section} 
\def\AA{\mathfrak{A}}

\def\CC{\mathbb{C}}

\def\TT{\boldsymbol{\Theta}}

\def\A{{\rm A}}
\def\B{{ B}}

\def\E{{ E}}
\def\F{{ F}}
\def\G{{ G}}
\def\H{{ H}}
\def\I{{ I}}
\def\J{{\rm J}}
\def\K{{ K}}
\def\L{{ L}}
\def\M{{ M}}
\def\N{{ N}}
\def\O{{\rm O}}
\def\P{{ P}}
\def\Q{{ Q}}
\def\R{{ R}}

\def\T{{ T}}
\def\U{{\rm U}}
\def\V{{ V}}
\def\W{{\rm W}}

\def\Aa{\boldsymbol{\sf A}}

\def\Cc{\EuScript{C}}

\def\Ii{\mathscr{I}}

\def\Kk{\EuScript{K}}
\def\Ll{\EuScript{L}}

\def\Vv{{\rm V}}

\def\La{\Lambda}

\def\a{\alpha} 
\def\b{\beta}
\def\d{\delta}

\def\l{\lambda}
\def\n{\eta}

\def\p{\mathfrak{p}}
\def\r{{\textbf{\textsf{r}}}}
\def\s{\sigma}
\def\t{\theta}
\def\v{\upsilon}
\def\w{\varpi}
\def\x{\varkappa}

\def\kk{\boldsymbol{k}}
\def\ll{\mathrm{l}}

\def\ss{s}

\def\({\left(}
\def\){\right)}
\def\>{\geqslant}
\def\<{\leqslant}

\def\tdt{\times\cdots\times}

\def\Hom{\operatorname{Hom}}
\def\End{\operatorname{End}}
\def\Aut{\operatorname{Aut}}
\def\Mat{\operatorname{M}}
\def\GL{\operatorname{GL}}
\def\Gal{\operatorname{Gal}}
\def\tr{\operatorname{tr}}

\def\Ind{\operatorname{Ind}}
\def\ind{\operatorname{ind}}

\def\dim{\operatorname{dim}}

\def\diag{\operatorname{diag}}
\def\val{\operatorname{val}}

\def\st{\operatorname{st}}
\def\St{\operatorname{St}}

\def\BJ{{\bf J}}

\def\Ql{{\overline{\mathbb{Q}}_\ell}}
\def\Zl{{\overline{\mathbb{Z}}_\ell}}
\def\Fl{{\overline{\mathbb{F}}_{\ell}}}

\def\qlb{\Ql}
\def\zlb{\Zl}
\def\flb{\Fl}

\def\oo{\EuScript{O}}
\def\pp{\mathfrak{p}}

\def\bl{\boldsymbol{\lambda}}
\def\bd{\boldsymbol{\delta}}
\def\bk{\boldsymbol{\kappa}}
\def\bx{\boldsymbol{\xi}}
\def\tbk{\bk_{\t}}

\def\bt{\boldsymbol{\tau}}
\def\bx{\boldsymbol{\xi}}
\def\det{\mathrm{det}}

\def\rl{\r_\ell}
\def\1{\mathbf{1}}

\def\Nm{{\rm N}}

\def\EE{{\F}}
\def\FF{\F_0}

\def\nf{\nu_{0}}
\def\ke{\boldsymbol{k}}
\def\kf{\boldsymbol{k}_{0}}

\def\et{e}
\def\ll{\boldsymbol{l}}

\def\ZZ{\mathbb{Z}}
\def\BJT{\BJ_{{\rm t}}}
\def\blw{\bl_{{\rm w}}}

\def\btt{\bt_{{\rm t}}}
\def\pit{\pi_{{\rm t}}}

\def\rhot{\rho_{{\rm t}}}

\def\ic{\omega}

\def\aa{\mathfrak{a}}
\def\bb{\mathfrak{b}}

\def\Mat{\boldsymbol{{\sf M}}}

\def\Cc{\EuScript{C}}
\def\NF{\boldsymbol{\pi}}
\def\Vc{\W}

\title[Cuspidal $\flb$-representations of $\GL_n(\F)$ distinguished by a Galois involution]
{Cuspidal $\ell$-modular representations of $\GL_n(\F)$ distinguished by a Galois involution} 
\author{Robert Kurinczuk}
\address{School of Mathematics and Statistics, University of Sheffield, Sheffield, S3 7RH, United Kingdom}
\email{robkurinczuk@gmail.com}
\author{Nadir Matringe}
\address{Institute of Mathematical Sciences, NYU Shanghai, Shanghai, China}
\address{Institut de Math\'ematiques de Jussieu-Paris Rive Gauche, Universit\'e Paris Cit\'e, 75205, Paris, France}
\email{nrm6864@nyu.edu}
\email{matringe@imj-prg.fr}
\author{Vincent S\'echerre} 
\address{Laboratoire de Math\'emati\-ques de Versailles\\
UVSQ\\
CNRS\\
Universit\'e Paris-Saclay\\
78035, Versailles, France}
\email{vincent.secherre@uvsq.fr}

\begin{document}

\begin{abstract}
Let $\F/\F_0$ be a quadratic extension of non-Archimedean locally compact 
fields of~re\-sidual characteristic $p\neq2$ with Galois
automorphism~$\sigma$, and let~$R$ be an algebraically~closed~field~of 
characteristic~$\ell\notin\{0,p\}$.
We reduce the classification of 
$\GL_n(\F_0)$-distinguished cuspidal~$R$-represen\-tations of~$\GL_n(\F)$ to 
the level $0$ setting.
Moreover, under a parity condition, we give necessary~con\-ditions
for a~$\sigma$-selfdual cuspidal $R$-representation to be 
distinguished.
Finally, we classify the~distin\-guished 
cuspidal~$\Fl$-representations of~$\GL_n(\F)$ having a distinguished cuspidal 
lift to~$\Ql$.  
\end{abstract}

% \begin{nouppercase}
\maketitle
% \end{nouppercase}

\tableofcontents

\section{Introduction}

\subsection{}

Let $\F/\F_0$ be a quadratic extension of non-Archimedean locally compact 
fields whose resi\-dual cha\-racteristic is a prime number $p$ different from $2$.
Let $\s$ be its non-trivial automorphism,
and $\G$ be the~gene\-ral linear group $\GL_n(\F)$ for some positive integer 
$n$.
It is a totally discon\-nec\-ted, locally compact group,
on which the involution $\s$ acts componentwise,
and~the group $\G^\s$ of its $\s$-fixed points is equal to $\GL_n(F_0)$.

Now fix an algebraically closed field $\R$ of characteristic different from 
$p$. 
A (smooth) represen\-ta\-tion $\pi$ of $\G$ on an $R$-vector space $V$
is said~to~be~\textit{distinguished} (by $\G^\s$) if $V$ carries a 
non-zero $\G^\s$-in\-variant linear form;
more generally,
if $\chi$ is a smooth character of $\G^\s$ with 
values~in~$\R^\times$,~the re\-pre\-sentation $\pi$ is 
said~to~be~$\chi$-\textit{distin\-guished} 
if $V$ carries a non-zero linear form ${\it\La}$~such that
\begin{equation*}
{\it\La}(\pi(h)v)=\chi(h){\it\La}(v),
\quad h\in\G^\s, \quad v\in\V.
\end{equation*}

\subsection{}

In the case where $R$ is the field of complex numbers,
distinguished irreducible representations~of $G$ have been extensively 
studied:
\begin{enumerate}
\item % [(i)]
they are $\s$-selfdual,
that is,
the contragredient $\pi^\vee$ of a distingui\-shed irreducible represen\-ta\-tion
$\pi$ of $G$ is isomorphic to its $\s$-conjugate $\pi^\s$
(\cite{Flicker,Prasad90,Prasad01})
and their central character~is~tri\-vial~on $\F_0^\times$,
\item % [(ii)]
a $\s$-selfdual discrete series representation of $G$ 
is either distinguished, or $\x$-distin\-gui\-shed ($\x$
denotes the character of $F_0^\times$ whose kernel is the subgroup 
of $F/F_0$-norms), 
but not both:
this is the Dichotomy and Disjunction Theorem
(\cite{Kable,AnandKableTandon,AKMSS}),
\item % [(iii)]
distinguished generic irreducible representations of $G$ are classified
in terms of~their cuspi\-dal support
(\cite{AnandRajan,MatringeIMRN09,Matringe11}),
\item % [(iv)]
distinguished cuspidal representations of $G$ are characterized
in terms of~their~Galois~para\-meter (\cite{GanRag})
and in terms of type theory (see \cite{VSANT19} and below).
\end{enumerate}

\subsection{}

Distinguished irreducible representations of $G$ with coefficients in a field $R$ of
positive~cha\-rac\-teristic have been less well studied
(see \cite{AKMSS,VSANT19,KuMaAsai,CLL}).
As in the complex case,
they~are~$\s$-self\-dual, and their central character~is trivial on 
$\F_0^\times$.
For $\s$-selfdual \textit{supercuspidal} 
representations, that is, irreducible representations
which~do~not~oc\-cur~as sub\-quo\-tients of 
parabolically~indu\-ced representations from a proper Levi subgroup,
one~has a~{Dichotomy and Disjunction~Theorem}~(see \S\ref{montoriol}).
One also has a characterization of distinction
in terms of Galois parameters (\cite{CLL} Proposition 3.15)
and in terms~of types (\cite{VSANT19} 
Theorem 10.9).
But~there are explicit~exam\-ples of~$\s$-selfdual
non-supercuspidal~\textit{cuspidal}~re\-presentations 
that are neither distinguished nor $\x$-distinguished~(as
in \cite{VSANT19} Remark 2.18)
and of Steinberg re\-pre\-sen\-ta\-tions
that   are   both   distinguished   and   $\x$-distin\-guished   (\cite{CLL}
Remark~1.9).
Also, there is no known
classification~of~distinguished cuspidal~repre\-sentations~of $\GL_n(F)$
for an arbitrary $n\>3$ (see \cite{CLL}~for~$n=2$).

In this paper,
which can be considered as a sequel to \cite{VSANT19},
where all distinguished supercuspidal $R$-representations of $G$ have been
classified,
we investigate the classification~of distinguished~cus\-pi\-dal
$R$-representations of $G$ in terms of their supercuspidal support.
We:
\begin{itemize}
\item 
reduce this classification to that of distinguished cuspidal representations 
of level $0$, and~from there to finite group theory (see Section 
\ref{pasfaim}), 
\item
give a necessary condition of distinction for $\s$-selfdual cuspidal 
representations of $G$~that~sa\-tis\-fy a certain parity condition
(see Section \ref{Vsoddcase}), 
\item
classify the (distinguished, cuspidal) $\flb$-representations of $G$ having a 
distinguished cuspidal lift to $\qlb$,
where $\qlb$ is an algebraic closure of the field of $\ell$-adic numbers
with residue field $\flb$. % (in Section \ref{seceven}).
\end{itemize}
Let us explain these results in more detail.

\subsection{}

Bushnell and Kutzko \cite{BK},
in work extended to the modular setting by Vign\'eras \cite{Vigbook},~have
given an explicit construction of a collection
of pairs~$(\BJ,\bl)$ called \emph{extended maximal simple types}
(which we will abbreviate to \textit{types} here), consisting
of a compact-mod-centre open subgroup~$\BJ$~of $G$ and an
irreducible~$R$-repre\-sen\-tation~$\bl$~of $\BJ$, such that the
representations $\ind_{\BJ}^{G}(\bl)$ 
are (irre\-du\-cible and) cuspidal,~and~such
that every cus\-pi\-dal $R$-representation of~$G$ appears in the collection
of~$\ind_{\BJ}^{G}(\bl)$.

We need the following invariants associated to
a cuspidal~$R$-representation of~$G$
fol\-lowing this explicit construction by compact induction
(see \S\ref{nom42} and \S\ref{peche}): 
\begin{enumerate}
\item
the \emph{endo-class} $\TT$: 
a fine refinement of the level 
introduced by Bushnell-Hen\-niart in \cite{BHLTL1}~and which 
applies equally well to the modular setting,
\item
the \emph{tame parameter field}~$\T$:
a tamely ramified
extension of $\F$ of degree dividing~$n$,~uni\-que\-ly determined up
to~$\F$-isomorphism by $\TT$,
\item
the \emph{relative degree} $m$: 
a positive integer 
such that $m[\T:\F]$ divides $n$,
uni\-quely~de\-ter\-mined by $\TT$ and $n$.
\end{enumerate}

Suppose further that~$\TT$ is~$\sigma$-selfdual
(which follows if for example the cuspidal representation itself 
is~$\sigma$-selfdual),
then there is a uniquely determined tamely ramified extension $T_0$
of $F_0$~con\-tained in $T$ such that $\T$ is isomorphic to
$\T_0\otimes_{\F_0} \F$.
The Galois group of $\T/\T_0$ canonically identifies with that of $\F/\F_0$, 
and the unique non-trivial automorphism of $\T/\T_0$ extending $\s$ 
will be denoted by $\s$ (see \S\ref{rappelstypesssd}). 
Our main theorem on reduction to the level $0$ setting is then
(see Theorem \ref{potimarrong}):

\begin{theo} 
\begin{enumerate}
\item
There is a natural bijection:
\begin{equation}
\label{douglas}
\pi\mapsto \pit
\end{equation} 
from the set of isomorphism classes of cuspidal representations 
of~$\G$~with~endo-class $\TT$ to the set of isomorphism classes of 
cus\-pi\-dal re\-presentations of level~$0$ of $\GL_m(\T)$.  
\item
The representation~$\pi$ is $\s$-selfdual if and only if~$\pit$ is~$\s$-selfdual.
\item
The representation~$\pi$ is~$\GL_n(F_0)$-dis\-tinguished if and only
if~$\pit$~is $\GL_m(\T_0)$-distinguished. 
\end{enumerate}
\end{theo}

The map \eqref{douglas} is also compatible with
supercuspidal support, see Proposition \ref{compapitrho} for a precise 
statement. 

\subsection{}

Let us briefly explain how the map \eqref{douglas} above is defined.
Let $(\BJ,\bl)$ be a type inducing~a cuspidal re\-pre\-sentation $\pi$ of $\G$ 
with $\s$-selfdual endo-class $\TT$,
tame parameter field $T$ and relative degree $m$.
Then:
\begin{enumerate}
\item
The group $\BJ$ has a unique maximal compact subgroup $\BJ^0$,
and a unique maximal nor\-mal pro-$p$ subgroup $\BJ^1$.
\item
There is a group isomorphism $\BJ^0/\BJ^1\simeq \GL_m(\ll) $,
where~$\ll$ is the residue field of $\T$.
\item
The restriction of~$\bl$ to~$\BJ^1$ is isotypic for an irreducible
representation~$\n$ of~$\BJ^1$, 
and this~re\-presentation $\n$ extends (non-canonically) to $\BJ$.
\item
The choice of a representation $\bk$ of $\BJ$ extending $\n$
determines a decomposition $\bl\simeq \bk\otimes\bt$,
where $\bt$ is a representation of $\BJ$ trivial on~$\BJ^1$,
uniquely determined up to isomorphism.  
\end{enumerate} 
The fact that $\TT$ is $\s$-selfdual implies that
there is a preferred choice for $(\BJ,\bl)$:
the group $\BJ$~is fixed~by $\s$, 
the representation $\n$ is $\s$-selfdual and
there exists a natural isomorphism between
the space of $G^\s$-invariant linear forms on $\pi$
and that of $\BJ\cap\G^\s$-invariant linear forms on $\bl$.~Such a type~is 
called \textit{generic} (see Definition \ref{genericssdualtype}).
We prove (see Proposition \ref{potiron}):

\begin{prop}
\label{potironintro}
The representation $\n$ 
has~a unique extension~$\bk$ to $\BJ$ which is both 
$\s$-selfdual and $\BJ\cap\G^\s$-distinguished, and
whose determinant has~order a power of~$p$.
\end{prop}

The choice of the representation $\bk$ given by Proposition \ref{potironintro}
thus uniquely determines a repre\-sen\-tation $\bt$ of $\BJ$ trivial on~$\BJ^1$.

Now there is a natural choice, as explained in \S\ref{proofpotimarron1}, 
of a $\s$-fixed~maximal compact subgroup $\BJT^0$
of $\GL_m(T)$,
with normalizer $\BJT^{\phantom{0}}$ and pro-$p$-radical $\BJT^1$,
such that there is a $\s$-equivariant group isomorphism:
\begin{equation*}
\BJ/\BJ^1\simeq\BJT^{\phantom{0}}/\BJT^1.
\end{equation*}
The representation $\bt$ then defines a representation of $\BJT^{\phantom{0}}$
trivial on $\BJT^1$, denoted $\btt$.
The cuspidal representation $\pit$ associated with $\pi$ by \eqref{douglas}
is then the compact induction 
of $\btt$ to $\GL_m(T)$.

\subsection{}

Having reduced the classification of distinguished 
cuspidal~$\R$-representations to level $0$,~we fur\-ther reduce this classification
to the finite group setting.
Let $\pi$ be a $\s$-selfdual cuspidal~$\R$-re\-pre\-sentation~of $\G$
of level~$0$ with central character~$c_{\pi}$ and generic type~$(\BJ,\bl)$.
Restricting~$\bl$~to $\BJ^0$~de\-fi\-nes a~cuspi\-dal~$\R$-re\-presentation $\Vv$ 
of $\GL_n(\kk)$,
where $\kk$~is the residue field of $\F$.
We prove (see~Theo\-rem \ref{piranesi}): 

\begin{theo}
Suppose~$n\neq 1$.
The representation $\pi$ is $\GL_n(\F_0)$-distinguished if 
and only if its central character $c_\pi$ is trivial on $\F_0^\times$ and 
\begin{enumerate}
\item
if $\F/\F_0$ is unramified, 
then $\Vv$ is $\GL_n(\kk_0)$-distinguished
($\kk_0$ is the residue field of $\F_0$);
\item
if $\F/\F_0$ is ramified, 
then $n$ is even,
$\Vv$ is $\GL_{n/2}(\kk)\times\GL_{n/2}(\kk)$-distinguished,
the vector~spa\-ce~of $\GL_{n/2}(\kk)\times\GL_{n/2}(\kk)$-invariant linear forms
on $\Vv$ has dimension $1$, and 
\begin{equation*}
\ss =
\begin{pmatrix}
0 & {\rm id} \\
{\rm id} & 0 \end{pmatrix}
\in \GL_n(\kk)
\end{equation*}
acts on this space by the sign $c_\pi(\w)$,
where $\w$ is any uniformizer of $\F$.
\end{enumerate}
\end{theo}

\subsection{}

Let~$\pi$ be a cuspidal non-supercuspidal~$\R$-representation of $G$. 
Following~\cite{MSc}, we recall~in \S\ref{par92} that there 
are a uniquely determined integer~$r=r(\pi)\geqslant 2$ and a 
supercuspidal~$\R$-re\-presen\-tation~$\rho$ of $\GL_{n/r}(F)$~such
that~$\pi$ is isomorphic to $\St_{r}(\rho)$, 
where~$\St_r(\rho)$ denotes the unique generic subquotient 
of the parabolically induced representation
\begin{equation*}
\rho\nu^{-(r-1)/2}\tdt\rho\nu^{(r-1)/2}
\end{equation*}
(where $\nu$ denote 
the unramified character which is the absolute value of~$\F$ composed with
the~de\-ter\-minant).~The~re\-presentation $\rho$ is not unique in general,
but,
if~$\pi$ is~$\sigma$-selfdual and $r$~is odd,~and
if one further~de\-mands that $\rho$ be $\sigma$-selfdual,
then $\rho$ is uniquely determined up to isomorphism~(see Proposition \ref{fission}). 
In this case, we obtain further necessary 
conditions for~distinction (see Theorem \ref{THModd}): 

\begin{theo}
Let $\pi$ be a $\s$-selfdual cuspidal non-supercuspidal
$\R$-representation of $\GL_n(\F)$.
Assume that the in\-teger $r=r(\pi)$ is odd,
thus $\pi$ is isomorphic to $\St_r(\rho)$ for a uniquely determined
$\s$-selfdual 
supercuspidal~re\-presentation $\rho$~of $\GL_{n/r}(\F)$.
If $\pi$ is $\GL_n(\F_0)$-distinguished, then 
\begin{enumerate}
\item 
the relative degree $m=m(\pi)$ and the ramification index of $\T/\T_0$
have the same parity,
\item 
the representation $\rho$ is $\GL_{n/r}(\F_0)$-dis\-tin\-guished. 
\end{enumerate}
\end{theo}
As a corollary, we extend the Disjunction Theorem from the supercuspidal 
setting (that is,~the statement that,
if~$\ell\neq 2$,
a supercuspidal~$\R$-representation is not both distinguished 
and~$\x$-dis\-tin\-guished) to include cuspidal
$\R$-representations~$\pi$ with~$r(\pi)$ odd. 

\subsection{} 

Say that an irreducible $\flb$-representation $\pi$ of~$G$
\textit{lifts to} $\qlb$ if there exists a free $\zlb$-lattice~$L$ equip\-ped~with a
linear action of $G$ such that the $\flb$-representation of $G$ on 
$L\otimes\flb$ is isomor\-phic~to $\pi$.
When this is the case,
say that the smooth $\qlb$-representation of $G$ on $L\otimes\qlb$~is
a \textit{lift} of $\pi$ to $\qlb$.
Following \cite{Vigbook},
any cuspidal $\flb$-representation of $G$ lifts to $\qlb$
and any of its lifts is cuspidal.

According to \cite{KuMaAsai} (see Theorem \ref{KuMaTh}),
any cuspidal~$\flb$-representation of~$G$ having a $G^\s$-dis\-tin\-guished lift
to $\qlb$ is $G^\s$-distinguished.
The converse holds for supercuspidal representations
(see \cite{VSANT19} and \cite{CLL}),
but not for cuspidal representations in general.
In the~fi\-nal section, 
we classify~the $G^\s$-distinguished cuspidal~$\Fl$-representations of~$G$
having a $G^\s$-distinguished cuspidal lift to~$\Ql$ 
(see Propositions \ref{poivronsimpairs} and \ref{poivronspairs} for a precise 
statement). 

\begin{center}
{\bf Structure of the paper}
\end{center}

\medskip

After setting some notation in Section \ref{Notationsection},
in Section \ref{soupe} we collect together necessary~back\-ground from the literature and
prove some basic results on~$\sigma$-selfdual
cuspidal~$\R$-representations.

Section \ref{pasfaim} constitutes the technical heart of the paper.
It reduces 
the problem of classifying~dis\-tingui\-shed~cuspidal~$\R$-representations to 
level $0$.

In Section \ref{Vsoddcase},~un\-der a~parity con\-dition,
we~provide~ne\-ces\-sary conditions for 
distinction,
allowing us to deduce the Disjunction Theorem and~a lifting
theorem.

Finally, in Section \ref{seceven}, we
classify~those cus\-pidal~$\Fl$-representations having a 
distinguished~cus\-pidal lift. 

\medskip

\begin{center}
{\bf Acknowledgements}
\end{center}

\medskip

The first author was supported by EPSRC grant EP/V001930/1 and the Heilbronn 
Institute for Mathematical Research.
The third author was partially supported by the Institut Universitai\-re de 
France. 
    
This work was partially supported  by the Erwin Schr\"odinger Institute in
Vienna, where we~be\-ne\-fitted from the Research in Teams grant
``$\ell$-modular Langlands  Quotient Theorem and Applica\-tions''.  We thank
the institute for hospitality, and for excellent working conditions. 
    
We thank Alberto M\'inguez and Shaun Stevens for their interest and useful 
conversations. 

\section{Notation}
\label{Notationsection}
\subsection{}
\label{par21}

Given any non-archimedean locally compact field $\F$,
we write $\oo_{\F}$ for its ring of integers,~$\pp_\F$ for the maximal ideal 
of $\oo_{\F}$, 
$\kk_\F$ for its residue field and
$q_\F$ for the cardinality of $\kk_\F$.

We also write $\val_\F$ for the valuation of $\F$
taking any uniformizer to $1$,
and 
$|\cdot|_\F$ for the~abso\-lute~value of $\F$ taking
any uniformizer to the inverse of $q_\F$. 

Given any finite extension $\L$ of $\K$, 
we write~$\Nm_{\L/\K}$ and~$\tr_{\L/\K}$ for the norm and trace maps. 

\subsection{}

Given
a locally compact, totally disconnected topological group $\G$ and
an algebraically~clo\-sed field $\R$ of characteristic different from $p$,
we consider smooth representations of $\G$ on $\R$-vector spaces. 
We will abbreviate \textit{smooth $\R$-representation} to 
\textit{$\R$-representation},
or even \textit{representation}~if the coefficient field $\R$ is clear from the context.

An $\R$-\textit{character} 
(or \textit{character})
of $\G$ is a group homomorphism from $\G$ to $\R^\times$ with open kernel.

Let $\pi$ be a representation of $\G$.
We write $\pi^\vee$ for its~con\-tra\-gredient. 
Given a character $\chi$ of $\G$,
we write $\pi\chi$ for the representation $g\mapsto\chi(g)\pi(g)$ of~$\G$.

Let $\pi$ be a representation of a closed subgroup $\H$ of $\G$. 
Given any element $g\in\G$,
we write~$\pi^g$ for the representation $x\mapsto\pi(gxg^{-1})$ of
$\H^g=g^{-1}\H g$.
Given any continuous involution $\s$ of~$\G$, we
write $\pi^\s$ for the representation $\pi\circ\s$ of $\s(\H)$.
Given any cha\-rac\-ter~$\mu$ of $\H\cap\G^\s$,
we say that $\pi$ is $\mu$-\textit{distinguished} if the space
$\Hom_{\H\cap\G^\s}(\pi,\chi)$ is non-zero.
If $\mu$ is the trivial character,
we will abbreviate $\mu$-\textit{distinguished} to
$\H\cap\G^\s$-\textit{distinguished},
or just \textit{distinguished}.

\subsection{}

Let us fix a separable quadratic extension $\EE/\FF$ of non-archimedean 
locally compact fields of residual cha\-rac\-teristic~$p$,
and let $\s$ denote its non-trivial automorphism.
Let
\begin{equation}
\x = \x_{\EE/\FF} : \FF^\times \rightarrow \{-1,1\} = \ZZ^\times
\end{equation}
denote the $\ZZ$-valued character of $\FF^\times$
with kernel $\Nm_{\EE/\FF}(\EE^\times)$. 
When needed, 
we will consider~$\x$~as a character with values in any algebraically closed
field $\R$.  
We abbreviate $q=q_\EE$ and $q_0=q_{\FF}$.

We fix a square root
\begin{equation}
\label{choixq012}
q_0^{1/2}\in \R
\end{equation}
of~$q_0$ in $\R$ an define
\begin{equation}
\label{choixq12}
q^{1/2} =
\left\{ 
\begin{array}{ll}
q_0^{1/2} & \text{if $\F/\F_0$ is ramified}, \\ 
q_0 & \text{if $\F/\F_0$ is unramified},
\end{array}\right.
\end{equation}
which we will use to normalize parabolic induction and restriction functors
(see below).

\subsection{}

Given a positive integer $n\>1$, 
the automorphism $\s$ acts~on~the group $\GL_n(F)$~com\-ponent\-wise,
thus defines a continuous involution of $\GL_n(F)$, 
still denoted $\s$.
Its fixed points form the subgroup $\GL_n(F_0)$.

We denote by $\nu$ the unramified character ``absolute value of the 
determinant'' of $\GL_n(F)$ and by $\nu^{1/2}$ the unramified character taking 
any element whose determinant has valuation $1$ to $q^{-1/2}$. 
We thus have $(\nu^{1/2})^2=\nu$.
Similarly, we define the characters $\nu^{\phantom{1}}_0$ and 
$\nu_0^{1/2}$ of $\GL_n(F_0)$.

Given positive integers $n_1,\dots,n_r$ such that $n_1+\dots+n_r=n$ 
and, for each $i=1,\ldots,r$, given 
an $\R$-representation $\pi_i$ of $\GL_{n_i}(F)$, 
we write
\begin{equation}
\label{defsupercusp}
\pi_1\times \cdots \times \pi_r
\end{equation}
for the representation of
$\GL_{n}(F)$ obtained by normalized parabolic induction from 
$\pi_1\otimes \cdots \otimes \pi_r$
along the parabolic subgroup generated by upper triangular matrices 
and the standard Levi~sub\-group $\GL_{n_1}(F)\times\dots\times\GL_{n_r}(F)$.

An irreducible $R$-representation of $\GL_n(F)$ is said to be
\textit{cuspidal} (respectively, \textit{supercuspidal})~if
it does not occur~as~a subrepresentation
(respectively, a subquotient)
of any representation of~the form \eqref{defsupercusp} with $r\>2$.
Any supercuspidal representation of $\GL_n(F)$ is cuspidal.
When $R$ has~cha\-racteristic $0$,
any cuspidal representation of $\GL_n(F)$ is supercuspidal.
When $R$~has~cha\-racteristic $\ell>0$,
the group $\GL_n(F)$ may have cuspidal non-supercuspidal representations
(see \S\ref{par92}).

Given a representation $\pi$ of $\GL_n(F)$ and a character $\chi$ of 
$F^\times$, we will write $\pi\chi$ for $\pi(\chi\circ\det)$.

\subsection{}

Let us fix an algebraic closure $\qlb$ of the field of $\ell$-adic numbers.
Let $\zlb$ denote its ring~of~in\-te\-gers,
and $\flb$ denote the residue field of $\zlb$.

We call an irreducible representation~$\pi$ of a
locally compact, totally disconnected 
% locally profinite 
group $\G$~on~a $\qlb$-vector space $V$
\emph{integral} if it~stabi\-lizes a~$\Zl$-lattice~$L$ in $V$.
In this case, we obtain a smooth~$\flb$-re\-presentation~$L\otimes\Fl$~of
$\G$~whose isomorphism class may depend on the choice of $L$.

If $\G$ is either the group of rational points of a connected reductive linear
algebraic 
$F$-group~or a finite group
(see \cite[Theorem~1]{VigIntegral} and the Brauer--Nesbitt principle),
the smooth $\flb$-representa\-tion $L\otimes\Fl$ has finite length,
and its~semisimplification is independent of the choice of $L$.
This semisimplification is called~the \emph{re\-duc\-tion modulo~$\ell$}
of~$\pi$, and is denoted by $\rl(\pi)$.

Given an irreducible $\flb$-representation~$\rho$ of $G$,
we call an irreducible~in\-tegral $\qlb$-representation with reduction 
modulo~$\ell$ equal 
to~$\rho$ a $\qlb$-\emph{lift} of~$\rho$. 

\section{Basic results}
\label{soupe}

In this section,
$p$ is an arbitrary prime number,
$\F/\F_0$ is a separable quadratic extension
and~$\R$ has characteristic~$\ell\neq p$.
We fix a positive integer $n\>1$.  

\subsection{}

Fundamental results of Flicker and Prasad \cite{Flicker,Prasad90,Prasad01}
on irreducible complex representa\-tions of $\GL_n(\F)$ distinguished by 
$\GL_n(\F_0)$ 
have been extended to irreducible~$\R$-re\-presentations~in 
\cite{VSANT19} Theorem 4.1.

\begin{theo}
\label{VSANTTHM41}
Let $\pi$ be an irreducible representation of $\GL_n(\F)$ distinguished by 
$\GL_n(\F_0)$.
\begin{enumerate}
\item The central character $c_\pi$ of $\pi$ is trivial on $\F_0^\times$.
\item The $\R$-vector space $\Hom_{\GL_n(\F_0)}(\pi,\R)$ has dimension $1$.
\item The contragredient $\pi^\vee$ of $\pi$ is isomorphic to $\pi^\sigma$.
\end{enumerate}
\end{theo}

We will say that a representation $\pi$ of $\GL_n(F)$
is $\sigma$-\textit{selfdual}
if $\pi^\vee$ is isomorphic to $\pi^\sigma$. 

\subsection{}\label{montoriol}

For supercuspidal representations,
we have the following Di\-cho\-tomy and Disjunction Theo\-rem
(\cite{Kable} Theorem 4, \cite{AnandKableTandon} Corollary 1.6 if $\ell=0$,
\cite{VSANT19} Theorem~10.8 if $p\neq2$ and
\cite{CLL}~Theorem~3.14 if $\ell\neq0,2$).

\begin{theo}
\label{rappelamical1}
Let $\rho$ be a $\s$-selfdual supercuspidal $\R$-repre\-sen\-ta\-tion of 
$\GL_n(\EE)$. 
\begin{enumerate}
\item
If $\ell=2$, then $\rho$ is distinguished.
\item
If $\ell\neq2$, then $\rho$ is either distinguished or $\x$-distinguished,
but not both.  
\end{enumerate} 
\end{theo}

\subsection{}

In this paragraph,
$\ell$~is~a pri\-me number different from $p$
and we will consider representations with co\-ef\-ficients in $\qlb$ or $\flb$. 
The following theorem is \cite{KuMaAsai} Theorem 3.4. 

\begin{theo}
\label{KuMaTh}
Let ${\pi}$ be an integral $\s$-selfdual cuspidal 
$\qlb$-representation of $\GL_n(\F)$.  
If ${\pi}$~is dis\-tin\-guished by $\GL_n(F_0)$,
then its reduction mod $\ell$ is
(irreducible, cuspidal and) distinguished.
\end{theo}

It follows that any $\s$-selfdual cuspidal $\flb$-representation of $\GL_n(\F)$ 
having a dis\-tin\-guished~lift to $\qlb$ is distinguished.
For super\-cus\-pidal representations,
one has the following converse 
(see~\cite{VSANT19} Theorem~10.11 if $p\neq2$,
and \cite{CLL} Theorem 3.4):

\begin{theo}
\label{rappelamical3}
Any $\GL_n(F_0)$-distinguished supercuspidal $\Fl$-repre\-sen\-ta\-tion 
of $\GL_n(\EE)$ has a $\GL_n(F_0)$-distinguished lift to $\qlb$.  
\end{theo}

We also have the following Distinguished Lift Theorem,
making Theorem \ref{rappelamical3} more precise. 

\begin{theo}
\label{rappelamical2}
Let $\rho$ be a $\s$-selfdual supercuspidal $\Fl$-repre\-sen\-ta\-tion 
of $\GL_n(\EE)$. 
\begin{enumerate}
\item 
The representation $\rho$ has a $\s$-selfdual lift to $\qlb$.
\item
Let $\mu$ be a $\s$-selfdual lift of $\rho$ to $\qlb$ 
and suppose that $\ell\neq2$.
Then $\mu$ is distinguished if~and only $\rho$
is distingui\-shed.
\end{enumerate}
\end{theo}

\begin{proof}
If $p\neq2$, this is \cite{VSANT19} Theorem 10.11. 
Assume now that $p=2$, thus $\ell\neq2$.

By Theorem \ref{rappelamical1},
the representation $\rho$ is either distinguished or 
$\x$-distinguished.
If it~is~distin\-guished,
it has a $\s$-selfdual lift thanks to 
Theorem \ref{rappelamical3} and Theorem \ref{VSANTTHM41}(3). 
If~it~is~$\x$-distingui\-shed,
fix a $\qlb$-char\-acter $\xi$ of $F^\times$ extending
the canonical $\qlb$-lift of $\x$.
The reduction~mod~$\ell$~of~$\xi$
is an $\flb$-char\-acter of $F^\times$ extending $\x$,
denoted $\chi$.
The representation $\rho\chi$ is distinguished~and~su\-per\-cus\-pi\-dal.
It thus has a $\s$-selfdual lift $\pi$.
Then $\pi\xi^{-1}$ is a distinguished lift of $\rho$.
This proves~(1).

Let $\mu$ be a $\s$-selfdual lift of $\rho$,
and assume that $\rho$ is distinguished. 
If $\mu$ is not distinguished,~it
must then be $\x$-distinguished.
By Theorem \ref{KuMaTh},
this implies that $\rho$ is $\x$-distinguished,
which~con\-tradicts the Di\-cho\-tomy and Disjunction Theorem.
Conversely, 
if $\mu$ is distinguished,
then $\rho$~is~dis\-tin\-guished thanks to Theorem \ref{KuMaTh}.
\end{proof}

\subsection{}
\label{par92}

From now on, we consider the case of cuspidal non-supercuspidal 
$R$-representations,
thus~$\ell$ is~a pri\-me number different from $p$.
Let us recall
how they are classified in terms of their~super\-cuspidal support. 

Recall that a representation $\pi$ of $\GL_n(F)$ on an $\R$-vector space $V$ 
is \textit{generic} if $V$ carries a~non-zero $\R$-linear form $\La$ such that
$\La(\pi(u)v) = \t(u)v$
for all $v\in V$ and all unipotent upper triangular matrices $u$, where
$\t(u) = \psi(u_{1,2}+\dots+u_{n-1,n})$
and $\psi$ is a non-trivial $R$-character of $F$.

Let $k\>1$ be a positive integer, 
and $\rho$ be a supercuspidal $\R$-representation of $\GL_k(\EE)$. 
Accor\-ding to \cite{MSc} 8.1,
for any $r\>1$, the in\-duced representation 
\begin{equation}
\label{INDRHO}
\rho\nu^{-(r-1)/2}\tdt\rho\nu^{(r-1)/2}
\end{equation}
contains a unique generic irreduci\-ble subquotient, denoted
$\St_{r}(\rho)$. 

Let $\et(\rho)$ be the smallest integer $i\>1$ 
such that $\rho\nu^i$ is isomorphic to $\rho$
and $t(\rho)$~be~the torsion number of $\rho$, 
that is, the number of unramified characters $\chi$ of $\EE^\times$
such that $\rho\chi$ is~iso\-mor\-phic~to $\rho$.
By \cite{MSjl} Lemme 3.6, these integers are related by the identity
\begin{equation}
\label{omegat}
\et(\rho) = \text{order of $q^{t(\rho)}$ mod $\ell$}.
\end{equation}
By \cite{MSc} Th\'eor\`eme~6.14, one has the following classification.

\begin{prop}
\label{Coupure}
Let $\pi$ be a cuspidal non-supercuspidal $\R$-representation of $\GL_n(\EE)$. 
\begin{enumerate}
\item 
There are a unique positive integer $r=r(\pi)\>2$ dividing $n$
and a supercuspidal represen\-tation $\rho$ of $\GL_{n/r}(\EE)$ 
such that $\pi$ is isomorphic to $\St_r(\rho)$.
\item
There is a unique integer $v\>0$ such that $r=\et(\rho)\ell^v$.
\item
Let $\rho'$ be a supercuspidal representation of $\GL_{n/r}(\EE)$.
The representation $\pi$ is isomorphic to $\St_{r}(\rho')$ if and only if 
$\rho'$ is isomorphic to $\rho\nu^i$ for some $i\in\ZZ$.
\end{enumerate} 
\end{prop}

Note that, conversely, by the same references,
if $\rho$ is a supercuspidal representation of $\GL_k(F)$
and $r=\et(\rho)\ell^v$ for 
some~$v\>0$, the representation $\St_r(\rho)$ is cus\-pidal. 

\subsection{}
\label{par93}

We now classify $\s$-selfdual cuspidal representations. 

\begin{lemm}
\label{fusion}
Let $\rho$ be a supercuspidal $\R$-representation of 
$\GL_k(\EE)$ for some $k\>1$.
Let~$r\>2$ be such that $\St_r(\rho)$ is cuspidal,
and suppose that $\St_r(\rho)$ is $\s$-selfdual.
Then there is an $i\in\ZZ$,~uni\-quely determined mod $\et(\rho)$,
such that $\rho^{\vee\s}$ is isomorphic to $\rho\nu^i$.
\end{lemm}

\begin{proof}
The representation $\St_r(\rho)$ is $\s$-selfdual if and only if the
representation
\begin{equation*}
\St_r(\rho)^{\vee\s} \simeq \St_r(\rho^{\vee\s})
\end{equation*}
is isomorphic to $\St_r(\rho)$. 
The result then follows from Proposition \ref{Coupure}.
\end{proof}

\begin{prop}
\label{fission}
Let $\pi$ be a cuspidal 
$\s$-selfdual representation of $\GL_{n}(\EE)$.
Set $r=r(\pi)$~and write $k=n/r$.
\begin{enumerate}
\item 
If $r$ is odd or $\ell=2$, 
there is a unique $\s$-selfdual supercuspidal representation $\rho$ of 
$\GL_k(\EE)$ such that $\pi$ is isomorphic to $\St_r(\rho)$.
\item
Suppose that $r$ is even and $\ell\neq2$.
\begin{enumerate}
\item 
There are a supercuspidal representation $\rho$ of $\GL_k(\EE)$
and an $i\in\{0,1\}$
such that $\pi$~is isomorphic~to $\St_r(\rho)$
and $\rho^{\vee\s}\simeq\rho\nu^i$.
\item
Let $\rho'$ be a supercuspidal representation of $\GL_k(\EE)$ 
and $j\in\{0,1\}$ 
such that $\pi$ is~iso\-morphic to $\St_r(\rho')$
and $\rho'^{\vee\s}\simeq\rho'\nu^{j}$.
Then $j=i$, 
and~ei\-ther $\rho'\simeq\rho$ or $\rho'\simeq\rho\nu^{r/2}$. 
\end{enumerate}
\end{enumerate}
\end{prop}

\begin{proof}
If $r=1$, the result is trivial.
Let us assume that $r\>2$.
Fix a supercuspidal irreducible representation $\rho$ of $\GL_k(\EE)$
such that $\pi$~is isomorphic~to $\St_r(\rho)$. 
By Lemma~\ref{fusion},~there is an $i\in\ZZ$
such that $\rho^{\vee\s}\simeq\rho\nu^i$.
Changing $\rho$ to $\rho'=\rho\nu^s$ for some $s\in\ZZ$
does not change $\St_r(\rho)$,
but~chan\-ges $i$ to $i-2s$.
If $r$ is odd or $\ell=2$,
then $\et(\rho)$ is odd,
thus $2\ZZ+\et(\rho)\ZZ=\ZZ$.
This~proves~(1).~Similar\-ly,~if~$r$ is even and $\ell\neq2$,
then $\et(\rho)$ is even:
we thus may assume that $i\in\{0,1\}$, proving (2.a).
Moreover, if $\rho'$ and $j$ are~as in (2.b),
then $j-i$ is even, thus $j=i$.
Moreover, $\rho'$ is isomorphic~to $\rho\nu^s$ for some 
$0\<s<\et(\rho)$ such that $\nu^{2st(\rho)}=1$, 
thus $\et(\rho)$ divides $2s$.
\end{proof}

\subsection{}

We will need the finite field analo\-gue of \ref{par92}
(see \cite{Vigbook} III.2.5 or \cite{Cabanes} Theorem 19.3). 

\begin{prop}
\label{Coupurefinie}
Let $\ke$ be a finite field of characteristic $p$. 
\begin{enumerate}
\item 
Let $f\>1$~be a positive integer and 
$\varrho$ be a~super\-cuspidal representation of $\GL_f(\kk)$.
\begin{enumerate}
\item 
For all $u\>1$, the in\-duced representation 
\begin{equation*}
\label{INDVARRHO}
\varrho\tdt\varrho
\quad
\text{($u$ times)}
\end{equation*}
has a unique generic~irreduci\-ble subquotient, 
denoted $\st_{u}(\varrho)$. 
\item
Let $\et(\varrho)$ be the order of $q^{f}$ mod $\ell$.  
The representation $\st_{u}(\varrho)$ is cuspidal
if and only if $u=1$ or $u=\et(\varrho)\ell^v$ for some $v\>0$.
\end{enumerate}
\item
Let $\Vc$ be a cuspidal representation of $\GL_n(\kk)$. 
There exist a unique integer $u=r(\Vc)\>1$ di\-vi\-ding~$n$
and a unique supercuspidal~represen\-ta\-tion $\varrho$ of $\GL_{n/u}(\kk)$ 
such that $\Vc\simeq\st_u(\varrho)$.
\end{enumerate}
\end{prop}

\subsection{}
\label{scaramuccia}

As in the previous paragraph,
$\kk$ is a~fi\-ni\-te field of characteristic $p$. 
Let us recall how~to~pa\-ra\-me\-trize cuspidal representations of $\GL_n(\kk)$
by~regu\-lar characters (\cite{Green}, \cite{Dipper} Theorem~3.5
and \cite{DipperJames,James}).

Let $\overline{\kk}$ be an algebraic closure of $\kk$.
For any integer $s\>1$,
let $\kk_{s}$ be the extension of~$\kk$ of~de\-gree $s$ contained in
$\overline{\kk}$.
Let $\Delta$ denote the group $\Gal(\kk_n/\kk)$.~A
cha\-racter of $\kk_n^\times$ is $\Delta$-regular if it is fixed by
no non-trivial element of $\Delta$. 

\begin{prop} 
\label{classifGreen}
\begin{enumerate}
\item 
Associated with any $\Delta$-regular $\qlb$-character $\xi$ of
$\kk_n^\times$,
there is~a~cus\-pi\-dal $\qlb$-representation
$\Vc_{\xi}$ of $\GL_n(\kk)$,
unique up to isomorphism,
such that
\begin{equation*}
\tr \Vc_{\xi} (x) = (-1)^{n-1}\cdot
\sum\limits_{\d\in\Delta} \xi(x^\d)
\end{equation*}
for all $x\in\kk_n^{\times}$ of degree $n$ over $\kk$,
where $\kk_n^{\times}$ is considered as a maximal torus in
$\GL_n(\kk)$.
\item 
The correspondence 
\begin{equation*}
\xi\mapsto\Vc_{\xi}
\end{equation*} 
induces a bijection from the set of 
$\Delta$-conjugacy~classes~of $\Delta$-regular
$\qlb$-cha\-rac\-ters of $\kk_n^\times$ to that of 
isomorphism classes of cuspidal $\qlb$-re\-pre\-sen\-ta\-tions of 
$\GL_n(\kk)$.
\end{enumerate}
\end{prop}

By reduction mod $\ell$,
we get the following classification.

\begin{prop}
\label{classifJames}
\begin{enumerate}
\item 
Given any $\Delta$-regular $\qlb$-character $\xi$ of $\kk_n^\times$, 
the reduction mod $\ell$ of $\Vc_{\xi}$,
denoted $\overline{\Vc}_\xi$,
is irre\-du\-cible~and cus\-pidal,
and it only depends on the reduction mod $\ell$ of $\xi$.
\item 
Reduction mod $\ell$ induces a bijection from the set of 
$\Delta$-conjugacy~classes~of $\flb$-characters~of 
$\kk_n^\times$~ha\-ving a $\Delta$-re\-gu\-lar lift to $\qlb$ 
to that of isomorphism classes of cuspidal $\flb$-re\-pre\-sen\-ta\-tions
of the group $\GL_n(\kk)$.
\item 
The integer $r(\overline{\Vc}_\xi)$ 
is the greatest divisor $r$ of $n$ such that 
the reduction of $\xi$ mod $\ell$~fac\-torizes through
a cha\-rac\-ter of $\kk_{n/r}^\times$.
\end{enumerate}
\end{prop}

\begin{defi}
\label{defparameter}
A \textit{parameter} of a cuspidal representation $\rho$ of $\GL_n(\kk)$
is a character of $\kk^\times_n$ whose $\Delta$-conjugacy~class
corresponds to $\rho$ by the bijection of either Proposition \ref{classifGreen} 
or \ref{classifJames}.
\end{defi}

\subsection{}

Finally,
we will need the following distinction criterion for cuspidal
$\qlb$-representations 
(see \cite{HakimMurnaghan}~Pro\-position 6.1
and \cite{Coniglio} Lemme 3.4.10) of $\GL_n(\kk)$
when $p$ is odd.

\begin{prop}
\label{chourineur}
Assume that $q$ is odd, $n$ is even and~write $n=2u$.
We consider the group $\GL_u(\kk)\times\GL_u(\kk)$ as a Levi subgroup of
$\GL_n(\kk)$.
Let $\xi$ be a $\Delta$-regular $\qlb$-charac\-ter of $\kk_n^\times$.
\begin{enumerate}
\item 
The following assertions are equivalent.
\begin{enumerate}
\item 
The cuspidal $\qlb$-representation $\Vc_\xi$ is
$\GL_u(\kk)\times\GL_u(\kk)$-distinguished. 
\item
The space of
$\GL_u(\kk)\times\GL_u(\kk)$-invariant linear forms on $\Vc_\xi$
has $\qlb$-dimension $1$.
\item
The cuspidal $\qlb$-representation $\Vc_\xi$ is selfdual. 
\item
The character $\xi$ is trivial on $\kk_u^\times$.
\end{enumerate}
\item
Assume that the conditions of {\rm (1)} are satisfied,
and fix an element $\a\in\kk_n^\times$ such that $\a\notin\kk_u^\times$
and $\a^2\in\kk_u^\times$. 
The element
\begin{equation}
\label{repy}
\ss =
\begin{pmatrix}
0 & {\rm id} \\
{\rm id} & 0 \end{pmatrix}
\in\GL_n(\kk),
\end{equation}
where ${\rm id }$ is the identity in $\GL_{u}(\kk)$, 
normalizes the group $\GL_u(\kk)\times\GL_u(\kk)$ and acts
on~the~$\qlb$-vector spa\-ce of $\GL_u(\kk)\times\GL_u(\kk)$-invariant
linear forms on $\Vc_\xi$ by the sign $-\xi(\a)$.
\end{enumerate}
\end{prop}

\begin{rema}
\label{comtesseartoff}
Suppose that $\Vc_\xi$ is $\GL_u(\kk)\times\GL_u(\kk)$-distingui\-shed.
By \cite{VSANT19} Lemma 2.6, 
the cuspidal $\flb$-representation
$\overline{\Vc}_\xi$ is $\GL_u(\kk)\times\GL_u(\kk)$-distinguished as well.
More precisely, 
if we fix 
a non-zero $\GL_u(\kk)\times\GL_u(\kk)$-in\-va\-riant
$\qlb$-linear form ${\it\La}$ on $\Vc_\xi$ together with 
a $\GL_n(\kk)$-stable $\zlb$-lattice ${L}\subseteq\Vc_\xi$,
then the associated $\flb$-linear form 
\begin{equation*}
\overline{{\it\La}} :
{L}\otimes\flb \to {\it\La}({L})\otimes\flb
\end{equation*}
is non-zero and $\GL_u(\kk)\times\GL_u(\kk)$-in\-va\-riant.
Moreover, 
if $\ss$ acts on ${\it\La}$ by a sign
$c\in\{-1,1\}\subseteq\overline{\mathbb{Q}}{}^\times_\ell$,
then $\ss$ acts on $\overline{{\it\La}}$
by the image of $c$ in $\overline{\mathbb{F}}{}^\times_\ell$.
\end{rema}

\section{Reduction to level zero}
\label{pasfaim}

In this section,
$p$ is odd,
$\ell$~is any pri\-me number different from $p$ and
$\R$ has characteristic $0$ or $\ell$.
Let us fix a positive integer $n\>1$, and set $\G=\GL_n(\EE)$.
We fix~a cha\-rac\-ter
\begin{equation}
\label{psiF}
\psi : \F \to \R^\times
\end{equation}
which is trivial on $\pp_\F$ but not on $\oo_\F$.

\subsection{}
\label{prelim}

Let us recall the definitions and main results of
\cite{BK,BHEffective,MSt,AKMSS} which we will need.

Let $[\aa,\b]$ be a simple stratum in 
the algebra $\Mat_{n}(\F)$ of $n\times n$ matrices with entries 
in $\F$.  
Recall that $\aa$ is a hereditary $\oo_\F$-order of $\Mat_{n}(\F)$
and $\b$ is an element of $\Mat_{n}(\F)$ such that
\begin{itemize}
\item
the $\F$-algebra $\E=\F[\b]$ is a field, and
\item
the multiplicative group $\E^\times$ normalizes $\aa$
\end{itemize}
(plus an extra technical condition on
$\b$ which is not necessary to recall here: see \cite{BK} 1.5.5).

Let $\Kk_{\aa}$ be the normalizer of $\aa$ in $\G$
and $\pp_\aa$ be its Jacobson radical, and set
$\U^1_{\aa} = 1 + \pp_\aa$. 
Let $\B$ be the cen\-trali\-zer of $\E$ in $\Mat_{n}(\F)$.
The inter\-section $\bb=\aa\cap\B$ is a hereditary order in $\B$.

Associated with $[\aa,\b]$ in \cite{BK} Chapter 3,
there are compact mod centre open subgroups
\begin{equation*}
\H^1(\aa,\b)\subseteq
\BJ^1(\aa,\b)\subseteq\BJ^0(\aa,\b)\subseteq\BJ(\aa,\b)\subseteq\Kk_{\aa}
\end{equation*}
and a non-empty finite set $\Cc(\aa,\b)$
of characters of $\H^1(\aa,\b)$ called \textit{simple characters}, 
depending on the choice of \eqref{psiF}.
We write 
$\BJ=\BJ(\aa,\b)$, $\BJ^0=\BJ^0(\aa,\b)$, $\BJ^1=\BJ^1(\aa,\b)$ 
and $\H^1=\H^1(\aa,\b)$~for simplicity.

We will only be interested in the case where $\bb$ is a maximal order in $\B$,
in which case the~sim\-ple~stratum $[\aa,\b]$ and the simple characters 
in $\Cc(\aa,\b)$ are said to be \textit{maximal}.
For the following result, see \cite{BHEffective} {\rm 2.1, 3.2}
and \cite{BK} {\rm 5.1.1}.

\begin{prop}
\label{patel}
Let $[\aa,\b]$ be a maximal simple stratum.
\begin{enumerate}
\item
The group $\BJ^0$ is the unique maximal compact subgroup of $\BJ$,
and $\BJ^1$ is its unique maximal normal pro-$p$-subgroup.
\item
One has $\BJ=\E^\times\BJ^0=(\BJ\cap\B^\times)\BJ^1$ and
\begin{equation}
\label{lapremiereegalite1}
\BJ\cap\B^\times=\Kk^{}_{\bb}, 
\quad
\BJ^0\cap\B^\times=\bb^\times,
\quad
\BJ^1\cap\B^\times=\U^1_{\bb}.
\end{equation}
\item
There is an isomor\-phism of $\E$-algebras 
\begin{equation}
\label{melon}
\B\simeq\Mat_m(\E),
\quad
m=n/[\E:\F],
\end{equation}
sending $\bb^\times$ to the maximal compact open subgroup $\GL_m(\oo_\E)$, 
which induces iso\-mor\-phisms 
\begin{equation}
\label{JJ1UU1GLmax}
\BJ^0/\BJ^1 \simeq \bb^\times/\U^1_{\bb} \simeq \GL_{m}(\ll)
\end{equation}
where $\ll$ is the residue field of $\E$.
\item 
Given any simple character $\t\in\Cc(\aa,\b)$,
we have
\begin{enumerate}
\item 
the normalizer of $\t$ in $\G$ is equal to $\BJ$, and
\item
there is an irreducible representation $\n$ of $\BJ^1$,
unique up to isomorphism,
whose~res\-triction to $\H^1$ contains $\t$,
and such a representation extends to $\BJ$.
\end{enumerate}
\end{enumerate}
\end{prop}

The representation $\n$ of (4.b) is called the
\textit{Heisenberg representation}
associated with $\t$.
If $\bk$ is a representation of
$\BJ$ extending $\n$, 
any other extension of $\n$ to $\BJ$ has the form $\bk\bx$
for a unique character $\bx$ of $\BJ$ trivial on $\BJ^1$.

\begin{rema}
\label{rembase1}
\begin{enumerate}
\item 
An isomorphism as in Proposition \ref{patel}(3) comes from the
choice of~an $\oo_\E$-basis of an $\oo_\E$-lattice $\Ll$ in $\F^n$
whose endomorphism algebra is $\bb$. 
\item
Changing the isomorphism \eqref{melon},
that is, changing the basis of $\Ll$ above, 
has the effect of conjugating 
the identification \eqref{JJ1UU1GLmax}~by 
an inner automorphism of $\GL_{m}(\ll)$.
\end{enumerate}
\end{rema}

A character $\t$ of an open pro-$p$-subgroup $\H$ of $\G$ will be called a
\textit{maximal simple~cha\-rac\-ter} if there is a maximal simple~stra\-tum 
$[\aa,\b]$ in $\Mat_n(\F)$ such that $\H=\H^1(\aa,\b)$ and
$\t\in\Cc(\aa,\b)$.~Gi\-ven~a maximal simple~cha\-rac\-ter $\t$ of $\G$,
we will write $\H^1_\t$ for the group on which $\t$ is defined,~$\BJ^{}_\t$
for its $\G$-normalizer, 
$\BJ^0_\t$ for its unique maximal compact subgroup, 
$\BJ^1_\t$ for its unique maximal~nor\-mal pro-$p$-subgroup 
and $\T$ for the maximal tamely ramified extension of $\F$ in 
$\E=\F[\b]$.
The following result shows how the latter depends on the choice of $[\aa,\b]$
(see \cite{BHEffective} 2.1, 2.5 and 2.6).

\begin{prop}
Let $[\aa',\b']$ be a sim\-ple stratum such that $\t\in\Cc(\aa',\b')$,
and set $\E'=\F[\b']$. 
\begin{enumerate}
\item 
The orders $\aa$, $\aa'$ are equal and $\E$, $\E'$ have the same degree over 
$\F$. 
\item
The simple stratum $[\aa,\b']$ is maximal.
\item
The maximal tamely ramified extension of $\F$ in $\E'$ is
$\BJ_\t^1$-conjugate to $\T$.
\end{enumerate}
\end{prop}

It follows that the $\G$-conjugacy class of the simple character $\t$
uniquely determines the integer $m$ in \eqref{melon},
as well as the extension $\T$ up to $\G$-conjugacy
(or equivalently up to $\F$-isomorphism). 
However, the fields $\E$, $\E'$ need not be isomorphic
(see \cite{BHEffective} Example 2.1).

\subsection{}
\label{nom42}

In this paragraph only,
we let $n$ vary among all positive integers,
and consider the set
\begin{equation*}
\Cc_{\rm max}(\F) = \bigcup\limits_{[\aa,\b]} \Cc(\aa,\b)
\end{equation*}
where the union is taken over all maximal simple strata of 
$\Mat_n(\F)$, for any $n\>1$.
It is endowed with an equivalence relation called \textit{endo-equivalence}
(\cite{BHLTL1,BHIMRN}).
An equivalence class for this equi\-va\-len\-ce relation is called an
\textit{endo-class} (see \cite{AKMSS} 3.2).

Given a maximal simple character $\t\in\Cc_{\rm max}(\F)$ with endo-class 
$\TT$, 
the degree $[\E:\F]$ and the $\F$-isomorphism class of its tame parameter 
field $\T$ only depend on $\TT$.
They are called the \textit{degree} and \textit{tame parameter field} of 
$\TT$, respectively.

For a given $n$,
any two maximal simple characters of $\GL_n(\F)$
are endo-equivalent if and only if they are $\GL_n(\F)$-conjugate.  

\begin{rema}
Note that endo-equivalence is defined in \cite{BHLTL1,BHIMRN} for arbitrary 
simple characters, not only for maximal ones,
but we will not need this extra generality. 
\end{rema}

\subsection{}
\label{rappelstypesssd}

We go back to the situation of Paragraph \ref{prelim},
assuming further that the character $\psi$ of~\eqref{psiF} is $\s$-invariant,
which is possible since $p\neq 2$.
As in \cite{VSANT19}, we will say that:
\begin{itemize}
\item 
a~simple stratum $[\aa,\b]$~in $\Mat_n(\F)$ is $\s$-\textit{selfdual} 
if $\aa$ is $\s$-stable and $\s(\b)=-\b$,
\item
a simple character $\t$ is $\s$-\textit{self\-dual}
if the group $\H^1_\t$ is $\s$-stable and $\t^{-1}\circ\s=\t$,
\item
an endo-class $\TT$ of (maximal) simple characters is
$\s$-\textit{selfdual} if for some (or equivalently for any)
$\t\in\TT$, the character $\t^{-1}\circ\s$ is in $\TT$.
\end{itemize} 

\begin{prop}
\label{pechessd}
Let $\t$ be a $\s$-selfdual maximal simple character.
\begin{enumerate}
\item
There is a $\s$-stable simple stratum $[\aa,\b]$ such that $\t\in\Cc(\aa,\b)$. 
\item
Let $\E_0$ be the $\s$-fixed points of $\E$
and $\ll_0$ be its residue field.
There exists an~iso\-mor\-phism \eqref{melon} inducing 
an isomorphism \eqref{JJ1UU1GLmax} 
which transports the action of $\s$ on $\BJ^0/\BJ^1$~to
\begin{enumerate}
\item 
the action of the non-trivial element of $\Gal(\ll/\ll_0)$
on $\GL_m(\ll)$ if $\E/\E_0$ is unramified,
\item 
the adjoint action of
\begin{equation}
\label{banane}
\begin{pmatrix}
-{\rm id}_{i} & 0 \\ 0 & {\rm id}_{m-i}
\end{pmatrix} \in \GL_m(\ll) 
\end{equation}
on $\GL_m(\ll)$ if $\E/\E_0$ is ramified,
for a uniquely determined integer $i\in\{0,\dots,\lfloor m/2\rfloor\}$.
\end{enumerate}
\end{enumerate}
\end{prop}

\begin{rema}
\label{rembase2}
Choose a ba\-sis~of $\Ll$ as in Remark \ref{rembase1},
with the additional condition that~its vectors are $\s$-invariant. 
The isomorphism \eqref{melon} associated with it then transports
the action of $\s$ on $\B$ to that of the generator of $\Gal(\E/\E_0)$
on $\Mat_m(\E)$.
When $\E/\E_0$ is ramified,
this~cor\-res\-ponds to the case where $i=0$ in \eqref{banane}. 
\end{rema}

If $\t$ is a $\s$-selfdual maximal simple character,
we will write $\T_0$ for the maximal tamely ra\-mified extension of
$\F_0$ in $\E_0$, that is, $\T_0=\T\cap\E_0$.
By \cite{AKMSS} Lemma 4.10,
the canonical homomorphism
\begin{equation}
\label{T0FT}
\T_0\otimes_{\F_0}\F \to \T
\end{equation}
is an isomorphism.
Also, $\T/\T_0$ and $\E/\E_0$ have the same ramification index. 
By \cite{AKMSS} Lemma 4.29,
the $\F_0$-isomorphism~class of $\T_0$ is uniquely determined by
the endo-class $\TT$ of $\t$. 
And it follows from \eqref{T0FT} that
the $\F_0$-isomorphism~class of $\T_0$ determines the 
$\F$-isomorphism class of $\T$. 

The following result is given by \cite{VSANT19} Proposition 6.12, Lemma 6.20
and \cite{VSPTB}~Lemme 3.28.~(The latter reference in \cite{VSPTB} 
is for representations with coefficients~in $\R=\CC$,
but its proof is still valid in the $\ell$-modular case.)

\begin{prop}
\label{etachi}
Let $\t$ be a $\s$-selfdual maximal simple character. 
\begin{enumerate}
\item 
The Heisenberg representation $\n$ of $\t$ is $\s$-selfdual and 
$\BJ^1\cap\G^\s$-distinguished, and the space
$\Hom_{\BJ^1\cap\G^\s}(\n, \R)$ has dimension $1$.
\item
For any representation $\bk$ of $\BJ$ extending $\n$, there are
\begin{enumerate}
\item 
a unique character $\bx$ of $\BJ$ trivial on $\BJ^1$ such that
$\bk^{\vee\s}$ is isomorphic to $\bk\bx$,
\item
a unique character $\chi$ of $\BJ\cap\G^\s$ trivial on
$\BJ^1\cap\G^\s$ such that
\begin{equation}
\Hom_{\BJ^1\cap\G^\s}(\n,\R) = \Hom_{\BJ\cap\G^\s}(\bk,\chi^{-1}),
\end{equation}
and the restriction of $\bx$ to $\BJ\cap\G^\s$ is equal to $\chi^2$. 
\end{enumerate}
\item
Given a representation $\bk$ as in {\rm (2)}
and an irreducible representation $\bt$ of $\BJ$ trivial on $\BJ^1$, 
the canonical linear map
\begin{equation}
\Hom_{\BJ^1\cap\G^\s}(\n, \R) \otimes \Hom_{\BJ\cap\G^\s}(\bt,\chi)
\to
\Hom_{\BJ\cap\G^\s}(\bk\otimes\bt,\R)
\end{equation}
is an isomorphism of $\R$-vector spaces. 
\item
There exists a $\s$-selfdual representation of $\BJ$ extending $\n$.  
\end{enumerate}
\end{prop}

Conversely,
let $\TT$ be a $\s$-selfdual endo-class of degree dividing $n$.
By \cite{AKMSS} Section 4,
it contains a $\s$-selfdual maximal simple character $\t$ in $\G$,
and we have the following classifica\-tion.

\begin{prop}
\label{poussin}
Let $\T/\T_0$ be the quadratic extension associated~with $\TT$ as above,
and let us write $m=n/\deg(\TT)$. 
\begin{enumerate}
\item 
If $\T/\T_0$ is unramified,
the $\G$-conjugacy class of $\t$ contains a unique $\G^\s$-conjugacy class 
of $\s$-selfdual simple charac\-ters.
\item 
If $\T/\T_0$ is ramified,
the number of $\G^\s$-conjugacy classes of $\s$-selfdual simple characters~in 
the $\G$-conjugacy class of $\t$ is equal to $\lfloor m/2\rfloor+1$,
each class corresponding bijectively to an integer
$i\in\{0,\dots,\lfloor m/2\rfloor\}$
characte\-rized by Propo\-si\-tion \ref{pechessd}(2.b).
\end{enumerate} 
\end{prop}

\begin{rema}
\label{indice}
When $\T/\T_0$ is ramified, 
we define (as in \cite{AKMSS,VSANT19})
the \textit{index} of~a $\s$-selfdual~ma\-xi\-mal simple character $\t$ 
to be the integer $i\in\{0,\dots,\lfloor m/2\rfloor\}$ 
associated with its $\G^\s$-conjugacy class.
By \cite{AKMSS} Remark 4.28 or \cite{VSANT19} 5.D,
if $\t$ has index $0$ and if
\begin{equation*}
y_i = {\rm diag}(\w,\dots,\w,1,\dots,1) \in\GL_{m}(\E)\simeq\B^\times, 
\quad
\text{$\w$ a uniformizer of $\E$ occurring $i$ times,} 
\end{equation*}
for some $i\in\{0,\dots,\lfloor m/2\rfloor\}$, 
then $\t^{y_i}$ is a $\s$-selfdual maximal simple character of index $i$.
\end{rema}

\subsection{}
\label{previousparagraph}

Let $\t$ be a maximal simple~cha\-rac\-ter,
and $[\aa,\b]$ be a simple stratum such that $\t\in\Cc(\aa,\b)$.
As in \ref{prelim}, write $\BJ=\BJ_\t^{}$, $\BJ^0=\BJ^0_\t$, $\BJ^1=\BJ^1_\t$ 
and $\H^1=\H^1_\t$.
Let $\n$ be the Heisenberg representation of $\BJ^1$ associated with 
$\t$.
In this paragraph, we give a slightly generalized version of \cite{BHEffective} 3.3. 

Fix an $\oo_\E$-lattice $\Ll$ in $\V=\F^n$ whose endomorphism algebra is 
$\bb$. 
(It is uniquely determined up to homothety as $\bb$ is maximal.)
Fix a divisor $u\>1$ of $m$ and 
decompositions 
\begin{equation}
\label{decVL}
\Ll = \Ll_* \oplus\dots\oplus \Ll_*,
\quad
\V = \V_* \oplus\dots\oplus \V_*,
\end{equation}
in which $\V_*$ is an $\E$-vector space of dimension $m/u$
% $1$
and $\Ll_*$ is an $\oo_\E$-lattice of rank $m/u$
% $1$
in~$\V_*$. 
It defines a Levi subgroup $\M$ of $\G$.
Fix a pair $(\Q_-,\Q_+)$ of $\M$-opposite parabolic subgroups of~$\G$
with Levi component $\M$, 
and write $\N_-$, $\N_+$ for their~unipotent radicals.
Define $\aa_*=\End_{\oo_\F}(\Ll_*)$.
It is a hereditary order, 
and $[\aa_*,\b]$ is a maximal simple stratum in 
$\End_\F(\V_*)$. 
Write $\BJ^{}_*$, $\BJ^{0}_*$,~$\BJ^{1}_*$ and $\H^1_*$ 
for the subgroups associated with it. 
Compare with \cite{BHEffective} 3.3, 
which corresponds to~the~par\-ticular ca\-se where $u=m$.

The next result follows from \cite{BHLTL1} Example 10.9 % De\-fi\-nition 10.8
(compare with Lemma 1 in \cite{BHEffective} 3.3).

\begin{lemm}
\label{ArseneLup1}
\begin{enumerate}
\item 
There are Iwahori decompositions
\begin{eqnarray*}
\H^1 &=& (\H^1\cap\N_-) \cdot (\H^1\cap\M) \cdot (\H^1\cap\N_+), \\
\H^1\cap\M &=& \H^1_* \times \dots \times \H^1_*
\end{eqnarray*}
and
\begin{eqnarray*}
\BJ^1 &=& (\BJ^1\cap\N_-) \cdot (\BJ^1\cap\M) \cdot (\BJ^1\cap\N_+), \\
\BJ^1\cap\M &=& \BJ^1_* \times \dots \times \BJ^1_*.
\end{eqnarray*}  
\item
The character $\t$ is trivial on $\H^1\cap\N_-$, $\H^1\cap\N_+$ and 
there exists a unique simple character $\t_*\in\Cc(\aa_*,\b)$ such that $\t$
agrees with $\t_*\otimes\dots\otimes\t_*$ on $\H^1\cap\M$.
\end{enumerate}
\end{lemm}

Moreover,
the map $\t\mapsto\t_*$ defined by Lemma \ref{ArseneLup1}(2)
is a bijection from $\Cc(\aa,\b)$ to $\Cc(\aa_*,\b)$:
this is an instance of the transfer of \cite{BK} 3.6.

Let $\n_*^{}$ denote the Heisenberg representation of $\BJ_*^1$ associated with 
$\t_*^{}$. 
Compare the next~result with Lemma 2 in \cite{BHEffective} 3.3 and 
the~dis\-cussion~after it. 

\begin{lemm}
\label{ArseneLup3}
Let $\bk_*$ be a representation of $\BJ_*$ extending $\n_*$.
\begin{enumerate}
\item 
The set $\BJ_+=(\H^1\cap\N_-)\cdot(\BJ\cap\Q_+)$
is a group, and there is a unique representation $\bk_+$ of $\BJ_+$
which is trivial on $\H^1\cap\N_-$, $\BJ^0\cap\N_+$ and agrees with
$\bk_*\otimes\dots\otimes\bk_*$ on $\BJ\cap\M$.
\item
The representation $\widetilde{\bk}_+$ of 
$(\J^1\cap\N_-)\cdot(\BJ\cap\Q_+)=\BJ^1\BJ_+$
induced by $\bk_+$ extends $\n$. 
\item
There is a unique irreducible representation $\bk$ of $\BJ$
extending $\widetilde{\bk}_+$.
\end{enumerate} 
\end{lemm}

\begin{proof}
That $\BJ_+$ is a group follows from the fact that $\H^1$ is normalized by 
$\BJ$, thus by $\BJ\cap\Q_+$.
The existence of $\bk_+$ follows from the containment
\begin{equation*}
(\BJ^0\cap\N_+)\cdot(\H^1\cap\N_-) \subseteq
(\H^1\cap\N_-)\cdot(\BJ^1\cap\M)\cdot(\BJ^0\cap\N_+)
\end{equation*} 
(see the argument of \cite{VS2} 2.3).
Mackey's formula implies that
the restriction of $\widetilde{\bk}_+$ to $\BJ^1$ is 
\begin{equation*}
\Ind^{\BJ^1}_{\BJ^1\cap\BJ_+} (\bk_+).
\end{equation*}
The restriction of $\bk_+$ to
$\BJ^1\cap\BJ_+=(\H^1\cap\N_-)\cdot(\BJ^1\cap\Q_+)$
is the unique representation $\n_+$
which is trivial on $\H^1\cap\N_-$, $\BJ^1\cap\N_+$ and agrees with
$\n_*\otimes\dots\otimes\n_*$ on $\BJ^1\cap\M$.
The representation it induces to~$\BJ^1$
is isomorphic to $\n$:
indeed,
this representation contains $\t$ by Lemma \ref{ArseneLup1}(2),
and it has dimension
\begin{eqnarray*}
\dim (\n_*\otimes\dots\otimes\n_*) \cdot (\BJ^1\cap\N_-:\H^1\cap\N_-)
&=& (\BJ^1\cap\M:\H^1\cap\M)^{1/2}\cdot (\BJ^1\cap\N_-:\H^1\cap\N_-) \\
&=& (\BJ^1:\H^1)^{1/2}
\end{eqnarray*}
which is the dimension of $\n$ (see \cite{MSt} 2.3).

It remains to prove (3).
First,
uniqueness follows from the fact that any two such extensions differ from
a character of $\BJ$ trivial on $\BJ^1\BJ_+$,
and such a character is trivial since $p\neq2$.
Existen\-ce~follows from \cite{BK} 5.2.4
(see \cite{MSt} 2.4 in the modular case).
\end{proof}

The reader will pay attention to the fact that $\BJ\cap\M$
is not equal to $\BJ_*\tdt\BJ_*$ in general (unless $u=1$),
but is generated by $\BJ^0_*\tdt\BJ^0_*$ and $\E^\times$
(the latter being diagonal in $\M$).

\begin{lemm}
\label{ArseneLup4}
\begin{enumerate}
\item 
The map
\begin{equation}
\label{kappa0kappa}
\bk_* \mapsto \bk
\end{equation}
from isomorphism classes of representations of $\BJ_*$ extending $\n_*$
to isomorphism classes of repre\-sen\-tation of $\BJ$ extending $\n$ is 
surjective. 
\item
Any two~re\-pre\-sentations of $\BJ_*$ extending $\n_*$ have the
same image 
if and only if they are twists of each other by a character of $\BJ_*^{}$
trivial on $\BJ_*^0$ and of order dividing $u$.
\end{enumerate}
\end{lemm}

\begin{proof}
The case where $u=m$ is given by Corollary 1 in \cite{BHEffective} 3.3, 
and the general case follows by transitivity. 
\end{proof}

\begin{rema}
\label{conjugaisony}
Suppose that $u$ is equal to $m$.
Let $y\in\M\cap\B^\times$ and write $\t'=\t^y\in\Cc(\aa^y,\b)$.
The
groups associated~with $\t'$ are $\BJ'=\BJ^y$, etc. 
The group isomorphism $\B^\times\simeq\GL_m(\E)$~iden\-tifies
$\M\cap\B^\times$ with the diagonal torus
$\E^\times\tdt\E^\times$,
and $\E^\times$ normalizes $\t_*$.
The character $\t'$ is thus trivial~on
$\H^{1y}\cap\N_-$, $\H^{1y}\cap\N_+$ and 
agrees with $\t_*\otimes\dots\otimes\t_*$ on
$\H^{1y}\cap\M=\H^1\cap\M$.
If $\bk_*$ is a~represen\-ta\-tion~of~$\BJ_*$ extending $\n_*$,
the repre\-sen\-ta\-tion of $\BJ'$ corresponding to it by 
\eqref{kappa0kappa} is $\bk^y$. 
\end{rema}

\subsection{}
\label{typeneutre}

Suppose now that the simple character $\t$ and the simple stratum $[\aa,\b]$ 
of \ref{previousparagraph} are~$\s$-self\-dual.
The groups $\BJ$, $\BJ^0$, $\BJ^1$ and $\H^1$ are thus 
$\s$-stable.
Suppose also that the decompositions \eqref{decVL}~are $\s$-sta\-ble. 
The Levi subgroup $\M$ is thus $\s$-stable,
and we may assume that 
$\Q_-$, $\Q_+$ are $\s$-stable as well.
Also,
the simple~stra\-tum $[\aa_*,\b]$
and the simple~cha\-rac\-ter $\t_*$ given by Lemma \ref{ArseneLup1}
are $\s$-self\-dual.
Let $\G_*$ denote the group $\Aut_\F(\V_*)$.
It is isomorphic to $\GL_{n/u}(\F)$.

We may also assume that our choice of basis induces an
isomorphism of groups \eqref{JJ1UU1GLmax}~bet\-ween 
$\BJ^0/\BJ^1$ and $\GL_m(\ll)$ as in Pro\-po\-sition~\ref{pechessd}(4), 
transporting the action of $\s$ on $\BJ^0/\BJ^1$ to
\begin{itemize}
\item 
the action of the non-trivial element of $\Gal(\ll/\ll_0)$
on $\GL_m(\ll)$ if $\T/\T_0$ is unramified,
\item 
the adjoint action of \eqref{banane} 
on $\GL_m(\ll)$
for some $i\in\{0,\dots,\lfloor m/2\rfloor\}$,
if $\T/\T_0$ is ramified. 
\end{itemize} 
Let $\bk_*$ be a representation of $\BJ_*$ extending $\n_*$,
and let $\bk$ correspond to it by \eqref{kappa0kappa}. 

\begin{lemm}
\label{preEtretat}
If $\bk_*$ is $\s$-selfdual, then $\bk$ is $\s$-selfdual. 
\end{lemm}

\begin{proof}
First, $\bk_+^{\vee\s}$
is trivial on both $\H^1\cap\N_-$, $\BJ^1\cap\N_+$ and agrees 
with $\bk_*^{\vee\s}\otimes\dots\otimes\bk_*^{\vee\s}$ on $\BJ\cap\M$.
If $\bk_*^{}$ is $\s$-self\-dual, it follows by uniqueness that 
$\bk_+^{\vee\s}$ is $\s$-self\-dual,
thus $\widetilde{\bk}_+^{\vee\s}$ is $\s$-self\-dual as well. 
The unique representation of $\BJ$
extending $\widetilde{\bk}_+^{\vee\s}$ is $\bk^{\vee\s}$,
hence $\bk$ is $\s$-self\-dual.
\end{proof}

We will need the following lemma. 
Let $\w$ be a uniformizer of $\E$ such that
\begin{equation}
\label{fixw}
\s(\w)=
\left\{ 
\begin{array}{rl}
\w & \text{if $\T/\T_0$ is unramified}, \\ 
-\w & \text{if $\T/\T_0$ is ramified}.
\end{array}\right.
\end{equation}
Note that the group $\BJ$~is generated by $\BJ^0$ and $\w$. 

\begin{lemm}
\label{YumAubergineindicei}
The group $\BJ\cap\G^\s$~is generated by 
$\BJ^0\cap\G^\s$ and an element $\w'$ such that 
\begin{enumerate}
\item
$\w'=\w$ if $\T/\T_0$ is unramified,
\item
$\w'=\w^2$ if $\T/\T_0$ is ramified and $m\neq2i$,
\item
$\w'\in\w\BJ^0$ if $\T/\T_0$ is ramified and $m=2i$.
\end{enumerate}
\end{lemm}

\begin{proof}
If $\T/\T_0$ is unramified, see \cite{VSANT19} Lemma 9.1.
Suppose that $\T/\T_0$ is ramified, 
and assume that there is an $x\in\BJ\cap\G^\s$,
$x\notin\langle\w^2,\BJ^0\cap\G^\s\rangle$.
We have $x=\w^{v}y$ where $v\in\ZZ$ is odd and $y\in\BJ^0$
satisfies $\s(y)=-y$.
Reducing mod $\BJ^1$,
we~get an $\a\in\GL_m(\ll)$ such that $\s(\a)=-\a$. 
Since the involution $\s$ acts on $\GL_m(\ll)$ by conjugacy by 
\begin{equation*}
\d = {\rm diag}(-1,\dots,-1,1,\dots,1)
\end{equation*}
(where $-1$ occurs $i$ times), 
this implies that $\d$ and $-\d$ are $\GL_m(\ll)$-conjugate, 
thus $m=2i$.~Conversely, if $m=2i$, then 
\begin{equation}
\label{defwavants}
w = 
\begin{pmatrix}
0&{\rm id}_i\\ {\rm id}_i&0
\end{pmatrix}
\in\BJ^0\cap\B^\times = \GL_{2i}(\oo_\E)
\end{equation}
is $\s$-anti-invariant, 
and $\w'=\w w$ has the required property. 
\end{proof}

We now investigate the behavior of the map \eqref{kappa0kappa} 
with respect to distinction.
The case where $u=m$ will be sufficient for our purpose
(see Paragraph \ref{valerie}).

\begin{lemm}
\label{Etretat}
Suppose that $u=m$ and 
$\bk_*^{}$ is $\BJ_*^{}\cap\G_*^\s$-distin\-guished.
\begin{enumerate}
\item 
If $\T/\T_0$ is unramified, 
or if $\T/\T_0$ is ramified and $m\neq2i$,
the representation $\bk$ is $\BJ\cap\G^\s$-distin\-guished.
\item 
If $\T/\T_0$ is ramified and $m=2i$,
there exists a quadratic character $\bx$ of $\BJ$ trivial on $\BJ^0$~such
that $\bk\bx$ is $\BJ\cap\G^\s$-distin\-guished.
\end{enumerate}
\end{lemm}

\begin{proof}
The representation $\bk_+$ is $\BJ_+\cap\G^\s$-distinguished,
thus $\widetilde{\bk}_+$ is $\BJ^1\BJ_+\cap\G^\s$-distinguished.
It follows that $\bk$ is $\BJ^1\BJ_+\cap\G^\s$-distinguished.
Let $\chi$ be the character of $\BJ\cap\G^\s$ associated with $\bk$ by
Proposition \ref{etachi}.
It is trivial on $\BJ^1\BJ_+\cap\G^\s$.
Restricting to $\BJ^0\cap\G^\s$, it is a character of
\begin{equation*}
(\BJ^0\cap\G^\s)/(\BJ^1\cap\G^\s) \simeq \GL_m(\ll)^\s.
\end{equation*}
Since $\BJ\cap\M\subseteq\BJ_+$, 
it is trivial on the image of $(\BJ\cap\M)\cap(\BJ^0\cap\G^\s)$ in 
$\GL_m(\ll)^\s$,
which is made of the $\s$-fixed points of
the diagonal torus $\textsf{M}=\ll^\times\tdt\ll^\times$.

If $\T/\T_0$ is unramified,
we have $\GL_m(\ll)^\s=\GL_m(\ll_0)$ and
$\textsf{M}^\s=\ll_0^\times\tdt\ll_0^\times$,
thus $\chi$ is trivial on $\BJ^0\cap\G^\s$.
If $\T/\T_0$ is ramified, 
we have $\GL_m(\ll)^\s=\GL_i(\ll)\times\GL_{m-i}(\ll)$ and
$\textsf{M}^\s=\textsf{M}$.
Again, $\chi$ is trivial on $\BJ^0\cap\G^\s$.

By Lemma \ref{YumAubergineindicei}, 
it remains to consider the value of $\chi$ at $\w'$.
If $\T/\T_0$ is unramified,
or if $\T/\T_0$ is ramified and $m\neq2i$,
we have $\w' \in\BJ^1\BJ_+\cap\G^\s$, 
thus $\chi$ is trivial.

Now assume that $\T/\T_0$ is ramified and $m=2i$. 
Let $\bx$ be the quadratic character of $\BJ$ trivial on $\BJ^0$ defined by
$\xi(\w')=\chi(\w')$.
Then $\bk\bx$ is $\BJ\cap\G^\s$-distinguished.
\end{proof}

We will prove in Paragraph \ref{valerie} that the quadratic character $\bx$
of Lemma \ref{Etretat}(2)
is always~tri\-vial: see Corollary \ref{Etretatcoro}.

\subsection{}
\label{valerie}

As in Paragraph \ref{typeneutre},
the simple character $\t$ and the simple stratum $[\aa,\b]$ 
are both maximal $\s$-self\-dual,
and $\n$ is the Heisenberg representation of $\BJ^1$ associated with $\t$. 
The next proposition,
which says that $\n$ has a canonical extension to $\BJ$, 
is the core of our proof of Theorem \ref{potimarron}.

\begin{prop}
\label{potiron}
There is, up to isomorphism,
a unique representation $\bk$ of $\BJ$ extending $\n$
satis\-fying the following conditions:
\begin{enumerate}
\item 
it is both $\s$-selfdual and $\BJ\cap\G^\s$-distinguished, 
\item 
its determinant has order a power of $p$. 
\end{enumerate} 
This unique representation will be denoted $\tbk$. 
\end{prop}

\begin{rema}
\label{bingo}
This extends (and makes more precise) the results of \cite{VSANT19}
(see \textit{ibid.}, Propositions 7.9, 9.4)
where $\t$ is assumed to be generic
and either $\T/\T_0$ is unramified and $m$ is odd,~or
$\T/\T_0$ is ramified and $m\in\{1,2i\}$.
See also \cite{VSANT19} Remarks 9.5 and 9.9.
\end{rema}

\begin{proof}
Suppose first that there exists a representation satisfying (1).
As in the proof of~\cite{VSPTB} Corollary 6.12, 
one then easily proves the existence of a representation $\bk$
satisfying (1) and~(2).
Let us prove that such a representation is unique.
Any other representation of $\BJ$ satisfying the conditions of the proposition 
is of the form $\bk\boldsymbol{\phi}$ for some character
$\boldsymbol{\phi}$ of $\BJ$ which is
$\s$-selfdual~and trivial on $(\BJ\cap\G^\s)\BJ^1$,
and whose order is a power of $p$.
The restriction of $\boldsymbol{\phi}$ to $\BJ^0$ can be~consi\-dered
as a character of $\GL_m(\ll)$.
Since the latter group is not isomorphic to $\GL_2(\mathbb{F}_2)$
(for $p\neq2$),
this character factors through the determinant.  
Its order is thus prime to $p$,
which implies that $\boldsymbol{\phi}$ is trivial on $\BJ^0$. 
It is thus a character of $\BJ/(\BJ\cap\G^\s)\BJ^0$ which,
by Lemma \ref{YumAubergineindicei}, 
has order at most $2$.
Uniqueness follows from the fact that $p\neq2$.

We are now reduced to proving the existence of 
a representation $\bk$ satis\-fying (1).
If $m=1$,
this follows from \cite{VSANT19} Propositions 7.9, 9.4.
(See also Remark \ref{bingo}.) 

Now consider the constructions of \ref{previousparagraph} and
\ref{typeneutre} with $u=m$.  
Thanks to the case where~$m$ is equal to $1$,
there is a~repre\-sentation $\bk_*$ of $\BJ_*$ extending $\n_*$
which is both $\s$-selfdual and $\BJ^{}_*\cap\G_*^\s$-distinguished.
Let $\bk$ be the representation of $\BJ$ extending $\n$
associated with it by \eqref{kappa0kappa}.
Lemma \ref{preEtretat} implies that
it is $\s$-selfdual, % by Lemma \ref{preEtretat},
and Lemma \ref{Etretat} implies that 
there is a quadratic character $\bx$ of $\BJ$ trivial on $\BJ^0$
such that $\bk\bx$ is $\BJ\cap\G^\s$-distinguished. 
Since $\bx$ is unramified and quadratic,
$\bk\bx$ is also $\s$-selfdual and extends $\n$.  
\end{proof}

\begin{rema}
Notice that this gives another proof of \cite{VSANT19} Propositions 7.9, 9.4, 
based on the case $m=1$ only. 
\end{rema}

Now we can improve Lemma \ref{Etretat}.
Suppose we are in the situation of Paragraphs
\ref{previousparagraph} and \ref{typeneutre}, with $u=m$. 

\begin{coro}
\label{Etretatcoro}
Suppose that $u=m$.  
Let $\bk_*$ be a representation of $\BJ_*$ extending $\n_*$
and $\bk$ correspond to it by the map \eqref{kappa0kappa}.  
If $\bk_*^{}$ is $\BJ_*^{}\cap\G_*^\s$-distin\-guished,
then $\bk$ is $\BJ\cap\G^\s$-distin\-guished.
\end{coro}

\begin{proof}
The result is given by Lemma \ref{Etretat},
except when $\T/\T_0$ is ramified and $m=2i$,~which
we assume now.
Suppose that $\bk_*^{}$ is $\BJ_*^{}\cap\G_*^\s$-distin\-guished. 
By Lemma \ref{Etretat}, 
there is a quadratic character $\bx$ of $\BJ$ trivial on $\BJ^0$ such
that $\bk\bx$ is $\BJ\cap\G^\s$-distin\-guished.
Let $\bk_\t$ be the representation given by Proposition 
\ref{potiron} and write $\bk\bx=\bk_\t\boldsymbol{\phi}$ for some
character $\boldsymbol{\phi}$ of $\BJ$ trivial on $(\BJ\cap\G^\s)\BJ^1$.
Restricting to $\BJ^0$,
the character $\boldsymbol{\phi}$ can be seen as a character of $\GL_m(\ll)$
of the form $\a\circ\det$,
for some character $\a$ of $\ll^\times$,
which is trivial on
$\GL_m(\ll)^\s=\GL_i(\ll)\times\GL_i(\ll)$.
This implies that $\a$ is trivial,
thus $\boldsymbol{\phi}$ is trivial on $\BJ^0$.
Also,
$\boldsymbol{\phi}$ is trivial at $\w'\in\w\BJ^0$
by Lemma \ref{YumAubergineindicei}.
It is thus trivial.
In conclusion, we have $\bk=\bk_\t\bx$.
Taking determinants, we get
\begin{equation}
\label{comp1}
\det\ \bk=\bx \cdot \det\ \bk_\t.
\end{equation}
Now there is a $y\in\M\cap\B^\times$ such that 
$\t'=\t^y\in\Cc(\aa^y,\b)$
is a $\s$-selfdual maximal simple charac\-ter~of index $0$
(in the sense of Remark \ref{indice}). 
By Remark \ref{conjugaisony},
the simple character of $\Cc(\aa_*,\b)$ associated with $\t'$
by Lemma \ref{ArseneLup1} is still $\t_*$,
and the representation of $\BJ'=\BJ^y$
corresponding to $\bk_*$ by \eqref{kappa0kappa} is
$\bk^y$.
Let $\bk_{\t'}$ be the representation associated with $\t'$
by Proposition \ref{potiron}. 
By Lemma \ref{Etretat}, $\bk'$ is distinguished.
By the discussion above, it follows that
\begin{equation}
\label{comp2}
\det\ \bk^y = \det\ \bk_{\t'}.
\end{equation}
But the characters $\det\ \bk$, $\det\ \bk^y$ have the same order
(since they are conjugate to each others),
and the latter one has order a power of $p$ thanks to
\eqref{comp2}.
Now \eqref{comp1} implies that $\bx$ has order a power of $p$.
Since $\bx$ is quadratic and $p\neq2$,
this character is trivial. 
\end{proof}

We extract from the proof of Corollary \ref{Etretatcoro}
the following valuable corollary.

\begin{coro}
\label{Etretatcoro2}
Suppose that $u=m$.  
Let $\bk_{\t_*}$ and $\bk_{\t}$ be the representations 
associated with $\t_*$ and $\t$ by Proposition \ref{potiron},
respectively.
Then the map \eqref{kappa0kappa} takes $\bk_{\t_*}$ to $\bk_{\t}$. 
\end{coro}

We also have the following corollary,
which extends \cite{VSANT19} Lemma 7.10(3), Corollary 9.6(1).

\begin{coro}
Any $\BJ\cap\G^\s$-distinguished representation of $\BJ$
extending $\n$ is $\s$-selfdual.
\end{coro}

\begin{proof}
Let $\bk$ be a $\BJ\cap\G^\s$-distinguished representation of $\BJ$
extending $\n$,
and $\bx$ be the unique character of $\BJ$ trivial on $\BJ^1$ such that
$\bk=\tbk\bx$.
We have to prove that $\bx^{-1}\circ\s=\bx$.
The fact~that $\bk$ is distinguished implies that $\bx$ is trivial on
$(\BJ\cap\G^\s)\BJ^1$.
Restricting to $\BJ^0$,
the character $\bx$ can be seen as a character of $\GL_m(\ll)$ of the
form $\a\circ\det$,
for some character $\a$ of $\ll^\times$.

If $\T/\T_0$ is unramified,
$\a$ is trivial on $\ll_0^\times$,
thus $\bx^{-1}\circ\s$ and $\bx$ coincide on $\BJ^0$.
They also~co\-in\-ci\-de~on~$\w\in\BJ\cap\G^\s$,
thus they are equal.

If $\T/\T_0$ is ramified, $\a$ is trivial,
thus $\bx^{-1}\circ\s$ and $\bx$ are both trivial on $\BJ^0$.
Since 
$\bx$ is trivial on $\w^2\in\BJ\cap\G^\s$,
we get
$\bx^{-1}\circ\s(\w)=\bx^{-1}(-\w)=\bx^{-1}(\w)=\bx(\w)$,
which finishes the proof.  
\end{proof}

\subsection{}
\label{groseille}
\label{peche}

We now come to the type theoretic description of cuspidal representations of 
$\G$.
The~follow\-ing proposition follows from
\cite{BK} Theorem 8.4.1, Corollary 6.2.3, Theorem 5.7.1
(see \cite{MSt}~Théo\-rèmes 3.4, 3.7 and \cite{SS6} Theorem 7.2
in the modular case).

\begin{prop}
Let $\pi$ be a cuspidal representation of $\G$.
There is, up to $\G$-conjugacy,~a
unique simple character $\t$ such that the restriction of~$\pi$ to $\H^1_\t$ 
contains $\t$, and it is maximal. 
\end{prop}

Let $\pi$ be a cuspidal representation of $\G$,
and let $\t$ be a simple character occurring in $\pi$.~As\-so\-ciated with it, 
there are: 
\begin{itemize}
\item 
the positive integer $m(\pi)=m\>1$ defined by \eqref{melon},
called the \textit{relative degree} of $\pi$,
\item
the $\G$-conjugacy class (or equivalently the $\F$-iso\-mor\-phism class)
of the tamely ramified~ex\-ten\-sion $\T$ of $\EE$ associated with $\t$,
called the \textit{tame parameter field} of $\pi$,
\item
the endo-class $\TT$ of $\t$,
called the \textit{endo-class} of $\pi$.
\end{itemize}
(Note that,
when $\pi$ has level $0$,
one has $m=n$ and $\T=\F$,
and $\TT$ is the null endo-class.) 

Write $\BJ=\BJ_\t^{}$, $\BJ^0=\BJ^0_\t$, $\BJ^1=\BJ^1_\t$ 
and let $\n$ be the Heisenberg representation of $\t$.
The~next pro\-position follows from \cite{MSt} Lemme 5.3, Theorem 3.11. 

\begin{prop} 
\label{deftau}
Let $\bk$ be a representation of $\BJ$ extending $\n$,
and define a representation of $\BJ$ on the space
$\Hom_{\BJ^1}(\bk,\pi)$ by making $x\in\BJ$
act on $f\in\Hom_{\BJ^1}(\bk,\pi)$ by
\begin{equation*}
\label{defKappa}
x \cdot f = \pi(x) \circ f \circ \bk(x)^{-1}.
\end{equation*}
This representation, denoted $\bt$, has the following properties:
\begin{enumerate}
\item 
It is irreducible, and trivial on $\BJ^1$. 
\item
If one identifies $\BJ^0/\BJ^1$ with a finite general linear group as in
\eqref{JJ1UU1GLmax}, 
its restriction to $\BJ^0$~is the inflation of a~cuspi\-dal~re\-presentation. 
\item
The compact induction of $\bk\otimes\bt$ from $\BJ$
to $\G$ is isomorphic to $\pi$.
\end{enumerate}
\end{prop}

Any two representations of $\BJ$ extending $\n$
differ from a character of $\BJ$ trivial on $\BJ^1$.
The pair 
\begin{equation}
\label{deftypekt}
(\BJ,\bk\otimes\bt)
\end{equation}
thus only depends on $\pi$ 
and the choice of $\t$, 
and not on the choice of $\bk$.

When $\pi$ varies among all cuspidal representations of $\G$ 
and $\t$ varies among all maximal simple charac\-ters in $\pi$, 
the pairs \eqref{deftypekt}
are called \textit{extended maximal simple types} in \cite{BK,MSt},
which we will abbreviate to \textit{types} here.
A given cuspidal~re\-presentation of $\G$ thus contains,
up to $\G$-conju\-gacy, a unique type $(\BJ,\bl)$:
there is a unique maximal simple character $\t$ such that $\BJ_\t=\BJ$ 
and $\t$ occurs in the restriction of $\bl$ to $\H^1_\t$,
a representation $\bk$ of $\BJ$ which res\-tricts irreducibly to $\BJ^1$
and a representation~$\bt$ of $\BJ$ trivial on $\BJ^1$ such that 
$\bl$ is isomorphic to $\bk\otimes\bt$.

\begin{rema}
\label{reftypelevel0}
If $\aa$ is a maximal order in $\Mat_n(\F)$,
the trivial character of $\U^1_\aa$ is a maximal simple character, 
with $\E=\T=\F$ and $m=n$.
The cuspidal representations of $\G$ that contain such a simple character 
are precisely the cus\-pidal~re\-pre\-sentations of level~$0$.  
\end{rema}

Fix a representation $\bk$ of $\BJ$ extending $\n$
and define $\bt$ as in Proposition \ref{deftau},
and fix a~simple stratum $[\aa,\b]$ such that $\t\in\Cc(\aa,\b)$
and an isomorphism \eqref{melon}.
This gives a field $\E$ and~an~iso\-mor\-phism
$\BJ^0/\BJ^1\simeq\GL_m(\ll)$,
where $\ll$ is the re\-si\-due field of $\T$.

By Proposition \ref{deftau}(2), 
the restriction of $\bt$ to $\BJ^0$ is~the inflation of a cuspi\-dal 
irreducible~re\-presentation, denoted $\Vv$.

On the other hand, the representation $\bt$ has a central character~:
it is a character~of~the~cen\-tre~$\E^\times \BJ^1/\BJ^1$ of $\BJ/\BJ^1$,
or~equi\-valently a tamely ramified character of $\E^\times$.
Since $\E$~is~totally~wild\-ly~ramified over its
maximal tamely ramified subextension $\T$,
any tamely ramified character~of $\E^\times$
factors through the norm $\Nm_{\E/\T}$.
The restriction of $\bt$ to $\E^\times$
is thus a multiple of $\ic\circ\Nm_{\E/\T}$
for a uniquely determined
tamely ramified character 
$\ic$ of $\T^\times$.

The data $\Vv$ and $\ic$ are subject to the compatibility condition that
the restriction of $\Vv$ to $\ll^\times$~is a~multiple
of the character of $\ll^\times$ whose inflation to $\oo_\T^\times$ is
the restriction of $\ic^{p^e}$, with~$p^e=[\E:\T]$.
Associated with $\Vv$ by Proposition \ref{Coupurefinie},
there are a unique integer $u\>1$ dividing~$m$
and a~uni\-que supercuspidal~represen\-ta\-tion $\varrho$ of $\GL_{m/u}(\ll)$ 
such that $\Vv$ is isomorphic to $\st_u(\varrho)$.
The~next~impor\-tant result is \cite{MSjl} Lemma~3.2.
The integer $r(\pi)$ has been defined in Paragraph~\ref{par92}.

\begin{lemm}
\label{poulet}
The integer $u$ is equal to $r(\pi)$.
\end{lemm}

It follows that $r(\pi)$ divides $m$,
and that $\pi$ is supercuspidal if and only if $\Vv$ is supercus\-pidal.

\subsection{}
\label{compatypes}

Write $r=r(\pi)$ and $k=n/r$,
and let $\rho$ be a~super\-cuspidal represen\-tation of $\GL_{k}(\EE)$ 
such
that $\pi$ is isomorphic to $\St_r(\rho)$ given by
Proposition \ref{Coupure}.
In this paragraph,
we will compare the type theoretic description of~$\pi$ with
that of $\rho$.
As in \ref{groseille},
we fix~a~re\-pre\-sentation $\bk$ of $\BJ$ extending $\n$.
It defines an~irreducible~re\-presentation $\bt$ of $\BJ$
trivial on $\BJ^1$,
then a cuspidal representation $\Vv$ of $\GL_m(\ll)$ and 
a tamely~rami\-fied character $\ic$ of $\T^\times$.
There is also a (unique) supercuspidal~repre\-sen\-ta\-tion $\varrho$ of
$\GL_{m/r}(\ll)$ such that $\Vv$ is isomorphic to $\st_r(\varrho)$.

Since $r$ divides $m$, 
we may apply the results of \ref{previousparagraph} to the case where $u=r$, 
which~we assume now.
Let $\t_*$ be the simple character associated with $\t$ by Lemma 
\ref{ArseneLup1}. 

\begin{lemm}
\label{step1}
The representation $\rho$ contains $\t_*$.
\end{lemm}

\begin{proof}
This follows from the
descrip\-tion\footnote{Warning: the representation denoted $\St(\rho,r)$
in \cite{MSc} corresponds to $\St_r(\rho\nu^{(r-1)/2})$,
and the one denoted $\St_r(\rho)$ in \cite{MSc} corresponds to 
$\St_{v}(\rho\nu^{(v-1)/2})$ with $v=e(\rho)\ell^r$.}
of $\St_r(\rho)$ in \cite{MSc} Section 6.
\end{proof}

Consequently,
$\pi$ and $\rho$ have the same endo-class. 
We have the following immediate corollary.

\begin{coro}
\label{proofRing} 
We have $m(\pi)=m(\rho)r$
and the representations $\pi$, $\rho$ have the same tame parameter field.  
\end{coro}

Let $\n_*$ be the Heisenberg representation associated with $\t_*$, 
and let $\bk_*$ be a representation of $\BJ_*$ extending $\n_*$
such that the representation of $\BJ$~asso\-ciated with it by 
\eqref{kappa0kappa} is~$\bk$.
It defines~an irreducible representation $\bt_*^{}$ of $\BJ^{}_*$
trivial on $\BJ^1_*$,
such that the pair $(\BJ_*,\bk_*\otimes\bt_*)$ is a
type in~$\rho$.
As\-so\-ciated with this, 
there are a cuspidal representation $\varrho_*$ of $\GL_{m/r}(\ll)$ 
(which is supercuspidal thanks to the comment after Lemma \ref{poulet}) and 
a tamely ramified character $\ic_*$ of $\T^\times$.
The~fol\-lowing proposition compares the pairs $(\varrho,\ic)$ and
$(\varrho_*,\ic_*)$ associated with $\bt$ and $\bt_*$.

\begin{prop}
\label{step2}
We have $\varrho\simeq\varrho_*$ and $\ic=\ic_*^r$.  
\end{prop}

\begin{proof}
Again, the fact that $\varrho$ is isomorphic to $\varrho_*$ follows from the
descrip\-tion of $\St_r(\rho)$ in~\cite{MSc} Section 6. 
It thus remains to prove the second equality.
For this, consider the action of $\BJ$ on 
\begin{equation}
\label{patati}
\Hom_{\BJ^1}(\bk,\Ii(\rho,r))
\end{equation}
where $\Ii(\rho,r)$ is the parabolically induced representation 
\eqref{INDRHO}.
By \cite{SSblocks} Proposition 5.6,
its~res\-tric\-tion to $\BJ^0$ is the inflation of
the induced representation
$\varrho_*\tdt\varrho_*$ of $\GL_m(\ll)$.
By tracking the action of $\E^\times$ in 
the arguments of \cite{SSblocks} Section 5,
we see that it acts on the space \eqref{patati}
by the character
\begin{equation*}
\ic_*^r\circ\Nm_{\E/\T}.
\end{equation*}
In particular, $\E^\times$ acts through this character
on the subquotient
$\Hom_{\BJ^1}(\bk,\pi)$,
which is $\bt$.
\end{proof}

\subsection{}
\label{lwcan}

Suppose that the cuspidal representation $\pi$ is $\s$-selfdual.
We say a type $(\BJ,\bl)$ is $\s$-\textit{self\-dual} if
$\BJ$ is $\s$-stable and $\bl^{\vee\s}$ is isomorphic to $\bl$.
The next result is \cite{AKMSS} Theorem 4.1.

\begin{prop}
\label{pechessdpi} 
The representation $\pi$ contains a $\s$-selfdual type. 
\end{prop}

A type $(\BJ,\bl)$ contains a unique simple character $\t$ such that 
$\BJ_\t=\BJ$: 
it follows that,
if $(\BJ,\bl)$ is $\s$-selfdual, 
$\t$ is $\s$-selfdual as well. 
In particular, $\pi$ contains a $\s$-selfdual simple character. 

Let $\t$ be a $\s$-selfdual simple character occurring in $\pi$,
and $[\aa,\b]$ be a $\s$-selfdual simple stratum such that
$\t\in\Cc(\aa,\b)$
(which exists by Proposition \ref{pechessd}).
The $\G^\s$-conjugacy class (or equivalently the $\F_0$-iso\-mor\-phism class) 
of the tamely ramified~ex\-ten\-sion $\T_0$ of $\E_0$ associated with $\t$ 
only~de\-pends on $\pi$.
Associated with $\pi$, there is thus a quadratic extension $\T/\T_0$.

\begin{rema}
When $\pi$ has level $0$,
one has $\T_0=\F_0$.
\end{rema}

If follows from Proposition \ref{poussin} that $\pi$ contains
\begin{itemize}
\item 
only one $\G^\s$-conjugacy class of $\s$-selfdual types
if $\T/\T_0$ is unramified,
\item
$\lfloor m/2\rfloor+1$ different $\G^\s$-conjugacy classes 
of $\s$-selfdual types if $\T/\T_0$ is ramified.
\end{itemize}
Among these $\G^\s$-conjugacy classes of $\s$-selfdual types,
one is of par\-ticular im\-portance. 

\begin{defi}\label{genericssdualtype}
A $\s$-selfdual type $(\BJ,\bl)$ is said to be \textit{generic} if 
either $\T/\T_0$ is unramified,
or $\T/\T_0$ is ramified and the integer $i$ of
Propo\-si\-tion \ref{pechessd}(2.b)
is equal to $\lfloor m/2\rfloor$. 
\end{defi}

A $\s$-selfdual cuspidal representation of $\G$ thus contains,
up to $\G^\s$-conjugacy,
a unique generic $\s$-selfdual type.
The next result is \cite{VSANT19} Theorem 10.3
(see also \cite{AKMSS} Section 6).

\begin{prop}
\label{pulledpork6heures}
Let $\pi$ be a $\s$-selfdual cuspidal representation of $\G$
and $(\BJ,\bl)$ be its~gene\-ric $\s$-selfdual type.
Then $\pi$ is distinguished 
if and only if $\bl$ is $\BJ\cap\G^\s$-distinguished.
\end{prop}

If $(\BJ,\bl)$ is a $\s$-selfdual type,
and if $\t$ is the unique simple
character contained in $\bl$ such~that $\BJ_\t=\BJ$,
we will write $\blw$ for the unique representation $\tbk$ of $\BJ$ extending 
the Heisenberg~re\-pre\-sentation~of~$\t$ given by Proposition \ref{potiron}.
The next result extends \cite{VSANT19} Propositions 7.9, 9.8
to the case of~ar\-bi\-trary $\s$-selfdual cuspidal representations.

\begin{prop}
\label{resB}
Let $\pi$ be a $\s$-selfdual cuspidal representation of $\G$.
Let $(\BJ,\bl)$~be~a~ge\-ne\-ric $\s$-selfdual type in $\pi$
and $\bt$ be the representation of $\BJ$ trivial on $\BJ^1$ such 
that $\bl$~is isomorphic to $\blw\otimes\bt$. 
Then $\pi$ is dis\-tin\-guished if and only if $\bt$ is
$\BJ\cap\G^\s$-distinguished. 
\end{prop}

\begin{proof}
This follows from Proposition \ref{pulledpork6heures}
together with the fact that
\begin{equation*}
\Hom_{\BJ\cap\G^\s}(\bl,\R) \simeq
\Hom_{\BJ\cap\G^\s}(\blw,\R) \otimes\Hom_{\BJ\cap\G^\s}(\bt,\R)
\end{equation*}
and $\Hom_{\BJ\cap\G^\s}(\blw,\R)$ has dimension $1$
(see Proposition \ref{etachi}(4)).
\end{proof}

Fix isomorphisms
\begin{equation}
\label{choixisoB}
\B\simeq\Mat_m(\E),
\quad
\BJ^0/\BJ^1\simeq\GL_m(\ll), 
\end{equation} 
as in Proposition \ref{pechessd}(4). 
The restriction of $\bt$ to $\BJ\cap\B^\times$ is a generic $\s$-selfdual
type of level $0$ in~$\B^\times\simeq\GL_m(\E)$
and $\BJ/\BJ^1$ is naturally isomorphic to
$(\BJ\cap\B^\times) / (\BJ^1\cap\B^\times)$.
The representation $\bt$ is thus distinguished by $\BJ\cap\G^\s$ if and only if
its restriction to $\BJ\cap\GL_m(\E)$ is
distin\-gui\-shed by~$\BJ\cap\GL_m(\E_0)$.
Pro\-position \ref{resB} used twice thus implies
that $\pi$ is distinguished by $\G^\s$
if and only if the cuspidal~repre\-sen\-ta\-tion
of level $0$ of $\GL_m(\E)$ 
compactly induced from~the~restriction
of~$\bt$ to $\BJ\cap\GL_m(\E)$ is
dis\-tin\-guished by $\GL_m(\E_0)$.

However, the field extension $\E$ is not~cano\-nical.
In~\ref{proofpotimarron1},
we will canonically associate with~$\pi$ a $\s$-selfdual~cus\-pi\-dal
representation $\pit$~of level~$0$ of $\GL_m(\T)$,
which is $\GL_m(\T_0)$-dis\-tin\-guished~if and only if $\pi$
is $\G^\s$-dis\-tinguished,
where $\T/\T_0$ is the quadratic extension associated with $\pi$.
Our strategy is inspired from \cite{BHEffective} Section 3.

The following proposition
relates the parity of $m/r$ to the ramification of $\T/\T_0$.

\begin{prop}
\label{fraise}
Let $\pi$ be a $\s$-selfdual cuspidal representation of $\GL_n(\EE)$
with quadratic~ex\-tension $\T/\T_0$, 
and write $m=m(\pi)$, $r=r(\pi)$.
Then
\begin{equation*}
\text{$m/r$ is }
\left\{ 
\begin{array}{l}
\text{odd if $\T/\T_0$ is unramified}, \\ 
\text{either even or equal to $1$ if $\T/\T_0$ is ramified}.
\end{array}\right.
\end{equation*} 
\end{prop}

\begin{proof}
Write $\pi$ as $\St_r(\rho)$ as in Proposition \ref{fission} with
$\rho^{\s\vee}\simeq\rho\nu^i$ for some $i\in\{0,1\}$.
Then~$\rho\nu^{i/2}$ is a $\s$-selfdual supercuspidal representation of
$\GL_{n/r}(\EE)$,
and the quadratic~ex\-tension associated with it is $\T/\T_0$.
Applying \cite{VSANT19} Propositions 8.1, 9.8,~we
get the expected result. 
\end{proof}

\subsection{}
\label{proofpotimarron1}
\label{par1092}

In order to prove Theorem \ref{potimarron},
it will be useful to consider the slightly more general~si\-tua\-tion where
$\pi$ is a cuspidal representation of $\G$ with $\s$-selfdual endo-class 
$\TT$.
Thus $\pi$ itself needs~not be $\s$-selfdual. 
However,
it has a relative degree $m$ and,
since $\TT$ is $\s$-selfdual,
there is a quadratic ex\-tension $\T/\T_0$ associated with it.
Moreover,
by Proposition \ref{poussin},
it contains,
up to $\G^\s$-conjugacy,
a unique generic $\s$-selfdual maximal simple character $\t$.
Let $\BJ$ be its normalizer in $\G$~and $\tbk$~be the~re\-presentation
of $\BJ$ given by
Proposition \ref{potiron}.
Then $\pi$ contains a unique type of the form
\begin{equation}
\label{hergemont}
(\BJ,\tbk\otimes\bt)
\end{equation}
for a uniquely determined irreducible representation $\bt$ of
$\BJ$ trivial on $\BJ^1$. 
Fix a $\s$-selfdual simple stratum $[\aa,\b]$ 
and isomorphisms \eqref{choixisoB} as in Proposition \ref{pechessd}.

First,
we define an open and compact mod centre subgroup $\BJT$ of
$\GL_m(\T)$ as follows:
\begin{itemize}
\item 
if $\T/\T_0$ is unramified,
$\BJT$ is the normalizer of $\GL_m(\oo_\T)$ in $\GL_m(\T)$,
\item 
if $\T/\T_0$ is ramified, 
and if $t$ is a uniformizer of $\T$ such that $\s(t)=-t$, 
then~$\BJT$ is the~nor\-mali\-zer in $\GL_m(\T)$ of the conjugate of
$\GL_m(\oo_\T)$ by the diagonal element
\begin{equation*}
{\rm diag}(t,\dots,t,1,\dots,1) \in \GL_m(\T)
\end{equation*}
where $t$ occurs $\lfloor m/2\rfloor$ times.
\end{itemize}
The group $\BJT$
(which does not depend on the choice of $t$ in the
ramified case)
has a~uni\-que~maxi\-mal com\-pact subgroup $\BJT^0$
and a unique normal maximal pro-$p$-subgroup $\BJT^1$.
The~natural~group iso\-morphism
\begin{equation}
\label{bjtGL}
\BJT^0/\BJT^1\simeq\GL_m(\ll)
\end{equation}
transports the action of $\s\in\Gal(\T/\T_0)$ on $\BJT^0/\BJT^1$ to
\begin{itemize}
\item 
the action of the non-trivial element of $\Gal(\ll/\ll_0)$
on $\GL_m(\ll)$ 
if $\T/\T_0$ is unramified,
\item 
the adjoint action of
\begin{equation}
\label{banane0bis}
\begin{pmatrix}
-{\rm id}_{\lfloor m/2\rfloor} & 0 \\ 0 & {\rm id}_{m-\lfloor m/2\rfloor}
\end{pmatrix} \in \GL_m(\ll),
\end{equation}
on $\GL_m(\ll)$ if $\T/\T_0$ is ramified.
\end{itemize}

\begin{rema}
\label{onchanget}
When $\T/\T_0$ is ramified,
the isomorphism \eqref{bjtGL} depends on the choice of~$t$~:
changing $t$ to another uniformizer $t'$ conjugates the isomorphism
by the $\s$-invariant element
\begin{equation*}
{\rm diag}(\a,\dots,\a,1,\dots,1) \in \GL_m(\ll)
\end{equation*}
where $\a$ (which occurs $i$ times)
is the image of $t't^{-1}$ in $\ll^\times$. 
This element is central in $\GL_m(\ll)^\s$.
\end{rema}

We now associate to $\bt$ an irreducible representation
$\btt^{}$ of $\BJT^{}$ trivial on $\BJT^1$.
On the one hand,
the restriction of $\bt$ to $\BJ^0$ is the inflation of an irreducible cuspidal
representation $\Vv$ of $\GL_m(\ll)$. 
On the other hand, the restriction of $\bt$ to $\E^\times$
is a multiple of $\ic\circ\Nm_{\E/\T}$ for a
uniquely~determi\-ned~tamely ramified character $\ic$ of $\T^\times$:
see \ref{groseille}.
Note that $[\E:\T]=p^e$ for some $e\>1$.

\begin{lemm}
\label{pinson}
Let $\Vv$ and $\ic$ be as above. 
\begin{enumerate}
\item
There is a unique representation $\btt^{}$ of $\BJT^{}$ trivial on 
$\BJT^1$ such that
\begin{enumerate}
\item
the restriction of $\btt^{}$ to $\T^\times$ is a multiple of the character $\ic$,
\item 
the restriction of $\btt^{}$ to $\BJT^0$ is the inflation of $\Vv^{(p^{-e})}$, 
where $\Vv^{(p^{-e})}$ is the representation of $\GL_m(\ll)$ obtained from 
$\Vv$ by applying the automorphism
$x\mapsto x^{p^{-e}}$.
\end{enumerate}
\item 
The pair $(\BJT,\btt)$ is a level $0$ type in $ \GL_m(\T)$.
\item
Up to isomorphism, 
the representation $\btt$ only depends on $\bt$,
and not on the choice of~the $\s$-selfdual simple stratum $[\aa,\b]$,
the uniformizer $t$ and
the identification $\BJ^0/\BJ^1\simeq\GL_m(\ll)$. 
\end{enumerate}
\end{lemm}

\begin{proof}
Uniqueness follows from the fact that $\BJT^{}$ is generated by $\BJT^0$
and $\T^\times$,
and the existence of $\btt^{}$ follows from the fact that 
the restriction of $\Vv^{(p^{-e})}$ to $\ll^\times$ is a multiple
of the character of~$\ll^\times$ defined by 
the restriction of $\ic$ to the units of $\T^\times$.
Since $\Vv^{(p^{-e})}$ is cuspidal,
the pair $(\BJT,\btt)$ is a level $0$ type by construction.
It remains to prove (3).
Since it will require techniques which are not used anywhere else in the 
paper, we will prove it apart, in Paragraph \ref{grange}.
\end{proof}

It will be convenient to give another description of the representation $\btt$. 

\begin{lemm} 
\label{defNnlemma}
\begin{enumerate}
\item 
There is a unique group isomorphism
$\NF:\BJ/\BJ^1 \to \BJT^{}/\BJT^1$ such that 
\begin{enumerate}
\item 
its restriction to $\GL_m(\ll)$ is the automorphism acting entrywise by 
$\phi:x\mapsto x^{p^{e}}$,
\item
for all $x\in\E^\times$,
the image of $x\BJ^1$ is $\Nm_{\E/\T}(x)\BJT^1$.
\end{enumerate} 
\item
The isomorphism $\NF$ is $\s$-equivariant. 
\item
The representation $\btt$ is isomorphic to $\bt\circ\NF^{-1}$.
\end{enumerate}
\end{lemm}

\begin{proof}
Again, uniqueness follows from the the fact that $\BJ$ is generated by $\BJ^0$ 
and $\E^\times$.
Exis\-tence follows from the fact that
$\Nm_{\E/\T}(x)=x^{p^{e}}$
for all $x\in\oo_{\E}^\times$ of order prime to $p$,
and that $\Nm_{\E/\T}$ induces a group isomorphism from 
$\E^\times/(1+\pp_\E)$ to $\T^\times/(1+\pp_\T)$.

Define $\NF_1$ to be $\s\circ\NF\circ\s^{-1}$.
The restriction of $\NF_1$ to~$\E^\times$ corresponds to 
$\s\circ\Nm_{\E/\T}\circ\s^{-1}$,
which is equal to $\Nm_{\E/\T}$ since $\E$ and $\T$ are stable by $\s$.
The restriction of $\NF_1$ to $\GL_m(\ll)$ is
\begin{itemize}
\item
the automorphism defined
by making $\s\circ\phi\circ\s^{-1} = \phi\in\Gal(\ll/\mathbb{F}_p)$
act entrywise if $\T/\T_0$ is unramified,
\item
the automorphism ${\rm Ad}(\d^{-1} \phi(\d))\circ\phi=\phi$ if $\T/\T_0$ is 
ramified, 
where $\d$ is the $\phi$-invariant~ma\-trix~defined by \eqref{banane}.
\end{itemize}
The fact that $\NF$ is $\s$-equivariant now 
follows from its uniqueness,
and (3) is immediate. 
\end{proof}

Now let us describe the behavior of $\bt\mapsto\btt$
with respect to duality and distinction.

\begin{lemm}
\label{pinsondeux}
\begin{enumerate} 
\item 
The representation $\btt$ is $\s$-selfdual if and only if 
$\bt$ is $\s$-selfdual.
\item 
The representation $\btt$ is $\BJT\cap\GL_m(\T_0)$-distinguished
if and only if $\bt$ is $\BJ\cap\G^\s$-dis\-tin\-gui\-shed.
\end{enumerate}
\end{lemm}

\begin{proof} 
Saying that $\bt$ is $\s$-selfdual
is equivalent to saying that the representation 
$\Vv$ and~the cha\-racter~$\ic\circ\Nm_{\E/\T}$ are $\s$-selfdual.
Assertion (1) follows from the fact that 
$(\Vv^{(p^{-e})})^{\vee\s}$
is isomor\-phic~to $(\Vv^{\vee\s})^{(p^{-e})}$
and $(\ic\circ\Nm_{\E/\T})^{-1}\circ\s$ 
is equal to $(\ic^{-1}\circ\s)\circ\Nm_{\E/\T}$.

Assertion (2) follows from the fact that $\btt\circ\NF=\bt$
and $\NF$ maps $(\BJ/\BJ^1)^\s$ to $(\BJT^{}/\BJT^1)^\s$. 
\end{proof}

\begin{coro}
\label{pinsontrois}
The pair $(\BJT,\btt)$ is a generic $\s$-selfdual type 
if and only if $(\BJ,\tbk\otimes\bt)$ is a generic $\s$-self\-dual~type. 
\end{coro}

\begin{proof}
This follows from Lemma \ref{pinsondeux}(1), 
thanks to our choice of $\BJT$~(see \eqref{banane0bis}).
\end{proof}

\subsection{}

We still are in the situation of Paragraph \ref{proofpotimarron1}.
Consider the compactly induced represen\-tation
\begin{equation}
\label{defpit}
\pit = \ind^{\GL_m(\T)}_{\BJT} (\btt).
\end{equation}
It satisfies the following properties.

\begin{prop}
\label{proppit}
\begin{enumerate}
\item
The representation $\pit$ 
is cuspidal, irreducible and has level $0$. 
\item
One has $m(\pit)=m$ and $r(\pit)=r$.
\item
The representation $\pit$ 
is $\s$-selfdual if and only if $\pi$ is $\s$-selfdual.
\item
The representation $\pit$ is $\GL_m(\T_0)$-distinguished if and only $\pi$ is 
$\GL_n(\FF)$-distinguished.
\end{enumerate}
\end{prop}

\begin{proof}
Assertion (1) follows from the fact that $\pit$ is compactly induced from a 
level $0$ type~in $\GL_m(\T)$
(see Lemma \ref{pinson} and Remark \ref{reftypelevel0}).
The first equality of Assertion (2) follows from Remark \ref{reftypelevel0},
and the second one from Lemma \ref{poulet}.

Suppose that $\pi$ is $\s$-selfdual.
Then $\bt$ is $\s$-selfdual (see \ref{lwcan}).
By Lemma \ref{pinsondeux},
the representa\-tion $\btt$ is $\s$-selfdual as well.
By compact induction, it follows that $\pit$ is $\s$-selfdual.
The~argu\-ment also works the other way round,
proving (3).
Assertion (4) follows from Proposition \ref{resB} to\-ge\-ther with
Lemma \ref{pinsondeux}(2).
\end{proof}

\begin{theo}
\label{potimarrong}
\label{potimarron}
\begin{enumerate}
\item 
The process
\begin{equation}
\label{bijpit}
\pi\mapsto\pit
\end{equation}
induces a bijection from the set of isomorphism classes of cuspidal
representations of~$\G$~with~endo-class $\TT$ to that of 
cus\-pi\-dal re\-presentations of level~$0$ of $\GL_m(\T)$.
\item
The bijection \eqref{bijpit} maps $\s$-selfdual representations onto
$\s$-selfdual representations~and~$\G^\s$-dis\-tinguished
representations onto 
$\GL_m(\T_0)$-distinguished representations.
\item 
For any cuspidal representation $\pi$ with~endo-class $\TT$
and any tamely ramified character $\chi$ of $\F^\times$,
the~re\-presentation $(\pi\chi)_{\rm t}$ is isomorphic to
$\pit(\chi\circ\Nm_{\T/\F})$.
\end{enumerate}
\end{theo}

\begin{proof}
For (1),
let $\pi_0$ be a cus\-pi\-dal re\-presentation of level~$0$ of $\GL_m(\T)$.
It contains a level $0$ type $(\BJT,\bt_0)$ for a~uni\-que\-ly determined
representation $\bt_0$ of $\BJT$ trivial on $\BJT^1$.
It then suffices to check that the process 
\begin{equation*}
\pi_0 \mapsto \ind^\G_{\BJ} (\tbk\otimes(\bt_0\circ\NF))
\end{equation*}
gives the inverse bijection.
For (3),
notice that if $\pi$ contains the type $(\BJ,\tbk\otimes\bt)$,
then $\pi\chi$~con\-tains the type
$(\BJ,\tbk\otimes\bt\chi^{\BJ})$,
where $\chi^{\BJ}$ is the unique character of $\BJ$ trivial on $\BJ^1$ 
whose restriction to $\BJ\cap\B^\times\simeq\GL_m(\E)$ is equal to
$(\chi\circ\Nm_{\E/\F})\circ\det_\E$
where $\det_\E$ is the determinant on $\B\simeq\Mat_m(\E)$.
Then $(\bt\chi^{\BJ})_{\rm t}$ is isomorphic to the representation
$\btt$ twisted by the character of $\BJT$ trivial on $\BJT^1$
given by $(\chi\circ\Nm_{\T/\F})\circ\det_\T$ , 
where $\det_\T$ is the determinant on $\Mat_m(\T)$.
Assertion (2) is given by Proposition \ref{proppit}.
\end{proof}

\begin{coro}
\label{potimarroncor}
Let $\mu$ be a tamely ramified character of $F_0^\times$. 
A cuspidal representation $\pi$ of $\GL_n(F)$ with endo-class $\TT$
is distinguished by $\mu$ if and only if $\pit$ is distinguished by
$\mu\circ{\rm N}_{T_0/F_0}$.
\end{coro}

\begin{proof}
Fix a tamely ramified character $\xi$ of $F^\times$ extending $\mu$.
Then $\pi$ is $\mu$-distinguished~if~and only if $\pi\xi^{-1}$ is 
distinguished, 
and $(\pi\xi^{-1})_{\rm t}$ is isomorphic to
$\pit(\xi^{-1}\circ\Nm_{\T/\F})$.
Thus $\pi$ is $\mu$-dis\-tingui\-shed if and only if $\pit$ is distinguished by
the character
$\xi\circ\Nm_{\T/\F}|_{T_0^\times}=\mu\circ{\rm N}_{T_0/F_0}$.
\end{proof}

Finally,
let us describe the compatibility of the process \eqref{bijpit}
with the des\-crip\-tion of cuspidal representations in terms of supercuspidal 
ones of \ref{compatypes}.

\begin{prop}
\label{compapitrho}
Let $\pi$ be a cuspidal representation of~$\G$~with~endo-class $\TT$
and $r=r(\pi)$.
Let $\rho$ be a supercuspidal representation of $\GL_{n/r}(\F)$
such that $\pi$ is isomorphic to $\St_r(\rho)$.
Then~$\pit$ is isomorphic to $\St_r(\rhot)$.  
\end{prop}

\begin{rema}
Note that this makes sense since, by Corollary \ref{proofRing}, 
the representations~$\pi$, $\rho$ have the same endo-class $\TT$,
thus the same
% \rob{tame parameter field~$\T$.}%
quadratic extension $\T/\T_0$. 
\end{rema}

\begin{proof}
The re\-pre\-sentation $\pi$ contains a type of the form
$(\BJ,\tbk\otimes\bt)$ for a unique representa\-tion $\bt$ of
$\BJ$ trivial on $\BJ^1$.  
Fix a $\s$-selfdual simple stratum $[\aa,\b]$~and
isomorphisms \eqref{choixisoB} as in Proposition~\ref{pechessd}.
Associated with $\bt$, there are
a tamely~ra\-mified character $\ic$ of $\T$, and 
a cuspidal representa\-tion
$\Vv=\st_r(\varrho)$ of $\GL_{m}(\ll)$,
for some supercuspidal representation $\varrho$ of $\GL_{m/r}(\ll)$. 
The re\-pre\-sentation $\bt$ is entirely determined by the fact that
\begin{itemize}
\item 
its restriction to $\BJ^0$ is the inflation of $\Vv$,
\item 
its restriction to $\E^\times$ is a multiple of the character 
$\ic\circ\Nm_{\E/\T}$. 
\end{itemize}
We now use the results of \ref{previousparagraph} and \ref{typeneutre}
for $u=r$.
Let $\t_*$ be the simple character associated~with $\t$ by Lemma 
\ref{ArseneLup1}. 
Thanks to Corollary \ref{Etretatcoro2},
the~re\-presentation $\bk_{\t_*}$ corresponds to $\tbk$ via the map
\eqref{kappa0kappa}. 
Paragraph \ref{compatypes} says that $\rho$ contains the
type $(\BJ_*, \bk_{\t_*}\otimes\bt_*)$,
where $\bt_*$ is the~re\-pre\-sentation of $\BJ_*$ trivial on $\BJ^1_*$
determined by 
\begin{itemize}
\item 
its restriction to $\BJ_*^0$ is the inflation of $\varrho$,
\item 
its restriction to $\E^\times$ is a multiple of 
$\ic_*\circ\Nm_{\E/\T}$, 
where $\ic_*$ is a tamely ramified character of $\T^\times$ 
such that $\ic_*^r=\ic$. 
\end{itemize} 
Thus $\rhot$ is compactly induced from the level $0$ type
$(\BJ_{*,{\rm t}},\bt_{*,{\rm t}})$ where $\bt_{*,{\rm t}}$ is determined by 
\begin{itemize}
\item 
its restriction to $\BJ_{*,{\rm t}}^0$ is the inflation of $\varrho^{(p^{-e})}$,
\item 
its restriction to $\T^\times$ is a multiple of $\ic_*$.
\end{itemize} 
Thus $\St_r(\rhot)$ is compactly induced from
the level $0$ type $(\BJT,\bd)$ where $\bd$ is determined by 
\begin{itemize}
\item 
its restriction to $\BJT^0$ is the inflation of
$\st_r(\varrho^{(p^{-e})})\simeq\Vv^{(p^{-e})}$,
\item 
its restriction to $\T^\times$ is a multiple of $\ic_*^r=\ic$.
\end{itemize}
It follows that $\bd$ is isomorphic to $\btt$, 
whence $\St_r(\rhot)$ is isomorphic to $\pit$.
\end{proof}

\subsection{}

Finally,
let $\pi$ be a $\s$-selfdual cuspidal representation~of $\G$, of level $0$.
It has a central~cha\-racter $c_\pi$ and its generic type
$(\BJ,\bl)$ defines a cuspi\-dal re\-presentation $\Vv$ of $\GL_n(\kk)$.
Assume~that $n\neq1$.
In the spirit of Proposition \ref{resB},
we give a necessary and sufficient condition for $\pi$ to~be
distin\-guished by $\GL_n(\F_0)$
in terms of $c_\pi$ and $\Vv$.

\begin{theo}
\label{piranesi}
The representation $\pi$ is $\GL_n(\F_0)$-distinguished if and only if 
its central~cha\-racter $c_\pi$ is trivial on $\F_0^\times$ and 
\begin{enumerate}
\item
if $\F/\F_0$ is unramified, 
then $\Vv$ is $\GL_n(\kk_0)$-distinguished,
\item
if $\F/\F_0$ is ramified, 
then $n$ is even,
$\Vv$ is $\GL_{n/2}(\kk)\times\GL_{n/2}(\kk)$-distinguished,
the vector~spa\-ce~of $\GL_{n/2}(\kk)\times\GL_{n/2}(\kk)$-invariant linear forms
on $\Vv$ has dimension $1$ and 
\begin{equation}
\label{pauillac2015}
\ss =
\begin{pmatrix}
0 & {\rm id} \\
{\rm id} & 0 \end{pmatrix}
\in \GL_n(\kk)
\end{equation}
acts on this space by the sign $c_\pi(\w)$,
where $\w$ is any uniformizer of $\F$.
\end{enumerate}
\end{theo}

\begin{proof}
By Proposition \ref{resB}, 
the representation $\pi$ is $\GL_n(\F_0)$-distinguished if and only if
$\bl$~is $\BJ\cap\G^\s$-distinguished.
In the unramified case,
the result follows from the fact that $\BJ\cap\G^\s$~is~ge\-nerated by
$\F_0^\times$ and $\BJ^{0}\cap\G^\s$
(see Lemma \ref{YumAubergineindicei})
and $(\BJ^{0}\cap\G^\s)/(\BJ^1\cap\G^\s)$ identifies with
$\GL_n(\kk_0)$.

Assume now that we are in the ramified case.
Since $c_\pi$ is trivial on $\F_0^\times$ and $1+\p_\F$,
its value at $\w$ does not depend on the choice of a uniformizer
$\w$ of $\F$.
We thus may and will assume that $\s(\w)=-\w$,
thus $c_\pi(\w)\in\{-1,1\}$.
By Lemma \ref{YumAubergineindicei} again,
$\BJ\cap\G^\s$~is~ge\-nerated by
$\F_0^\times$, $\BJ^{0}\cap\G^\s$ and~$\w w$,
where the element $w\in\BJ^0$ is defined by \eqref{defwavants}.
The quotient $(\BJ^{0}\cap\G^\s)/(\BJ^1\cap\G^\s)$~iden\-tifies
with~$\GL_{n/2}(\kk)\times\GL_{n/2}(\kk)$,
the image of $w$ in $\BJ^0/\BJ^1\simeq\GL_{n}(\kk)$ is the element $\ss$,
the space~of
$\GL_{n/2}(\kk)\times\GL_{n/2}(\kk)$-in\-va\-riant linear 
forms on $\Vv$
has dimension $1$ by \cite{VSANT19} Corollary 2.16
and~$\w\ss$ acts on this space as $c_\pi(\w)\ss$.
\end{proof}

Putting Theorems \ref{potimarrong} and \ref{piranesi} together, 
we have thus reduced the problem of characterizing
dis\-tin\-guished cuspidal
representations of $\GL_n(\F)$ to a problem about distinction
of cuspidal~re\-pre\-sentations of
finite general linear groups. 

\subsection{}
\label{grange}

In this paragraph,
we prove Lemma \ref{pinson}(3).
First,
by Remarks \ref{rembase1}, \ref{onchanget}, 
changing~\eqref{choixisoB} and $t$ does not affect 
the isomorphism class of $\btt$.
Let $[\aa,\b']$ be another $\s$-selfdual 
maximal~sim\-ple stratum such that
$\t\in\Cc(\aa,\b')$. 
Conjugating by $\BJ^1$,
we may and will assume that the~maxi\-mal
tamely~ramified~ex\-ten\-sion
of $\F$ in $\E'=\F[\b']$ is $\T$.  
This gives us another isomorphism $\NF'$ from 
$\BJ/\BJ^1$ to $\BJT^{}/\BJT^1$. 
By construction,
it coincides with $\NF$ on $\BJ^0/\BJ^1$
and the image of $x\BJ^1$ by $\NF'$ is equal to
$\Nm_{\E'/\T}(x)\BJ^1$ for all $x\in\E'^\times$.
We are going to prove that $\NF'$ is equal to $\NF$. 
The result will~then follow from the fact that
$\btt$~is equal to $\bt\circ\NF$.
For this,
it suffices to prove that $\NF$ and  $\NF'$
% these maps
take~the same value at some given uniformizer of $\E'$.
Let $\w$, $\w'$ be uniformizers of $\E$, $\E'$ respectively.

The centre of $\BJ/\BJ^1$ is $\E^\times\BJ^1/\BJ^1=\E'^\times\BJ^1$, 
thus $\E'^\times\subseteq\E^\times\BJ^1$.
We thus may write
$\w'\in\w\zeta\BJ^1$ for some root of unity $\zeta$ of $\T^\times$
of order prime to $p$.
Changing $\w'$ to $\w'\zeta^{-1}$,
we may and will assume that $\w'\in\w\BJ^1$.
It suffices to prove the following claim. 

\begin{enonce}{Claim}
\label{lemmebtt} 
We have $\Nm_{\E'/\T}(\w')\equiv\Nm_{\E/\T}(\w) \mod 1+\pp_\T$.
\end{enonce}

First, this is true when $m=1$.
Indeed, 
writing $\G_\T$ for the centralizer of $\T$ in $\G$ and 
$\det_\T$ for the determinant on $\G_\T$,
we have $\det_\T(x)=\Nm_{\E/\T}(x)$ for all $x\in\E^\times$,
thus 
\begin{equation*}
\Nm_{\E'/\T}(\w')=\det_\T(\w')\in\det_\T(\w)\cdot\det_\T(\BJ^1\cap\G_\T)
=\Nm_{\E/\T}(\w) \cdot (1+\pp_\T).
\end{equation*}

Now assume that $m>1$. 
We use the results of \ref{previousparagraph} for $u=m$. 
Let $\t_*\in\Cc(\aa_*,\b)$ denote the trans\-fer of $\t$
as in Lemma \ref{ArseneLup1}. 
Fix a $\T$-embedding
\begin{equation*}
\iota : \E' \to \End_\T(\V_*) \subseteq \End_\F(\V_*) 
\end{equation*}
such that $\aa_*$ is normalized by $\iota\E'^\times$,
and transfer $\t$ to $\t_\bullet\in\Cc(\aa_*,\iota\b')$
in the sense of \cite{BK} 3.6.
The simple character $\t$ is in $\Cc(\aa,\b)\cap\Cc(\aa,\b')$.
It follows from \cite{BHLTL1} Theorem 8.7
that $\t_*$, $\t_\bullet$ intertwine in $\G_*$,
and from \cite{BK} Theorem 3.5.11 that 
$\t_\bullet^{}=\t_*^x$ for some $x\in\Kk_{\aa_*}$.
Changing $\iota$ to ${\rm Ad}(x)\circ\iota$,
we thus may assume that 
$\t_\bullet=\t_*\in\Cc(\aa_*,\b)\cap\Cc(\aa_*,\iota\b')$. 
By using $\iota$, 
we get a diagonal embedding
\begin{equation*}
\E' \to \End_\T(\V_*)\tdt\End_\T(\V_*) 
\subseteq \End_\T(\V)
\end{equation*}
denoted $\phi$,
which is the identity on $\T^\times$.
The Skolem-Noether theorem implies that 
$\phi={\rm Ad}(g)$~for some $g\in\G_\T$.
Conjugating by $g$,
we thus may assume that $\E^\times$ and $\E'^\times$ are both diagonal in 
$\M$.
The identity $\w'\in\w\BJ^1$ thus implies $\w'\in\w\BJ_*^1$.
We are thus reduced to the case where $m=1$.

\begin{rema}
The fact that $\btt$ does not depend on the choice of $\b$ 
is claimed in \cite{BHEffective}~Lem\-ma~3.6.
However, Property (b) of this lemma does not hold:
using the notation of \textit{ibid.}, 
the~res\-tric\-tion of $\l_{\xi}^{\BJ}$ to $\P^\times$ is a multiple
of the character $\xi\circ\Nm_{\P/\T}$,
whereas the res\-tric\-tion of $(\xi|_{\T^\times})^{\BJ}$
to $\P^\times$ is $(\xi\circ\Nm_{\P/\T})^s$.
(Note that $\P$ corresponds to our $\E$,
and $s$ corresponds to our $m$.)
\end{rema}

\section{The odd case}
\label{Vsoddcase}
In this section,
$p$ is odd,
$\ell$~is any pri\-me number different from $p$ and
the field $\R$ has characteris\-tic $\ell$.
This~section is de\-voted to the proof of the following theorem.

\begin{theo}
\label{THModd}
Let $\pi$ be a $\s$-selfdual cuspidal non-supercuspidal
$\R$-representation of $\GL_n(\F)$.
Assume that the in\-teger $r=r(\pi)$ is odd,
thus $\pi$ is isomorphic to $\St_r(\rho)$ for a uniquely determined
$\s$-selfdual 
supercuspidal~re\-presentation $\rho$~of $\GL_{k}(\F)$, with $k=n/r$. 
If $\pi$ is $\GL_n(\F_0)$-distinguished, then 
\begin{enumerate}
\item 
the relative degree $m=m(\pi)$ and the ramification index of $\T/\T_0$
have the same parity,
\item 
the representation $\rho$ is $\GL_k(\F_0)$-dis\-tin\-guished. 
\end{enumerate}
\end{theo}

Note that the fact that $r$ is odd and $r\neq1$ implies that $\ell\neq2$.  

\subsection{}

Before we start the proof of Theorem \ref{THModd},
let us prove the following Disjunction Theorem.

\begin{coro}
Let $\pi$ be a $\s$-selfdual cuspidal $\R$-representation of $\GL_n(\F)$.
Assume that~$r(\pi)$ is odd.
Then $\pi$ cannot be both distinguished and $\x$-distinguished.
\end{coro}

\begin{proof}
Assume that $\pi$ is both distinguished and $\x$-distinguished,
and let $\chi$ be a tamely rami\-fied character of $\F^\times$ extending $\x$. 
Then~$\pi\chi$ is distinguished,
it is isomorphic to $\St_r(\rho\chi)$~and
$\rho\chi$ is su\-per\-cuspidal and $\s$-selfdual.
Theorem \ref{THModd}
applied to both $\pi$ and $\pi\chi$
implies that $\rho$~is~both~dis\-tin\-guished and $\x$-distinguished.
This contradicts Theorem \ref{rappelamical1}.
\end{proof}

We also have the following Distinguished Lift Theorem. 

\begin{coro}
\label{corobitmap}
Let $\pi$ be a $\GL_n(\F_0)$-distinguished cuspidal 
$\flb$-representation of $\GL_n(\F)$ with $r(\pi)$ odd. 
There is a $\GL_n(\F_0)$-distinguished integral generic $\qlb$-representation
of $\GL_n(\F)$ whose reduction mod $\ell$ contains $\pi$.
\end{coro}

\begin{proof}
Write $\pi$ as $\St_r(\rho)$ with $\rho$ distinguished. 
Let $\mu$ be a distinguished integral cuspidal~lift~of $\rho$,
which exists by Theorem \ref{rappelamical2}.
Then the generic representation $\St_r(\mu)$ satisfies the
required~pro\-per\-ty
(see \cite{AnandIMRN08} Theorem 1.3
% \cite{AnandRajan},
or \cite{MatringeIMRN09} Corollary 4.2 when $F$ has characteristic zero, and observe that the argument in \cite{MatringeIMRN09} holds verbatim in positive characteristic thanks to \cite[Theorem 4.7]{Jo23}).
\end{proof}

\begin{rema}
A $\GL_n(\F_0)$-distinguished integral generic $\qlb$-representation
of $\GL_n(\F)$ as in Corollary \ref{corobitmap} may not be cuspidal. 
See Section \ref{seceven} for the classification of all distinguished cuspidal 
$\flb$-representations of $\GL_n(F)$ having a cuspidal distinguished 
lift to $\qlb$.
\end{rema}

Finally, compare Theorem \ref{THModd} with the following finite field analogue. 

\begin{prop}
Let $\ke/\kf$ be a quadratic extension of finite fields~of
characteristic $p$.~Let~$\varrho$ be a supercuspidal $R$-representation of 
$\GL_f(\ke)$ for 
some $f\>1$,~and $r$~be an odd integer such that $\st_r(\varrho)$ is cuspidal.
If $\st_r(\varrho)$ is distinguished by $\GL_{fr}(\kf)$,
then $\varrho$ is distinguished by $\GL_f(\kf)$.
\end{prop}

\begin{proof} 
First, \cite{VSANT19} Remark 4.3 tells us that $\st_r(\varrho)$ is 
$\s$-selfdual (where $\s$~is here the nontrivial automorphism of $\ke/\kf$).
Proposition \ref{Coupurefinie} implies that $\varrho$ is $\s$-selfdual.
By \cite{VSANT19} Lemma 2.5,~it~is~dis\-tinguished by $\GL_f(\kf)$.
\end{proof}

\subsection{}

Let us prove Theorem \ref{THModd}(1).
Since $r$ is odd, $m$ has the same parity as $m/r$, and,
since $\pi$~is non-supercuspidal, we have $r>1$, thus $m>1$.
It follows from \cite{VSANT19} Proposition 7.1
that,~if~$m$ is odd, $\T/\T_0$ is unramified,
and from Proposition \ref{fraise} that,
if $m$ is even, $\T/\T_0$ is~ra\-mi\-fied. 

\subsection{}

We now start the proof of Theorem \ref{THModd}(2).
We thus have a distinguished cuspidal~repre\-sen\-ta\-tion $\pi$ of $\GL_n(\F)$, 
which we write $\St_r(\rho)$ with $\rho$ supercuspidal and $\s$-selfdual. 

Associated with $\pi$,
there are a positive divisor $m$ of $n$,
a quadratic extension $\T/\T_0$ and a~cus\-pi\-dal representation $\pit$ of 
$\GL_m(\T)$.
By Proposition \ref{proppit},
the representation $\pit$ has level $0$,~it~is
distinguished by $\GL_m(\T_0)$ and it satisfies $r(\pit)=r$.

Similarly,
associated with $\rho$,
there is a supercus\-pi\-dal $\s$-selfdual representation $\rhot$ of 
$\GL_{m/r}(\T)$,
which has level $0$,
and is distinguished by $\GL_{m/r}(\T_0)$ if and only if $\rho$ is
distinguished by $\GL_k(\F_0)$.
By Proposition \ref{compapitrho},
the representation $\pit$ is isomorphic to $\St_r(\rhot)$.

It follows that,
in order to prove Theorem \ref{THModd}(2),
we may assume that $\pi$ has level $0$.

\subsection{}

Let $\pi$ be a distinguished cuspidal representation 
of level $0$ of $\GL_n(\F)$.
Associated with~it,
there are its central character $c_\pi$ and
a cuspidal representation $\Vv$ of $\GL_n(\kk)$
(see \S\ref{groseille}).

The representation $\pi$ is isomorphic to $\St_r(\rho)$
for a unique $\s$-selfdual supercuspidal representation $\rho$,
and $\rho$ has level $0$. 
Associated with $\rho$,
there are its central character $c_\rho$ and
a~super\-cuspidal representation $\varrho$~of $\GL_k(\kk)$.
We have the relation
\begin{equation}
\label{urberville}
c^{\phantom{r}}_\pi = \(c_\rho\)^r
\end{equation}
and, by Proposition \ref{step2},
the repre\-sentation $\Vv$ is isomorphic to $\st_r(\varrho)$.

Since $\pi$ is distinguished,
its central character is trivial on $\F_0^\times$. 
Since $\rho$ is $\s$-selfdual,
the restric\-tion of $c_\rho$ to $\F_0^\times$ has order at most $2$.
Restricting the relation \eqref{urberville} to $\F_0^\times$,
and since $r$ is odd,~we
deduce that $c_\rho$ is trivial on $\F_0^\times$. 

\subsection{}

In this paragraph,
we will assume that $\F/\F_0$ is unramified.
By Theorem \ref{piranesi},
the represen\-tation~$\Vv$ is distinguished by $\GL_n(\kk_0)$.
By \cite{VSANT19} Remark 4.3,
it is thus $\s$-selfdual,
that is
\begin{equation*}
\st_r(\varrho) \simeq \Vv \simeq \Vv^{\s\vee} \simeq \st_r(\varrho^{\s\vee}).
\end{equation*}
It follows from Proposition \ref{Coupurefinie} that
$\varrho$ is $\s$-selfdual.
By \cite{VSANT19} Lemma 2.5,
it is thus distinguished by $\GL_k(\kk_0)$.
Applying Theo\-rem \ref{piranesi} again,
we deduce that $\rho$ is distinguished by $\GL_k(\F_0)$.
This proves Theorem \ref{THModd}
in the unramified case. 

\subsection{}
\label{cocorico}

From now on,
and until the end of this section,
we assume that $\F/\F_0$ is ramified.
By~Theo\-rem \ref{piranesi},
we may write $n=2u$ for some integer $u\>1$.
We wri\-te $\G=\G_n=\GL_n(\kk)$,
$\H=\H_n=\GL_{u}(\kk)\times\GL_{u}(\kk)$
and $\K=\K_n$ for the normalizer of $\H$ in $\G$,
which is generated by $\H$ and 
\begin{equation*}
\ss = \ss_n = 
\begin{pmatrix}
0 & {\rm id} \\
{\rm id} & 0 \end{pmatrix}
\in \G
\end{equation*}
where ${\rm id}$ is the identity in $\GL_{u}(\kk)$.
It will be convenient to introduce the following definition. 

\begin{defi}
\label{epsidist}
Let $c\in\{-1,1\}\subseteq\R^\times$.
An irreducible $R$-representation $V$ of $G$ is said to~be
$c$-dis\-tin\-guished by $H$ if $V$ is $H$-distinguished
and $\ss$ acts on the space of $H$-invariant linear forms on $V$
by multiplication by $c$.
\end{defi}

By Theorem \ref{piranesi}, the representation $\Vv$ is $\H$-distinguished
and $\ss$ acts on the $1$-dimensional~vec\-tor~space
$\Hom_\H(\Vv,\R)$ by the sign $c=c_\pi(\w)$.
In other~words, 
$\Vv$~is $c$-distinguished by $H$.
We~are now reduced to proving the following result.
(Note that $k$ is even since $n$ is even and $r$ is odd.)

\begin{prop}
\label{BastardBattle}
The supercuspidal representation $\varrho$ is $c$-distinguished by $H_{k}$.
\end{prop}

Indeed, 
since $r$ is odd,
the identity \eqref{urberville} together with Proposition \ref{BastardBattle}
will give us $c=c_\rho(\w)$.~It
will then follow from Theorem~\ref{piranesi} 
that $\rho$ is $\GL_k(\F_0)$-dis\-tin\-guished.

\subsection{}

Let~$\pi$ be an irreducible~$\flb$-representation of~$G$.
The natural map 
\begin{equation*}
\Hom_{\flb H}(\pi,\flb)\otimes R \to \Hom_{R H}(\pi\otimes R,\flb\otimes R)
\end{equation*}
defined by 
$f\otimes r\mapsto r(f\otimes \mathrm{id})$ is an isomorphism of $R$-vector 
spaces. 
Moreover, these spaces~have dimension at most $1$, 
% are~ei\-ther zero or isomorphic to~$R$,
and it follows from this isomorphism that~$\pi$ is $c$-distinguished by $H$ if
and only if~$\pi\otimes R$ is $c$-distinguished by $H$.

Since $G$ is finite, 
any irreducible $R$-representation of $G$ is defined over $\flb$,
that is,
isomorphic~to $\pi_0\otimes\R$ for some irreducible $\flb$-representation
$\pi_0$ of $G$. 
In order to~prove~Pro\-po\-sition \ref{BastardBattle}, 
we thus may assume that $R$ is equal to $\flb$.  

\subsection{}
\label{antoinette48}

From now on, 
we assume that $R$ is equal to $\flb$.
The remaining part of the section will~be~de\-vo\-ted to the
proof of Proposition \ref{BastardBattle}. 

\begin{lemm}
\label{le1}
There exists a $c$-distinguished irreducible $\qlb$-representation of $\G$
whose reduction mod $\ell$ contains $\Vv$. 
\end{lemm}

\begin{proof}
Let $\chi$ denote the unique $\flb$-character of $\K$ trivial on $\H$ such that 
$\chi(\ss)=c$.
Since~$\Vv$~is $c$-distinguished,
it embeds in $\Ind^G_K(\chi)$.
Equivalently,
the representation $\Ind^G_K(\chi)$,
which is selfdual (as $\chi$ is equal to $\chi^{-1}$),
surjects onto the contragredient $\W$ of $\Vv$. 
Let $\Pi$ be a~projective~indecompo\-sable $\flb$-re\-pre\-sen\-tation of $\G$ 
whose unique irreducible quotient is isomorphic to~$\W$.
Let~$\widetilde{\Pi}$ be~the unique~pro\-jective $\zlb$-representation 
of $\G$~such that $\widetilde{\Pi}\otimes\flb$~is isomorphic to $\Pi$. 
Let~$\La$ be~a~sur\-jec\-ti\-ve homo\-morphism~from
$\Ind^\G_\K(\chi)$ to $\W$. 
By projectivity,
it defines a non-zero~homo\-mor\-phism~$\La'$ from
$\Pi$ to $\Ind^\G_\K(\chi)$,
then a~non-zero~ho\-mo\-morphism $\La''$ from
$\widetilde{\Pi}$ to $\Ind^\G_\K(\widetilde{\chi})$,
where $\widetilde{\chi}$ is the~cano\-ni\-cal $\zlb$-lift of $\chi$.

By inverting~$\ell$,
we deduce that there is an irreducible~$\qlb$-re\-presentation $X$
of $G$ occurring in each of the semi-sim\-ple representations
$J=\Ind^\G_\K(\widetilde{\chi})\otimes\qlb$~and
$\widetilde{\Pi}\otimes\qlb$.
It is thus $c$-distinguished and,
by \cite{Serre} 15.4, 
its reduction mod $\ell$ contains $\W$.

Now observe that,
since $\widetilde{\chi}$ is quadratic,
$J$ is selfdual.
The contragredient of $X$ is thus $c$-dis\-tin\-guished
and its reduction mod $\ell$ contains $\Vv$.
\end{proof}

\subsection{}

Let $\tau$ be a $c$-distinguished
irreducible $\qlb$-representation as in Lemma \ref{le1}.
Consider its~cus\-pi\-dal support:
there are~positi\-ve~integers $n_1,\dots,n_t$ such that
$n_1+\dots+n_t=n$ and,
for each~$i$ in~$\{1,\dots,t\}$,
a cuspidal~ir\-re\-ducible $\qlb$-representation $\rho_i$ of $\GL_{n_i}(\kk)$,
such that $\tau$ occurs as a~com\-ponent of the paraboli\-cally 
induced representation $\rho_1\times\dots\times\rho_t$,
denoted $\W$.
The representation $\W$ is thus $c$-distinguished. 
We claim the following.

\begin{enonce}{Claim}
\label{MattKellner}
There is an $i\in\{1,\dots,t\}$ such that $n_i$ is even and
$\rho_i$ is $c$-distinguished by $H_{n_i}$.
\end{enonce}

Before proving this claim in the next paragraph,
let us explain how it implies Proposition \ref{BastardBattle}. 

Propositions \ref{Coupurefinie} and \ref{classifJames} imply that,
for each $i\in\{1,\dots,t\}$,~the
reduction mod $\ell$ of $\rho_i$ is~ir\-re\-ducible and cuspidal, 
of the form $\st_{r_i}(\varrho_i)$ for a unique~posi\-tive integer $r_i$
and a unique supercuspi\-dal representation $\varrho_i$.
Since the reduction mod $\ell$ of $\tau$ contains $\Vv$, 
the representation $\Vv$ occurs~as an irreducible component of
the parabolically induced representation
$\rl(\rho_1)\times\dots\times\rl(\rho_t)$.
Uni\-queness of the supercuspidal support implies that
$\varrho_i\simeq\varrho$ for all $i$.
It follows that either $r_i=1$~or $r_i=e(\varrho)\ell^{v_i}$ for some $v_i\>0$.
Observe that,
as $r=e(\varrho)\ell^{v}$ for some $v\>0$ and $r$ is
odd,~the~in\-te\-ger $e(\varrho)$ is odd, thus $r_i$ is odd in any case,
for all $i$. 

Fix an integer $i$ as in Claim \ref{MattKellner},
and let $\xi_i$ be a parameter for $\rho_i$
in the sense of Definition~\ref{defparameter}. 
It is a $\Gal(\kk_{n_i}/\kk)$-regular $\zlb$-character of $\kk_{n_i}^\times$.
By~Propo\-si\-tion \ref{chourineur}, 
it is trivial on $\kk_{u_i}^\times$,
where $u_i$ is defined by $n_i=2u_i$,
and it takes the unique element of
$\kk_{n_i}^\times/\kk_{u_i}^\times$~of~or\-der~$2$ to $-c$.

Since the reduction mod $\ell$ of $\rho_i$ is $\st_{r_i}(\varrho)$,
the reduction mod $\ell$ of $\xi_i$ takes the form
$\vartheta\circ\Nm_{\kk_{n_i}/\kk_{k}}$
where $\vartheta$ is a~para\-meter for $\varrho$.

Since $n_i=r_ik$ and $r_i$ is odd,
$\vartheta$ is trivial on $\kk_{l}^\times$ (where $k=2l$)
and takes the element of
$\kk_{k}^\times/\kk_{l}^\times$ of order $2$ to $-c$.

By
Propo\-si\-tion \ref{chourineur}, 
the canonical $\zlb$-lift of $\vartheta$
is the parameter of a $c$-distinguished $\qlb$-lift of $\varrho$,
which implies that $\varrho$ is $c$-distin\-guished
(see Remark \ref{comtesseartoff}).
This proves Proposition \ref{BastardBattle}.  

\subsection{}

The remaining part of this section will~be~de\-vo\-ted to the
proof of Claim \ref{MattKellner}.
We follow~the argument of \cite{NadirCRELLE15} Section 3,
which simplifies in our situation since we deal with finite groups.
Let $\Aa$ denote the set of $t$-uples $\a = (\a_1,\dots,\a_t)$ where
\begin{enumerate}
\item 
for each $i$, the element $\a_i$ is a family of $t+1$ non-negative integers
of the form
\begin{equation*}
\a_i=(n^{\phantom{+}}_{i,1},\dots,n^{\phantom{+}}_{i,i-1},
n_{i,i}^+,n_{i,i}^-,n^{\phantom{+}}_{i,i+1},\dots,n^{\phantom{+}}_{i,t})
\end{equation*}
of sum $n_i$,
\item
one has $n_{1,1}^++\dots+n_{t,t}^+=n_{1,1}^-\dots+n_{t,t}^-$
and $n_{i,j}=n_{j,i}$ for all $i\neq j$.
\end{enumerate}
For an $\a\in\Aa$,
it will be convenient to set $n^{\phantom{+}}_{i,i}=n_{i,i}^++n_{i,i}^-$
for each integer $i\in\{1,\dots,t\}$.

As in \cite{NadirCRELLE15} 3.1,
the set $\Aa$ parametrizes the set of $(\P,\H)$-double cosets in $\G$,
where $\P$ in the~para\-bo\-lic subgroup of $\G$ generated by upper triangular
matrices and the standard Levi subgroup ~$\M$ isomorphic to 
$\G_{n_1}\times\dots\times\G_{n_t}$.
Let us explain how this parametrization works.
Associated with any $\a\in\Aa$, there are 
\begin{itemize} 
\item
a standard Levi subgroup
\begin{equation*}
\M_\a = (G_{n_{1,1}} \times G_{n_{1,2}} \times\dots\times G_{n_{1,t}}) 
\times\dots\times
(G_{n_{t,1}} \times G_{n_{t,2}} \times\dots\times G_{n_{t,t}}) 
\subseteq \M,
\end{equation*} 
\item
a diagonal element
\begin{equation*}
d_\a = \diag\left(
\begin{pmatrix} {\rm id}_{n_{1,1}^+} & \\ & -{\rm id}_{n_{1,1}^-}\end{pmatrix}
, {\rm id}_{n_{1,2}},\dots,{\rm id}_{n_{1,t}},\dots,
{\rm id}_{n_{t,1}}, {\rm id}_{n_{t,2}}, \dots,
\begin{pmatrix} {\rm id}_{n_{t,t}^+} & \\ & -{\rm id}_{n_{t,t}^-}\end{pmatrix}
\right) \in M_\a,
\end{equation*} 
\item
a permutation matrix $w_\a\in G$ defi\-ned~as~fol\-lows:
decompose $\{1,\dots,n\}$ as the disjoint union of in\-tervals
$J_{i,j}= \{a_{i,j},a_{i,j}+1,\dots,b_{i,j}\}$ of length $n_{i,j}$,
for each $i,j\in\{1,\dots,t\}$, 
where $a_{1,1}=1$, 
$a_{i,j+1}=b_{i,j}+1$ if $j\neq t$
and $a_{i+1,1}=b_{i,t}+1$ if $i\neq t$; 
then $w_\a$ is the involution which
\end{itemize}
\begin{itemize}
\item[$\bullet$]
restricts to the identity on $J_{i,i}$ for each $i$,
\item[$\bullet$]
exchanges the ~in\-ter\-vals $J_{i,j}$ and $J_{j,i}$ if $i\neq j$,
and sends the $k$th element of any of these~inter\-vals
to the $k$th ele\-ment of the other one, for all $k\in\{1,\dots,n_{i,j}\}$.
\end{itemize}
A system of representatives $(x_{\a})_{\a\in \Aa}$ of $(\P,\H)$-double cosets 
in $\G$ is then obtained by any~choi\-ce of $x_{\a}\in G$ such that
\begin{equation}
\label{condua}
x^{\phantom{1}}_{\a}
\begin{pmatrix}
{\rm id}_u & \\ & -{\rm id}_u
\end{pmatrix}
x_{\a}^{-1} = e_\a, 
\end{equation}
where $e_\a = d_\a w_\a$. 

\begin{defi}
An $\a\in\Aa$ is called \textit{admissible} if,
for any $i$,
there exists a unique $j$~such~that $n_{i,j}\neq0$.
This defines an involution $\s_\a:i\mapsto j$ on $\{1,\dots,t\}$.
\end{defi}

When this is the case,
let us write $H_\a$ for the subgroup of $\M$ made of the
${\rm diag}(g_1,\dots,g_t)\in\M$ such that 
\begin{equation*}
\begin{array}{rcll}
g_{\s_\a(i)} & = & g_i & \text{ for all $i\in\{1,\dots,t\}$}, \\ 
g_i &\in & \GL_{n_{i,i}^+}(\kk) \times \GL_{n_{i,i}^-}(\kk)
& \text{ for all $i$ fixed by $\s_\a$}.
\end{array}
\end{equation*} 
Moreover,
if $n_{i,i}^+=n_{i,i}^-$ for all $i\in\{1,\dots,t\}$,
we define a matrix $k_\a={\rm diag}(k_1,\dots,k_t)\in\M$ by
\begin{equation*} 
k_{i}={\rm id}_{n_i}=-k_{\s_\a(i)} \text{ if } i<\s_\a(i),
\quad
k_{i}=\ss_{n_i} \text{ if } i=\s_\a(i).
\end{equation*}
This matrix normalizes $H_\a$,
and we write $\K_\a$ for the group generated by $H_\a$ and $k_\a$.

We denote by $\t_\a$ the inner automorphism of the group ${\rm PGL}_n(\kk)$
induced by conjugacy by $e_\a$ 
(which normalizes $\M_\a$).
It is not~hard to check that:

\begin{lemm}
Let $Z$ denote the centre of $G$. 
\begin{enumerate}
\item 
An $\a\in\Aa$ is admissible if and only if 
$\M/Z$ is~$\t_\a$-sta\-ble in $G/Z={\rm PGL}_n(\kk)$. 
\item
Suppose that $\a\in\Aa$ is admissible. 
The preimage of $(\M/Z)^{\t_\a}$ in $\G$, denoted by $L_\a$, is
\begin{equation*}
\left\{ 
\begin{array}{ll}
\K_\a & \text{if $n_{ii}^+=n_{ii}^-$ for all $i$}, \\ 
H_\a & \text{otherwise}.
\end{array}\right.
\end{equation*} 
\end{enumerate}
\end{lemm}

When $L_\a=K_\a$, we denote by $\chi_\a$ the character of $K_\a$
trivial on $H_\a$ and sending $k_\a$ to~$c$.~Other\-wise,
we set $\chi_\a$ to be the trivial character of $L_\a=H_\a$. 
We have the following lemma. 

\begin{lemm}
\label{le20}
Suppose that $\a$ is admissible and $L_\a=\K_\a$.
Then there is 
a system of~re\-pre\-sen\-tatives $(x_{\a})_{\a\in \Aa}$ of $(\P,\H)$-double cosets 
of $G$
satis\-fying both \eqref{condua} and
$x^{\phantom{1}}_{\a} \ss^{\phantom{1}}_n x_{\a}^{-1}=k_\a$. 
\end{lemm}

\begin{proof}
Let us set $m_i=n_{i,i}^+=n_{i,i}^-=n_{i,i}/2$ for any integer $i\in\{1,\dots,t\}$
such that $\s_\a(i)=i$.~For
each $\a\in\Aa$, we look for a matrix $x_\a\in G$ such that
\begin{equation*}
x^{\phantom{1}}_{\a}
\begin{pmatrix} {\rm id}_u & \\ & -{\rm id}_u \end{pmatrix}
x_{\a}^{-1} = e_\a
\quad \text{and} \quad 
x^{\phantom{1}}_{\a}
\begin{pmatrix} & {\rm id}_u \\ {\rm id}_u & \end{pmatrix}
x_{\a}^{-1} = k_\a.
\end{equation*}
To make an explicit~choice of $x_\a \in G$,
it will be convenient to introduce the matrix $v_\a\in G$
defined as follows:
for all integers $i,j\in\{1,\dots,t\}$,
the $(i,j)$-block of $v_\a$ in $\Mat_{n_i,n_j}(\kk)$ is
\begin{itemize}
\item 
the identity matrix ${\rm id}_{n_i}$ if $j=i$ or $j=\s_\a(i)<i$,
\item
its opposite $-{\rm id}_{n_i}$ if $j=\s_\a(i)>i$, 
\item
and $0$ otherwise.
\end{itemize} 
Then we choose $y_\a$ the permutation
matrix corresponding to the permutation of minimal length (with the usual
generators of the symmetric group) satisfying
\begin{equation*}
y_\a
\begin{pmatrix}
{\rm id}_u & \\ & -{\rm id}_u
\end{pmatrix}
y_{\a}^{-1}=l_\a 
\end{equation*}
where 
$l_\a={\rm diag}(l_1,\dots,l_t)\in\M$ is defined by
\begin{equation*} 
l_{i}={\rm id}_{n_i}=-l_{\s_\a(i)} \text{ if } i<\s_\a(i),
\quad
l_{i}=
\begin{pmatrix}
{\rm id}_{m_i} & \\ & -{\rm id}_{m_i}
\end{pmatrix}
% \diag({\rm id}_{m_i},-{\rm id}_{m_i})
\text{ if } i=\s_\a(i).
\end{equation*}
Finally we put  $x_\a=v_\a y_\a$, which has the
desired property thanks to the equality 
\begin{equation}\label{eq conj}
\begin{pmatrix} {\rm id}_k & -{\rm id}_k \\ {\rm id}_k & \phantom{-} {\rm id}_k \end{pmatrix} 
\begin{pmatrix}
{\rm id}_k & \\
& -{\rm id}_k 
\end{pmatrix}
\begin{pmatrix} {\rm id}_k & -{\rm id}_k \\ {\rm id}_k & \phantom{-} 
{\rm id}_k \end{pmatrix}^{-1}=\ss_{2k}\end{equation} 
valid for any $k\geq 1$. With this choice, the careful reader checks by a
computation relying again~on Equality (\ref{eq conj}),
that
$y^{\phantom{1}}_\a \ss^{\phantom{1}}_n y_\a^{-1}
= v_{\a}^{-1} k^{\phantom{1}}_\a v^{\phantom{1}}_\a$,
which is the desired equality. 
\end{proof}

Now we have the following lemma. 

\begin{lemm}
\label{le2}
There is an admissible $\a\in\Aa$ such that
$\Hom_{L_\a}(\rho_1\otimes\dots\otimes\rho_t,\chi_\a)$ is non-zero.
\end{lemm}

\begin{proof}
Given any subgroup $X$ of $G$,
we will write $\overline{X}$ for its image in
$G/Z={\rm PGL}_n(\kk)$.~In~par\-ticular, we have $\overline{G}=G/Z$.
Note that $\overline{\K}=\K/Z$ is the subgroup of $\overline{G}$ made
of all elements fixed by conjugacy by 
\begin{equation*} 
\begin{pmatrix}{\rm id}_u&0\\ 0 & -{\rm id}_u\end{pmatrix} \text{ mod } Z.
\end{equation*}
Let $\chi$ be the unique character of $\K$ trivial on $\H$ such that 
$\chi(\ss)=c$.
The character~that~it~indu\-ces on $\overline{\K}$ 
will still be denoted by $\chi$. 
Since $\W$ is $c$-distinguished,
Mackey's formula implies that there is an $x\in G$ 
such that $\rho$,
the representation of $P$
inflated from $\rho_1\otimes\dots\otimes\rho_t$,
is distinguished by~the character 
$\chi^{x}|_{{\P}\cap{\K}^{x}}$.
We derive from $\rho$ a representation $\overline{\rho}$ of $\overline{P}$
distinguished by
$\chi^{x}|_{\overline{\P}\cap\overline{\K}{}^{x}}$.

In fact, because $H$ is a subgroup of $K$,
we can chose $x$ to be some 
${x_{\a}}$ for $\a \in \Aa$.
Now we claim that for all non admissible $\a \in \Aa$, the space
\begin{equation}
\label{timoleon}
\Hom_{\overline{\P}\cap\overline{\K}{}^{{x_{\a}}}}
(\overline{\rho},{\chi}{}^{x_\a}) 
\end{equation}
is zero, 
so in particular $x$ can only be of the form $x_\a$ for admissible $\a$.
Indeed, it follows~from~\cite{NadirCRELLE15} Proposition 3.5 that,
for a non admissible $\a$, the group $\P\cap H^{x_{\a}}$ contains a non
trivial unipotent radical $U_\a$ of some parabolic subgroup of $\M$, but 
the character $\chi_\a$ is trivial on $U_\a$, so if
the space \eqref{timoleon} 
were not reduced to zero, we would deduce that $\Hom_{U_\a}({\rho},R)$ is
non-zero, contradicting~the cuspidality of $\rho$.
Hence we deduce $x=x_\a$ for an $\a$ which is admissible.
In this case,
$\overline{\M}\cap\overline{\K}{}^{{x_{\a}}}$~is equal to
$\overline{\M}{}^{{\t_\a}}$, 
so that the space
\begin{equation*}
\Hom_{\overline{\M}{}^{{\t_\a}}}(\overline{\rho},{\chi}^{{x_{\a}}})
=\Hom_{L_\a}(\rho,{\chi}^{{x_{\a}}})
\end{equation*}
is non-zero.
If $L_\a=K_\a$,
then ${\chi}^{x_{\a}}$ is equal to $\chi_\a$ thanks to Lemma \ref{le20}.
Otherwise,
${\chi}^{x_{\a}}$ and $\chi_\a$ are trivial, thus equal.
The statement now follows. 
\end{proof}

Recall that,
for any $i\in\{1,\dots,t\}$,
either $r_i=1$ or $r_i=e(\varrho)\ell^{v_i}$ for some $v_i\>0$.

\begin{lemm}
\label{le1000}
Let $\a\in\Aa$ be as in Lemma \ref{le2}.
Then the involution $\s_\a$ has a fixed point. 
\end{lemm}

\begin{proof} 
Let $\I_1$ be the set of $i\in\{1,\dots,t\}$ such that $r_i>1$,
let $t_1$ be the cardinality of this set,
and define $t_0=t-t_1$.
The identity $r = r_1 + \dots + r_t$ implies
\begin{equation*}
r = t_0 + e(\varrho)\cdot\sum\limits_{i\in\I_1} \ell^{v_i}.
\end{equation*}
Since $r$, $e(\varrho)$ and $\ell$ are odd,
it follows that $t_0+t_1=t$ is odd.
Thus $\s_\a$ has a fixed point. 
\end{proof}

Claim \ref{MattKellner} now follows from Lemmas \ref{le2}, \ref{le1000}.  
Indeed,
by Lemma \ref{le2}, 
there is~an~ad\-mis\-sible $\a\in\Aa$ such that
$\Hom_{L_\a}(\rho_1\otimes\dots\otimes\rho_t,\chi_\a)$ is non-zero.
Since $L_\a$ contains $H_\a$,
the re\-presentation $\rho_i$ is, for all $i$ fixed by $\s_\a$,
distinguished~by the Levi subgroup
\begin{equation*}
\GL_{n_{i,i}^+}(\kk)\times \GL_{n_{i,i}^-}(\kk).
\end{equation*}
By \cite{VSANT19} Propo\-sition 2.14,
this implies that $n_{i,i}^+=n_{i,i}^-$
for all $i$ fixed by $\s_\a$, thus $L_\a=K_\a$.
By Lemma  \ref{le1000},
there in an in\-te\-ger $i\in\{1,\dots,t\}$ fixed by $\s_\a$.
The $i$th block of $L_\a=K_\a$ is $K_{n_i}$
and $\chi_\a(k_i)=c$.
Thus $\rho_i$ is $c$-distinguished.

\section{Distinguished lift theorems}
\label{seceven}

In this section,
$p$ is odd
and $\ell$~is a pri\-me number different from $p$. 
We look for a necessary~and sufficient condition
for an $\flb$-cuspidal representation of $\GL_n(\F)$ to have
a $\GL_n(\F_0)$-dis\-tingui\-shed lift to $\qlb$. 
Since the case of supercuspidal representations is treated by
Theorems \ref{KuMaTh} and \ref{rappelamical3},~we
will concentrate on non-supercuspidal cuspidal representations.

\subsection{}

We will prove the following two propositions.

\begin{prop}
\label{poivronsimpairs} 
Let $\pi$ be a $\s$-selfdual cuspidal 
$\flb$-representation of $\GL_n(\F)$ with~quadratic
ex\-tension $\T/\T_0$ and $m=m(\pi)$.
Assume that $r=r(\pi)>1$ is odd. 
Then $\pi$ has~a dis\-tin\-gui\-shed~lift to $\qlb$ if and only if
\begin{enumerate}
\item
the representation $\pi$ is~iso\-mor\-phic to $\St_r(\rho)$ for some 
$\GL_{n/r}(\F_0)$-distingui\-shed~su\-per\-cuspi\-dal representation~$\rho$
of $\GL_{n/r}(\F)$,
\item
if $e$,~$e_0$ are the orders of the car\-di\-na\-lities of the residue
fields $\ll$, $\ll_0$ of $\T$, $\T_0$ mod $\ell$, then
\begin{enumerate}
\item 
either $\T/\T_0$~is~un\-ra\-mi\-fied and $e_0$ is even,
\item
or $\T/\T_0$~is~ra\-mi\-fied and $m/e$ is odd.
\end{enumerate}
\end{enumerate}
\end{prop}

Note that the assumption ``$r>1$ is odd'' in Proposition \ref{poivronsimpairs}
implies that $\ell\neq2$. 

\begin{prop}
\label{poivronspairs} 
Let $\pi$ be a $\s$-selfdual cuspidal 
$\flb$-representation of $\GL_n(\F)$ with quadratic exten\-sion $\T/\T_0$
and $m=m(\pi)$.
Assume that $r=r(\pi)$ is even. 
Then $\pi$ has a~dis\-tin\-gui\-shed lift~to $\qlb$ if and only if
\begin{enumerate}
\item
the extension $\T/\T_0$~is~ra\-mi\-fied,
\item
one has $m=r$,
\item
if we denote by $\nu_0$ the normalized absolute value of $F_0$, then 
\begin{enumerate}
\item
either $\ell\neq2$ and $\pi$ is isomorphic~to $\St_m(\rho)$ for some 
supercus\-pidal representation $\rho$~of $\GL_{n/m}(\F)$
which is either $\x$-distinguished or~{$\nf^{-1}$}-distinguished, 
\item 
or $\ell=m=r=2$, the car\-di\-na\-lity of the residue field of
{$\T_0$} is congruent to $-1$~mod~$4$
and $\pi$ is~iso\-morphic~to $\St_2(\rho)$ where $\rho$ is a
$\GL_{n/2}(\F_0)$-distinguished supercuspidal~represen\-ta\-tion of
$\GL_{n/2}(\F)$.
\end{enumerate}
\end{enumerate}
\end{prop}

We also formulate the following conjecture making
Proposition \ref{poivronspairs} more precise. 

\begin{conj}
If $\pi$ is a $\s$-selfdual cuspidal $\flb$-representation of $\GL_n(\F)$
such that the~in\-te\-ger $r(\pi)$ is~even,
the following assertions are equivalent:
\begin{enumerate}
\item 
the representation $\pi$ is ~dis\-tin\-gui\-shed,
\item 
the representation $\pi$ has a~dis\-tin\-gui\-shed lift to $\qlb$,
\item
the three conditions of Proposition \ref{poivronspairs} hold.
\end{enumerate}
\end{conj}

By Proposition \ref{poivronspairs}, Theorem \ref{KuMaTh}, 
we know that (2) implies (1) and is equivalent to (3). 
We~thus conjecturate that (1) implies (3).
See \cite{CLL} Theorem 4.6 for the case $n=r=2$.

\subsection{}
\label{carottedeniveaunonnul}

Let $\pi$ be a $\s$-selfdual cuspidal $\flb$-representation of $\G=\GL_n(\F)$.
Let $(\BJ,\bl)$ be~a~ge\-ne\-ric~$\s$-self\-dual type in $\pi$,~let
$\blw$ be the~re\-pre\-sentation of $\BJ$ given by Proposition \ref{potiron}
(see Paragraph~\ref{lwcan})
and $\bt$ be the representation of~$\BJ$ trivial~on $\BJ^1$ such that~$\bl$
is~iso\-mor\-phic to $\blw\otimes\bt$.
Associated with $\pi$ by \eqref{defpit},
there is also a $\s$-selfdual cuspidal $\flb$-representation $\pit$
of $\GL_m(\T)$.

\begin{lemm}
\label{tortue}
The following assertions are equivalent.
\begin{enumerate}
\item 
The representation $\pi$ has a $\GL_n(\F_0)$-dis\-tin\-guished lift to $\qlb$.
\item
The representation $\bl$ has a $\BJ\cap\GL_n(\F_0)$-dis\-tin\-guished lift to
$\qlb$.
\item
The representation $\bt$ has a $\BJ\cap\GL_n(\F_0)$-dis\-tin\-guished lift
to $\qlb$.
\item
The representation $\pit$ has a $\GL_m(\T_0)$-dis\-tin\-guished lift to $\qlb$. 
\end{enumerate}
\end{lemm}

\begin{proof}
Fix a $\s$-selfdual simple stratum $[\aa,\b]$~as well
as~iso\-morphisms \eqref{choixisoB} as in Proposition \ref{pechessd}.
Let $\t\in\Cc(\aa,\b)$ be the $\s$-selfdual maximal simple character
associated with $\bl$,
and $\widetilde{\t}$ be its unique $\qlb$-lift:
this is a $\s$-self\-dual maximal simple character
(with respect to the unique $\qlb$-lift~$\widetilde{\psi}$ of the
character $\psi$ given by \eqref{psiF}) having the same
$\G$-normalizer~$\BJ$ as $\t$.

Let $\widetilde{\bl}_{{\rm w}}$ be the $\qlb$-representation 
of $\BJ$ associated with $\widetilde{\t}$ 
by Proposition \ref{potiron}. 
It is $\BJ\cap\G^\s$-dis\-tinguished and $\s$-selfdual,
and its~de\-termi\-nant has order a power of $p$.
It is thus integral.
Let~us consider~its reduction mod $\ell$.
On the one hand, it is $\BJ\cap\G^\s$-distinguished, $\s$-selfdual,
and its~de\-termi\-nant has order a power of $p$. 
On the other hand,
\cite{MSt} Proposition 2.37 implies~that~it~is~an irreducible representation
extending the Heisenberg representation associated with $\t$.
By~unique\-ness, 
we~de\-duce that $\widetilde{\bl}_{{\rm w}}$ is a $\qlb$-lift of
$\blw$.

Suppose that $\pi$ has a $\G^\s$-distinguished $\qlb$-lift $\widetilde{\pi}$.
Thus $\widetilde{\pi}$ is a $\s$-selfdual and~cus\-pi\-dal
representation of $\G$ containing the maximal simple character $\widetilde{\t}$.
By Proposition \ref{pulledpork6heures},
this representation $\widetilde{\pi}$
contains a distinguished generic $\s$-selfdual type,
which we may assume to be~of the form $(\BJ,\widetilde{\bl})$
with $\widetilde{\bl}=\widetilde{\bl}_{{\rm w}}\otimes\widetilde{\bt}$
and the representation $\widetilde{\bt}$ is
$\BJ\cap\G^\s$-distinguished.~Reducing mod $\ell$,
we deduce that $\pi$ con\-tains the type $\blw\otimes\bd$
where $\bd$ is the reduction mod $\ell$ of
$\widetilde{\bt}$.
But $\pi$ also contains the type $\blw\otimes\bt$,
thus $\bd$ is~iso\-morphic to $\bt$,
and the reduction mod $\ell$ of $\widetilde{\bl}$
is~iso\-morphic to $\bl$. 
Thus (1) implies both (2) and (3). 

Conversely,
suppose that $\bt$ has a distinguished $\qlb$-lift $\widetilde{\bt}$.
Then the pair
$(\BJ,\widetilde{\bl}_{{\rm w}}\otimes\widetilde{\bt})$
is a~dis\-tin\-guished type whose compact induction 
to $\G$ is a $\G^\s$-distinguished $\qlb$-lift of $\pi$,
and whose~re\-duction mod $\ell$ is iso\-morphic to
$\blw\otimes\bt\simeq\bl$. 
Thus (3) implies both (1) and (2).

Applying these results to the representation $\pit$, 
we get that $\pit$ has a~dis\-tin\-guished lift to $\qlb$~if
and only if $\bt_{\rm t}$ has a~dis\-tin\-guished lift to $\qlb$.
The fact that $\bt$ is isomorphic to $\bt_{\rm t}\circ\NF$
(by Lemma \ref{defNnlemma})
thus implies that (4) is equivalent to (3).
\end{proof}

It follows from Lemma \ref{tortue},
together with Corollary \ref{potimarroncor} and Proposition \ref{compapitrho}, 
that,
in order to prove Propositions \ref{poivronsimpairs} and 
\ref{poivronspairs},
it suffices to prove them for $\s$-selfdual cuspidal $\flb$-representations
of level $0$.
(For Proposition \ref{poivronspairs}(3.a),
it also follows from the fact that
$\x_{F/F_0}\circ {\rm N}_{T_0/F_0}=\x_{T/T_0}$
and $\nu_{F_0}\circ {\rm N}_{T_0/F_0}=\nu_{T_0}$.)

\subsection{}
\label{carottedeniveaunul}

We continue with the situation of Paragraph \ref{carottedeniveaunonnul},
assuming further that $\pi$ has level $0$.
Thus
$\pi$ is a $\s$-selfdual cuspidal $\flb$-representation of $\G$
of level $0$.
We will also assume that $\pi$ is non-su\-percuspidal,
that is, $r=r(\pi)>1$.
Let $(\BJ,\bl)$ be~a~ge\-ne\-ric~$\s$-self\-dual type in $\pi$.
Associated~with it in Paragraph \ref{groseille}, there are
\begin{itemize}
\item 
a $\s$-self\-dual tamely~rami\-fied character $\ic$ of $\F^\times$, 
which is the central character $c_\pi$ of $\pi$,
\item
and a $\s$-selfdual cus\-pidal representation
$\Vv$~of~$\GL_n(\kk)$ of the form $\st_r(\varrho)$
for some supercuspi\-dal representation $\varrho$ of
$\GL_{n/r}(\kk)$,
uniquely determined up to isomorphism
(thus $\Vv$ is non-super\-cus\-pidal).
\end{itemize}
Recall that the restriction of $\bl$ to $\BJ^0$ is the inflation of $\Vv$,
and that its restriction to $\F^\times$ is~a~mul\-tiple of $\ic$. 
Since
$\Vv$ is~$\s$-self\-dual,
Proposition \ref{Coupurefinie} implies that $\varrho$ is
$\s$-selfdual.

The action of $\s$ on $\GL_n(\kk)$ is described in Proposition \ref{pechessd}:
this is the action of the~non-trivial automorphism of $\kk/\kk_0$
if $F/F_0$ is unramified,
and the adjoint action of \eqref{banane}
with $i=\lfloor m/2 \rfloor$~other\-wise.

Let us fix a uniformizer $\w$ of $F$ such that
$\w\in\F_0$ if $F/F_0$ is unramified, 
and $\w^2\in\F_0$ if $F/F_0$~is ramified.
(One thus has $\s(\w)=-\w$ in the ramified case.)

\begin{lemm}
\label{piranesilift}
The representation $\pi$ has a $\GL_n(\F_0)$-dis\-tin\-guished lift
to $\qlb$ if and only if $\Vv$~has a $\GL_n(\kk)^\s$-distinguished lift
$\widetilde{\Vv}$ to $\qlb$ such that
\begin{enumerate}
\item 
if $F/F_0$ is unramified, then $\ic(\w)=1$,
\item
if $F/F_0$ is~ra\-mi\-fied,
then $n$ is even and \eqref{pauillac2015}
acts on the~spa\-ce of $\GL_{n/2}(\kk)\times\GL_{n/2}(\kk)$-in- va\-riant
linear forms~on $\widetilde{\Vv}$ by a sign whose reduction mod $\ell$ 
is equal to $\ic(\w)$.
\end{enumerate} 
\end{lemm}

\begin{proof} 
By Lemma \ref{tortue}, 
the representation $\pi$ has a $\GL_n(\F_0)$-dis\-tin\-guished lift 
if and~only if~the type~$\bl$ has a $\BJ\cap\GL_n(F_0)$-dis\-tin\-guished
lift to $\qlb$.  
Suppose $\bl$ has a dis\-tin\-guished lift $\widetilde{\bl}$. 
Then the pair
$(\BJ,\widetilde{\bl})$ is~the generic type
of~a~distin\-guished cuspidal $\qlb$-representation $\widetilde{\pi}$,
compactly~in\-duced from $\widetilde{\bl}$.
Associa\-ted with it, there are
\begin{itemize}
\item 
a cus\-pidal $\qlb$-re\-pre\-sen\-ta\-tion~$\widetilde{\Vv}$ 
of~$\GL_n(\kk)$ lifting $\Vv$,
\item 
a tamely~rami\-fied $\qlb$-character
$\widetilde{\ic}$ of $\F^\times$ lifting $\ic$.
\end{itemize}
By Theorem~\ref{piranesi}, 
the character $\widetilde{\ic}$ is trivial on $F_0^\times$
and $\widetilde{\Vv}$ is distinguished by $\GL_n(\kk)^\s$.
If $F/F_0$ is unramified, then $\widetilde{\ic}(\w)=1$, thus $\ic(\w)=1$. 
If $F/F_0$ is~ra\-mi\-fied,
then $n=2u$ for some $u\>1$ and $\widetilde{\Vv}$ is 
$\widetilde{\ic}(\w)$-distinguished
(in the sense of Definition \ref{epsidist}),
and the reduction mod $\ell$ of $\widetilde{\ic}(\w)$ is equal to $\ic(\w)$.

Conversely, 
suppose that $\Vv$ has a $\GL_n(\kk)^\s$-distinguished lift $\widetilde{\Vv}$
satisfying the conditions of~the lemma.
Let $\widetilde{\ic}$ be a $\qlb$-lift of $\ic$ coinciding on the units of $\F$ with 
the inflation of the central~cha\-rac\-ter of $\widetilde{\Vv}$, 
and
\begin{enumerate}
\item 
if $F/F_0$ is unramified, then $\widetilde{\ic}(\w)=1$,
\item
if $F/F_0$ is~ra\-mi\-fied,
then $\widetilde{\ic}(\w)\in\{-1,1\}$
and the representation $\widetilde{\Vv}$ is $\widetilde{\ic}(\w)$-distinguished.
\end{enumerate} 
Inflate $\widetilde{\Vv}$ to $\BJ^0$,
and extend~it to a representation $\widetilde{\bl}$ of $\BJ$
by demanding that the restriction of $\widetilde{\bl}$ to $F^\times$
is a multiple of $\widetilde{\ic}$. 
The representation
$\widetilde{\bl}$ is then a $\BJ\cap\GL_n(F_0)$-dis\-tin\-guished lift of 
$\bl$.  
\end{proof}

\subsection{}
\label{par113}

In this paragraph, we assume that $\F/\F_0$ is~unra\-mi\-fied.

\begin{lemm}
\label{poulain}
Let $\Vc$ be a $\s$-selfdual cuspidal $\flb$-repre\-senta\-tion
of $\GL_n(\kk)$. 
It has a $\GL_n(\kk_0)$-distinguished lift to $\qlb$ if and only 
if $n$ is odd and
\begin{enumerate}
\item 
either $\Vc$ is supercuspidal, 
\item
or $\Vc$ is non-supercuspidal and the order of the cardinality of $\kk_0$
mod $\ell$ is even (thus $\ell\neq2$).
\end{enumerate}
\end{lemm}

\begin{proof}
By \cite{Gow} Theorem 3.6,
an irreducible $\qlb$-representation of $\GL_n(\kk)$ is 
$\GL_n(\kk_0)$-dis\-tin\-gui\-shed if and only if it is $\s$-selfdual. 

First, the condition on the parity of $n$ is necessary:
see \cite{VSANT19} Lemma 2.3 for instance. 
Now~as\-su\-me that $n$ is odd. 
If $\ell\neq2$, the result is given by \cite{KuMaAsai} Proposition 4.6. 
If $\ell=2$,
then $\Vc$ has the form $\st_r(\varrho)$,
where $\varrho$ is a supercuspidal representation of $\GL_{n/r}(\kk)$
and $r=2^v$ for some $v\>0$.
Since $n$ is odd, $\Vc$ must be supercuspidal,
and the result is given by \cite{VSANT19} Remark 2.7. 
\end{proof}

\begin{rema}
Let~$q$ be~the cardinality of $\kk$ and $q_0$ be that of $\kk_0$.
Let $e$ and $e_0$ be the orders of $q$ and $q_0$ mod $\ell$,
respectively. 
Note that $r=e(\varrho)\ell^{v}$ for some $v\>0$,
where~$e(\varrho)$ is the order of $q^{f}$ mod $\ell$ with $f=n/r$.
If $n$ is odd,
then $f$ and $r$ are odd,
thus $e(\varrho)$ is odd.
But $e(\varrho)=e/(e,f)$.
It follows that $e=e_0/(e_0,2)$ is odd.
Thus $e_0$ is not divisible by $4$.
\end{rema}

\begin{exem}
Let $\Vc$ be the $\s$-selfdual cuspidal
$\flb$-repre\-senta\-tion $\st_e(1)$ of $\GL_e(\kk)$.
We have $e=e_0/(e_0,2)$,
which is odd if and only if $e_0$ is not divisible by $4$.
Thus $\Vc$ has a $\GL_e(\kk_0)$-dis\-tin\-guished lift to
$\qlb$ if and only if $e_0$ is divisible by $2$ but not by $4$.
\end{exem}

Suppose first that $\pi$ has a dis\-tin\-guished lift to $\qlb$.
On the one hand, 
the generic type of such a lift defines a $\s$-selfdual cuspidal
$\qlb$-representation of $\GL_n(\kk)$,
and \cite{VSANT19} Lemma 2.3 implies that $n$ is odd,
thus $r$ is odd. 
On the other hand, Theorem \ref{KuMaTh} implies that $\pi$ is
distinguished.~It~thus follows from 
Theorem \ref{THModd} that $\pi$ is isomorphic to $\St_r(\rho)$ for
some dis\-tinguished~su\-per\-cuspi\-dal~re\-presentation $\rho$ of
$\GL_{n/r}(\F)$. 
Finally, 
Lemma \ref{piranesilift} says that
$\Vv$ has a dis\-tin\-guished lift.
It follows from Lemma \ref{poulain} that the order $e_0$
of~the~car\-dinality of $\kk_0$
mod $\ell$ is even. 

We thus proved that,
when $\F/\F_0$ is unramified,
if $\pi$ has a dis\-tin\-guished lift,
then $r$ is odd and Conditions (1), (2.a) of Proposition \ref{poivronsimpairs}
are satisfied. 

Conversely,
suppose that the conditions (1), (2.a) of Proposition \ref{poivronsimpairs}
are satisfied.
Then~$\Vv$~has~a dis\-tinguished lift $\widetilde{\Vv}$.
By Lemma \ref{piranesilift}, the representation
$\pi$ has a $\GL_n(F_0)$-distinguished lift to $\qlb$ if and only 
if~$\ic(\w)=1$.  
By Paragraph \ref{compatypes},
the central character $\ic_*$ of $\rho$ satisfies $\ic_*^r=\ic$.
Since $\rho$ is distin\-guished,
we have $\ic_*(\w)=1$,
thus $\ic(\w)=\ic_*(\w)^r=1$.

We pro\-ved Proposition \ref{poivronsimpairs}
in the case when $\F/\F_0$ is unramified. 

\subsection{}
\label{poupinpar}

In this paragraph, we assume that $\F/\F_0$ is ra\-mi\-fied.
Let~$q$ denote~the cardinality of $\kk$,
and let $e$ denote the order of $q$ mod $\ell$.

\begin{lemm}
\label{poupin}
Let $\Vc$ be a selfdual cuspidal $\flb$-representation~of $\GL_n(\kk)$,
isomorphic to $\st_r(\varrho)$ for some selfdual supercuspidal representation 
of $\GL_{n/r}(\kk)$.
Write $u=\lfloor n/2\rfloor$.
Then $\Vc$ has~a~lift to $\qlb$ which is distinguished by 
$\GL_{u}(\kk)\times\GL_{n-u}(\kk)$ if and only if
\begin{enumerate}
\item 
either $\Vc$ is supercuspidal, 
\item
or $\Vc$ is non-supercuspidal, $n$ is even and
\begin{enumerate}
\item
either $\ell\neq2$ and $r$, $n/e$ are odd, 
\item
or $\ell\neq2$ and $r=n$, 
\item
or $\ell=n=r=2$ and $q$ is congruent to $-1$ mod $4$,
and $\varrho$ is trivial.  
\end{enumerate}
\end{enumerate}
\end{lemm}

\begin{proof}
First note that $n$ must be either even or equal to $1$:
see \cite{VSANT19} Lemma 2.17 for instance.
Also, the supercuspidal case is given by \cite{VSANT19} Remark 2.21. 
Let us assume that $\Vc$ is non-super\-cuspidal
(thus $n$ is even, and we will write $n=2u$).
We use the notation of Paragraph \ref{scaramuccia}.

Set $f=n/r$
and let~$\a$~be a $\Gal(\kk_{f}/\kk)$-regular
$\flb$-character of $\kk_{f}^\times$ of order $\A$ which
is a~para\-me\-ter of $\varrho$ in the sense of Definition \ref{defparameter}.
Let $\widetilde{\v}$ be~the canonical $\qlb$-lift of 
\begin{equation*}
\v=\a\circ\Nm_{\kk_n/\kk_f},
\end{equation*}
that is,
its unique lift of order~$\A$.
Let $\widetilde{\Vc}$ be~a~cuspi\-dal~lift of $\Vc$.
It~is parametrized~by~a~$\Gal(\kk_n/\kk)$-re\-gu\-lar
character of $\kk_n^\times$ lifting $\v$,
that is, of the form $\widetilde{\v}\phi$,
where $\phi$ is a $\qlb$-character of $\kk_n^\times$
of order $\ell^s$ for some $s\>0$. 
Since $\Vc$ is not supercuspidal,
one has $s\>1$.
The character $\widetilde{\v}\phi$ has order $\A\ell^s$.

By Proposition \ref{chourineur}, % \cite{HaMu} Proposition 6.1, 
the representation $\widetilde{\Vc}$ is distinguished
by $\GL_{u}(\kk)\times\GL_{u}(\kk)$
if and~on\-ly~if it is selfdual,
which is also equivalent 
(see for instance \cite{VSANT19} (2.7))
to $\A\ell^s$ dividing $q^{u}+1$.
Similarly,
the fact that $\varrho$ is selfdual is equivalent to
\begin{itemize}
\item 
either $f=1$ and $\varrho$ is a quadratic character
(thus $\A$ is equal to $1$ or $2$), 
\item
or $f$ is even and $\A$ divides $q^{f/2}+1$
(thus $\A>2$ since $q$ has order $f\>2$ mod $\A$).
\end{itemize} 

Suppose that $\ell\neq2$ and $f$ is even.
If $\widetilde{\Vc}$ is distinguished,
then $\A$ divides $q^{f/2}+1$ and $q^{u}+1$.
Since $u=rf/2$, we have
\begin{equation*}
q^{u} + 1 = 1 + (-1)^{r} +
\sum\limits_{i=1}^{r} \binom {r} {i} (-1)^{r-i} (q^{f/2}+1)^{i}
\end{equation*}
thus $\A$ divides $1+(-1)^r$.
Since $\A>2$, 
it follows that $r$ is odd.
Also, $\ell$ divides $q^u+1$,
that is, the order of $q^u$ mod $\ell$ is
$e/(e,u)=2$, which implies that $n/e$ is odd. 
Conversely,~sup\-po\-se that $r$ and $n/e$ are odd.
The fact that $\A$ divides $q^{f/2}+1$ and $r$ is odd
implies that $\A$ divides $q^u+1$.
Now $\ell^s$ divides $q^n-1=(q^u+1)(q^u-1)$.
If $\ell$ divides $q^u-1$,
then $e$ divides $u=n/2$, thus $n/e$ is even: contradiction.
Thus $\ell^s$ divides $q^u+1$,
thus $\widetilde{\Vc}$ is distinguished.

Suppose that $\ell\neq2$ and $f=1$.
Then $\varrho$ is a character of $\kk^\times$,
thus $r=e\ell^{v}$ for some $v\>0$.~This
gives $n/e=\ell^{v}$, which~is odd.
The same argument as above implies that $q^u+1$
is a multiple of $\ell^s$.
It is also a multiple of $\A\in\{1,2\}$
since it is even.
Thus $\widetilde{\Vc}$~is~dis\-tinguished.

Now suppose that $\ell=2$.
If $\widetilde{\Vc}$ is distinguished and $f$ is even,
then,
as in the case where~$\ell\neq2$,
the integer $\A>2$ divides $q^{f/2}+1$ and $q^{rf/2}+1$,
thus $r$ is odd. 
But the fact that $\Vc$ is cuspidal implies that $r$ is a power of $2$.
It follows that $r=1$: contradiction.
Thus $f=1$, that is $\Vc$ is the representation
$\st_n(1)$ with $n=2^t$ for some $t\>1$. 
Moreover, $q$ has order $m$ mod $2^s$,
that is, $2^s$ divides $q^{n}-1$ but not $q^{u}-1$.
Set
\begin{equation*}
a=\v_2(q^u+1), \quad b=\v_2(q^u-1).
\end{equation*}
We have $b<s\<a+b$ and $\min(a,b)=1$.
The fact that $\widetilde{\Vc}$ is distinguished implies $s\<a$,
which gives $b=1<a$,
that is $4$ divides $q^u+1$.
Since $u$ is a power of $2$,
we deduce that $4$ divides $q+1$ and $u=1$.

Conversely,
suppose that $\ell=n=r=2$ and $4$ divides $q+1$
(hence $b=1<a$).
Then~any $\overline{\mathbb{Q}}_{2}$-character of $\kk_{2}^\times$ of order $2^a$
parametrizes a distinguished cuspidal 
$\overline{\mathbb{Q}}_{2}$-representation of
$\GL_{2}(\kk)$~lif\-ting $\Vc=\st_2(1)$.
\end{proof}

\begin{exem}
The fact that $\GL_f(\kk)$ has a selfdual supercuspidal 
$\flb$-repre\-senta\-tion is equi\-valent to the fact that there is
an $\kk$-regular $\flb$-character of $\kk_{f}^\times$ which is trivial on
$\kk_{f/2}^\times$, that~is,
there exists an integer $\A$ with the following properties:
\begin{enumerate}
\item $\A$ is prime to $\ell$ and the order of $q$ mod $\A$ is equal to $f$, 
\item $\A$ divides $q^{f/2}+1$.
\end{enumerate}
Now suppose that $\ell>2$ and $f=2$. 
Thus $\GL_2(\kk)$ has a selfdual super\-cuspidal $\flb$-repre\-sentation
if~and only if there exists an integer $\A$ prime to $\ell$
dividing $q+1$ but not $q-1$, that is,~if~and
only~if~$q+1$ has a prime divisor different from $2$ and $\ell$.
Assume this is the case,
and let $\varrho$ be a selfdual supercuspidal
$\flb$-repre\-sen\-ta\-tion of $\GL_2(\kk)$.
Let $\Vc$ be the selfdual cuspidal
$\flb$-repre\-senta\-tion $\st_r(\varrho)$ of $\GL_n(\kk)$ with $r=e/(e,2)$ 
and $m=2r$.
Then $r$ is odd if and only if $e$ is not divisible by $4$,
and $n/e=2/(e,2)$ is odd if and only if $e$ is even. 
If we take $q=9$ and $\ell=7$,
we get $r=3$ and $n/e=2$.
If we take $q=5$,
we get $q-1=4$ and $q^6-1=1953\times8$.
Thus if $\ell$ is a prime divisor of $1953$,
we get $r=3$ and $n/e=6$.
\end{exem}

In addition, we have the following result.
We assume that $n=2u$ for some $u\>1$.

\begin{lemm}
\label{Milon}
Let $\W$ be a selfdual cuspidal $\flb$-representation of $\GL_n(\kk)$
of the form $\st_n(\varrho)$~for some quadratic character $\varrho$
of $\kk^\times$.
Assume $\W$ is distinguished by $\GL_{u}(\kk)\times\GL_{u}(\kk)$.
Then \eqref{pauillac2015}
acts on the space of $\GL_{u}(\kk)\times\GL_{u}(\kk)$-invariant linear 
forms on $\W$ by 
\begin{equation*}
\left\{ 
\begin{array}{ll}
-1 & \text{if $\varrho$ is trivial}, \\ 
% \x(-1)^{u}
(-1)^{u(q-1)/2} & \text{if $\varrho$ is non-trivial}.
\end{array}\right.
\end{equation*}
\end{lemm}

\begin{proof}
Let $c$ be the sign such that $\W$ is $c$-distinguished.
If $\ell=2$, the result is immediate since the only sign is $1$.
Assume that $\ell\neq2$. 
By~Lemma \ref{poupin},
the representation $\W$ has a distinguished cuspidal $\qlb$-lift.
Let $\widetilde{\W}$ be such a $\qlb$-lift
and~$\xi$~be~a parameter for $\widetilde{\W}$.
Let~$\a$ be an element of $\kk_{n}$ such that $\a\notin\kk_{u}$
and $\a^2\in\kk_{u}$.
By Proposition \ref{chourineur}, the representation
$\widetilde{\W}$~is $-\xi(\a)$-dis\-tingui\-shed by
$\GL_{u}(\kk)\times\GL_{u}(\kk)$. 
Since $\widetilde{\W}$ lifts $\W$,
we have
\begin{itemize}
\item
the reduction mod $\ell$ of the parameter $\xi$ is equal to
$(\varrho\circ\Nm_{\kk_n/\kk})\phi$ where $\phi$ is a character
whose order is a power of $\ell$ (see Proposition \ref{classifJames}),
\item
the reduction mod $\ell$ of $-\xi(\a)$ is equal to $c$
(see Remark \ref{comtesseartoff}). 
\end{itemize}
On the one hand,
the character $\xi$ is trivial on $\kk_{u}^\times$
since $\widetilde{\W}$ is selfdual (see Proposition \ref{chourineur}).
On the other hand, $\varrho\circ\Nm_{\kk_n/\kk}$ 
is trivial on $\kk_{u}^\times$ since $\varrho$ is quadratic and
the index of $\kk_{u}^\times$ in $\kk_{n}^\times$ is even.
We deduce that $\phi$ is trivial on~$\kk_{u}^\times$, 
thus $\phi(\a)$ is a sign. 
Since it has order a power of $\ell\neq2$, it is trivial.
It follows that 
\begin{equation*}
c = -\varrho(\Nm_{\kk_n/\kk}(\a)).
\end{equation*}
If $\varrho$ is trivial, this gives $c=-1$, as expected. 
Assume now that $\varrho$ is non-trivial.
It thus coincides with $\x$ on $\kk^\times$. 
Since $\a^2$ is not~a square in $\kk_{u}^\times$,
its $\kk_u/\kk$-norm is not~a square in $\kk^\times$.
Thus 
\begin{equation*}
c = -\x(\Nm_{\kk_{u}/\kk}(\a^{q^u+1}))
= -(-1)^{(q^u+1)/2}
% = (-1)^{(q^u-1)/2}
\end{equation*} 
and one verifies that this is equal to $\x(-1)^{u}=(-1)^{u(q-1)/2}$ as expected. 
\end{proof}

\subsection{}

Let us prove Proposition \ref{poivronsimpairs}
when $\F/\F_0$ is ramified.
Assume that $r$ is odd,
and suppose that $\pi$ has a dis\-tin\-guished lift to $\qlb$. 
By Theorem \ref{KuMaTh}, it is distinguished.~Thus
Theorem \ref{THModd} implies that $n/r$ is even and $\pi$ is
isomorphic to $\St_r(\rho)$ for
some dis\-tinguished~su\-per\-cuspi\-dal~representation $\rho$ of
$\GL_{n/r}(\F)$.
Lemma \ref{piranesilift} says that $\Vv$ has~a dis\-tin\-guished lift.
It follows from~Lem\-ma \ref{poupin} that $n/e$ is odd.

Conversely, 
assume that $r$ is odd, 
$\pi$ is isomorphic to $\St_r(\rho)$ for some dis\-tinguished
supercuspi\-dal~re\-pre\-sen\-tation $\rho$ of $\GL_{n/r}(\F)$
of level $0$,
$n$ is even and $n/e$ is odd. 
It follows from~Lemma~\ref{poupin} that 
$\Vv$ has a dis\-tin\-guished lift $\widetilde{\Vv}$.
Let $\varepsilon\in\{-1,1\}$ be the unique sign such that $\widetilde{\Vv}$
is $\varepsilon$-distingui\-shed by $\GL_{u}(\kk)\times\GL_{u}(\kk)$
in the sense of Definition \ref{epsidist},
with $n=2u$.
By Lemma \ref{piranesilift},~the~re\-pre\-sen\-ta\-tion
$\pi$ has a distinguished lift if and only if
$\ic(\w)$ is equal to the image of $\varepsilon$ 
in $\overline{\mathbb{F}}{}_\ell^\times$, denoted $c$. 
We are going to prove that this is the case.
Let $\ic_*$ be the central character of $\rho$.
By Theo\-rem \ref{piranesi}, we have
\begin{itemize}
\item 
the representation $\varrho$ is $\ic_*(\w)$-distinguished by 
$\GL_{k/2}(\kk)\times\GL_{k/2}(\kk)$. 
\end{itemize}
(Note $k$ is even since $n$ is even and $r$ is odd.)
By Proposition \ref{step2}, we have
\begin{itemize}
\item 
the sign $\ic(\w)$ is equal to $\ic_*(\w)^r=\ic_*(\w)$.
\end{itemize}
Let $\a$ be the unique sign such that $\Vv$
is $\a$-distinguished by $\GL_{u}(\kk)\times\GL_{u}(\kk)$.
By Remark \ref{comtesseartoff}, 
we have $\a=c$.
On the~other hand, we have $\a=\ic_*(\w)$
by~Propo\-sition \ref{BastardBattle}.
Putting~these~facts together,
we get $\ic(\w)=\ic_*(\w)=\a=c$ as expected. 
This proves~Propo\-sition \ref{poivronsimpairs}
if~$\F/\F_0$~is~ra\-mified.
Together with Paragraph \ref{par113},
this finishes the proof of~Pro\-po\-sition \ref{poivronsimpairs}.

\subsection{}

In this paragraph and the next one, 
we prove Proposition \ref{poivronspairs}.
Assume that $r$ is even,
and let $q$ be the~cardi\-na\-lity~of $\kk$.
Since $r$ divides $n$,
we have $n=2u$ for some $u\>1$.

Suppose that $\pi$ has a distin\-gui\-shed lift. 
By Paragraph \ref{par113},
this implies that $F/F_0$ is ramified. 
By Lemma \ref{piranesilift},
the representation $\Vv$ has a dis\-tin\-guished $\qlb$-lift.
By Lemma \ref{poupin},
one has $r=n$,
thus $\Vv$ is isomorphic to $\st_n(\varrho)$ for a character $\varrho$ of
$\kk^\times$ of order at most $2$.
Besides,
if $\ell=2$,~then
$n=2$ and $q$ is congruent to $-1$ mod $4$. 
Since $r=n$,
the representation $\pi$ is isomorphic to $\St_n(\rho)$ for a
tamely ramified character $\rho$ of $\F^\times$ whose restriction to the
units of $F$ is the inflation of $\varrho$.

Suppose first that $\ell=2$. 
Since $\pi$ is distinguished (by Theorem \ref{KuMaTh}),
it is $\s$-selfdual (by~Theo\-rem \ref{VSANTTHM41}).
It follows from Proposition \ref{fission} that
the representation $\rho$ is $\s$-selfdual
and from~Theo\-rem~\ref{rappelamical1} that it is $F_0^\times$-distinguished,
as expected.

Suppose now that $\ell\neq2$. 
By Proposition \ref{fission},
we may choose $\rho$ so that $\rho^{-1}\circ\s=\rho\nu^i$
for some $i\in\{0,1\}$, that is, $\rho\circ\Nm_{F/F_0}=\nu^{-i}$.~It
remains to prove that the restriction of $\rho$ to $F_0^\times$ is
either $\x$ or $\nu_0^{-1}$.

Let $c$ be the sign by which the element \eqref{pauillac2015}
acts on the space of $\GL_{u}(\kk)\times\GL_{u}(\kk)$-invariant linear 
forms on $\Vv$, which is given by Lemma \ref{Milon}. 
Remind that we have fixed a uniformizer~$\w$~of $F$ such that $\s(\w)=-\w$, 
thus $\w_0=\w^2$ is a uniformizer of $F_0$.
The representation $\pi$ is dis\-tin\-guished (by Theorem \ref{KuMaTh})
and it follows from Theorem \ref{piranesi} that $c=c_\pi(\w)$.
We have
\begin{equation}
\label{moonfleet}
c_\pi(\w) = \rho(\w)^n = \rho(\w_0)^{u}. 
\end{equation}
On the other hand,
the identity $\rho\circ\Nm_{F/F_0}=\nu^{-i}$ implies that $\rho(-\w_0)=q^i$.

\begin{lemm}
\label{patatedouce}
We have $q^{u}\equiv-1$ mod $\ell$.
\end{lemm}

\begin{proof} 
Since $r\>2$ and $\pi$ is cuspidal,~$r$
has the form $\et(\rho)\ell^v$ for some $v\>0$,
where $\et(\rho)$ is the order of $q^k$ mod $\ell$ by \eqref{omegat}.
In particular, $(q^k)^r=q^n$ is congruent to $1$ mod $\ell$.
Moreover, since~$\ell$ is odd,
one has $q^{u}\equiv-1\neq1$ mod $\ell$.
\end{proof}

It follows from Lemma \ref{patatedouce} and \eqref{moonfleet} that 
\begin{equation*}
c =
\left\{ 
\begin{array}{ll}
(-1)^i & \text{if $\varrho$ is trivial}, \\ 
(-1)^i \cdot \x(-1)^{u} & \text{otherwise (that is, if $\varrho=\x$)}.
\end{array}\right.
\end{equation*}
Comparing with Lemma \ref{Milon},
we get the following corollary.

\begin{coro}
We have $i=1$ if $\varrho$ is trivial,
and $i=0$ if $\varrho$ is non-trivial.
\end{coro}

If $i=0$,
then $\rho$ is selfdual.
By Theorem \ref{rappelamical1},
it is either distinguished or $\x$-distinguished.
Since its restriction to the units of $F$ is the inflation of $\varrho=\x$,
we deduce that $\rho$ is $\x$-distinguished.

If $i=1$,
then $\rho\nu^{1/2}$ is unramified and selfdual. 
By Theorem \ref{rappelamical1},
it is distinguished.
Thus the restriction of $\rho$ to $\F_0^\times$ is equal to 
$\nu^{-1/2}|_{\F_0^\times}=\nu_0^{-1}$.

\subsection{}

Let us finish the proof of Proposition \ref{poivronspairs}. 
Assume that $n=r=2u$ for some $u\>1$,~the
ex\-ten\-sion $F/F_0$ is ramified and $\pi$ is isomorphic to $\St_n(\rho)$ 
for some tamely ramified character $\rho$~of $\F^\times$.
We also assume that
\begin{itemize}
\item 
either $\ell\neq2$ and the restriction of $\rho$ to $\F_0^\times$
is either $\x$ or $\nu_0^{-1}$,
\item
or $\ell=n=r=2$,
$q$ is congruent to $-1$ mod $4$ and $\rho$ is trivial on $F_0^\times$.
\end{itemize}
It follows from~Lem\-ma \ref{poupin} that 
$\Vv$ has a dis\-tin\-guished $\qlb$-lift $\widetilde{\Vv}$,
which is $\varepsilon$-distinguished for some sign $\varepsilon\in\{-1,1\}$.
By Lemma \ref{piranesilift}, the representation
$\pi$ has a~dis\-tin\-guished lift to $\qlb$ if and only~if
the reduction of $\varepsilon$ mod $\ell$, denoted $c$,
is equal to $\ic(\w)$.
Let us prove that this is the case. 
On the one hand,
we have $\ic(\w)=\rho(\w)^n=\rho(\w_0)^{u}$.
If $\ell=2$, we have $\ic(\w)=1$.
Otherwise, we have 
\begin{equation*}
\ic(\w) = \left\{ 
\begin{array}{ll}
q^{u} & \text{if the restriction of $\rho$ to $\F_0^\times$ is $\nu_0^{-1}$}, \\ 
\x(-1)^{u} & \text{if the restriction of $\rho$ to $\F_0^\times$ is $\x$}.
\end{array}\right.
\end{equation*} 
On the other hand,
$\Vv$ is distinguished,
and it is isomorphic to $\st_n(\varrho)$ where $\varrho$ is the character of 
$\kk^\times$ defined by the restriction of $\rho$ to the units of $F$.
One thus may apply Lemma \ref{Milon},
which says that $\Vv$ is $\a$-distinguished, with 
\begin{equation*}
\a = \left\{ 
\begin{array}{ll}
-1 & \text{if $\varrho$ is trivial}, \\ 
(-1)^{u(q-1)/2} & \text{if $\varrho$ is non-trivial}.
\end{array}\right.
\end{equation*}
Note that $\a=c$ by Remark \ref{comtesseartoff},
and that $\varrho$ is trivial if and only if the restriction of $\rho$ to 
$\F_0^\times$ is equal to $\nu_0^{-1}$. 
Together with Lemma \ref{patatedouce}, this gives $\ic(\w)=c$ as expected. 

\bibliographystyle{plain}
\bibliography{nrv.bib}

\end{document}